\documentclass[a4paper]{amsart}
\pdfoutput=1
\usepackage[pdftex,
            pdfauthor={Alexis Marchand},
            pdftitle={Bavard duality for the relative Gromov seminorm},
            colorlinks=true]{hyperref}
\usepackage[T1]{fontenc}
\usepackage[english]{babel}
\usepackage{lmodern}
\usepackage[dvipsnames]{xcolor}
\usepackage{graphicx}
\usepackage{enumitem}
\usepackage{amsmath,amsthm,amssymb,mathrsfs,eqnarray,stmaryrd,eqnarray,relsize,mathtools}
\usepackage[msc-links]{amsrefs}
\usepackage{caption,subcaption}

\usepackage{tikz}
\usepackage{tikz-3dplot}
\usetikzlibrary{tqft}
\usetikzlibrary{shapes,arrows,decorations.markings}
\usetikzlibrary[patterns,patterns.meta]
   
\theoremstyle{plain}
\newtheorem{prop}{Proposition}
\numberwithin{prop}{section}
\newtheorem{thm}[prop]{Theorem}
\newtheorem{cor}[prop]{Corollary}
\newtheorem{lmm}[prop]{Lemma}

\newtheorem{innermthm}{Theorem}
\newenvironment{mthm}[1]
  {\renewcommand\theinnermthm{#1}\innermthm}
  {\endinnermthm}

\newtheorem{innermprop}{Proposition}
\newenvironment{mprop}[1]
  {\renewcommand\theinnermprop{#1}\innermprop}
  {\endinnermprop}

\theoremstyle{definition}
\newtheorem{rmk}[prop]{Remark}
\newtheorem{ex}[prop]{Example}
\newtheorem{defi}[prop]{Definition}

\newtheorem*{claim}{Claim}

\newcommand{\normm}[1]{\ensuremath{\left\|#1\right\|}}
\newcommand{\gromnorm}[1]{\ensuremath{\normm{#1}_{1}}}
\newcommand{\inftynorm}[1]{\ensuremath{\normm{#1}_{\infty}}}
\newcommand{\st}{\ensuremath{\:\middle\vert\:}}

\renewcommand{\emptyset}{\varnothing}

\DeclareMathOperator{\scl}{scl}

\DeclareMathOperator{\Imm}{Im}
\DeclareMathOperator{\Ker}{Ker}

\DeclareMathOperator{\area}{area}

\DeclareMathOperator{\Hom}{Hom}
\DeclareMathOperator{\Homeo}{Homeo}
\DeclareMathOperator{\eu}{eu}
\DeclareMathOperator{\Or}{Or}
\DeclareMathOperator{\hmtp}{conj}
\DeclareMathOperator{\Supp}{Supp}

\title{Bavard duality for the relative Gromov seminorm}
\author{Alexis Marchand}
\date{\today}
\address{DPMMS, Centre for Mathematical Sciences, Wilberforce Road, Cambridge CB3 0WB, United Kingdom}
\email{\href{mailto:aptm3@cam.ac.uk}{aptm3@cam.ac.uk}}

\begin{document}

\begin{abstract}
    The relative Gromov seminorm is a finer invariant than stable commutator length where a relative homology class is fixed.
    We show a duality result between bounded cohomology and the relative Gromov seminorm, analogously to Bavard duality for scl.
    We give an application to computations of scl in graphs of groups.
    We also explain how our duality result can be given a purely algebraic interpretation via a relative version of the Hopf formula.
    Moreover, we show that this leads to a natural generalisation of a result of Calegari on a connection between scl and the rotation quasimorphism.
\end{abstract}

\maketitle

\section{Introduction}
    
    Stable commutator length, or $\scl$, is an invariant of groups that can be thought of as a kind of homological $\ell^1$-norm on the commutator subgroup.
    It has attracted attention for its connections with various topics in geometric topology and group theory --- see Calegari's book \cite{cal-scl} for a comprehensive survey.
    However, $\scl$ has proved very hard to compute: Calegari \cite{cal-sclrat} showed that $\scl$ is computable and has rational values in free groups, and Chen \cite{chen} generalised this to certain graphs of groups, encompassing previous results of various authors \cites{walker,cal-sails,chen-freeprods,cfl,susse}, but neither computability nor rationality of $\scl$ are known for closed surface groups.
    
    In \cite{m:isom-embed}, the author approaches the problem of understanding $\scl$ in surface groups by examining whether or not certain embeddings of surfaces are isometric for $\scl$.
    A conclusion of that paper is that some of the results that one can prove for $\scl$ in free groups can only be generalised to closed surface groups if one works in a fixed relative homology class.
    More precisely, the author generalises a result about $\scl$ in free groups to one about the \emph{relative Gromov seminorm}.
    
    The relative Gromov seminorm can be defined as the $\ell^1$-seminorm on $H_2\left(X,\gamma\right)$, where $X$ is a topological space, and $\gamma:\coprod S^1\rightarrow X$ is a collection of loops in $X$ --- see \S{}\ref{sec:gromov} for complete definitions.
    This was first introduced in \cite{m:isom-embed} but the idea is implicit in the work of Calegari (see for instance \cite{cal-fnb}*{Remark 3.18}).
    The relative Gromov seminorm is connected to $\scl$ via the following:
    
    \begin{mprop}{\ref{prop:scl-gromnorm}}[Gromov seminorm and $\scl$]
        Let $X$ be a path-connected topological space and let $[c]\in C_1^{\hmtp}\left(\pi_1X;\mathbb{Z}\right)$ be an integral conjugacy class represented by a map $\gamma:\coprod S^1\rightarrow X$. Then
        \[
            \scl_{\pi_1X}\left([c]\right)=\frac{1}{4}\inf\left\{\gromnorm{\alpha}\st{}\alpha\in H_2\left(X,\gamma;\mathbb{Q}\right),\:\partial\alpha=\left[\textstyle\coprod S^1\right]\right\},
        \]
        where $\partial:H_2(X,\gamma;\mathbb{Q})\rightarrow H_1\left(\coprod S^1;\mathbb{Q}\right)$ is the boundary map in the long exact sequence of the pair $\left(X,\gamma\right)$ (see Proposition \ref{prop:les}).
    \end{mprop}
    
    The metastrategy here is that one might be able to obtain information about stable commutator length by first understanding the relative Gromov seminorm, and then infimising over $H_2\left(X,\gamma\right)$.
    
    A pioneering result in the study of stable commutator length was the discovery of \emph{Bavard duality} \cite{bavard}, showing that the dual space of the $\scl$-seminorm can be understood in terms of \emph{quasimorphisms} --- see \cite{cal-scl}*{\S{}2.5} for more details.
    Bavard duality has led to a vast array of work on $\scl$, most notably yielding various spectral gap results \cites{calegari-fujiwara,bbf,heuer-raags,fft,cfl}, and it is natural to ask for an analogue in the context of the relative Gromov seminorm.
    Combining several well-known results, we show that \emph{bounded cohomology} provides such an analogue:
    
    \begin{mthm}{\ref{thm:bav-dual-rel-grom}}[Bavard duality for the relative Gromov seminorm]
        Let $X$ be a countable CW-complex and $\gamma:\coprod S^1\rightarrow X$.
        Given a real class $\alpha\in H_2(X,\gamma;\mathbb{R})$, the relative Gromov seminorm of $\alpha$ is given by
        \[
            \gromnorm{\alpha}=\sup\left\{\frac{\left<u,\alpha\right>}{\inftynorm{u}}\st{}u\in H^2_b\left(X;\mathbb{R}\right)\smallsetminus\{0\}\right\}.
        \]
    \end{mthm}
    
    The purpose of the present paper is to apply Theorem \ref{thm:bav-dual-rel-grom} in three different directions.
    
    The first one is in the context of graphs of groups, where we are able to use Theorem \ref{thm:bav-dual-rel-grom} and an isometric embedding theorem in bounded cohomology of Bucher et al.\ \cite{bbfplus} to show that vertex groups are isometrically embedded for the relative Gromov seminorm:
    
    \begin{mthm}{\ref{thm:isom-grph-gps}}[$\ell^1$-isometric embedding of vertex groups in graphs of groups]
        Let $\mathcal{G}$ be a graph of groups whose underlying graph $\Gamma$ is finite, with countable vertex groups $\left\{G_v\right\}_{v\in V(\Gamma)}$, and amenable edge groups $\left\{G_e\right\}_{e\in E(\Gamma)}$.
        Fix a vertex $v$ and consider the inclusion map $i_v:G_v\hookrightarrow\pi_1\mathcal{G}$.
        Then for each class $[c]\in C_1^{\hmtp}\left(G_v;\mathbb{Z}\right)$, the embedding
        \[
            {i_v}_*:H_2\left(G_v,[c];\mathbb{R}\right)\hookrightarrow H_2\left(\pi_1\mathcal{G},\left[i_v(c)\right];\mathbb{R}\right).
        \]
        is isometric for $\gromnorm{\cdot}$.
    \end{mthm}
    
    We show in \S{}\ref{sec:hnn} that the analogous statement for $\scl$ does not hold, but that Theorem \ref{thm:isom-grph-gps} can still yield computations of stable commutator length in certain HNN-extensions.
    For example, we compute the spectral gap of $\scl$ in Dyck's surface group and deduce from the Duncan--Howie Theorem \cite{duncan-howie} that it is not residually free.
    
    The second application is a purely algebraic interpretation of Theorem \ref{thm:bav-dual-rel-grom} via the \emph{Hopf formula}.
    We prove in Theorem \ref{thm:rel-hopf} that there is a relative version of the Hopf formula: if $G=F/R$ is a quotient of a free group $F$, and an infinite-order conjugacy class $[w]$ in $G$ is represented by $\bar{w}\in F$, then there is an isomorphism
    \[
        H_2\left(G,[w];\mathbb{Z}\right)\cong\left<\bar{w}\right>R\cap[F,F]/[F,R].
    \]
    We obtain a new interpretation of Theorem \ref{thm:bav-dual-rel-grom} from this point of view:
    
    \begin{mthm}{\ref{thm:bavard-hopf}}[Bavard duality via the Hopf formula]
        Consider an integral class $\alpha\in H_2\left(G,[w];\mathbb{Z}\right)$ represented by a product of commutators
        \[
            \left[\bar{a}_1,\bar{b}_1\right]\cdots\left[\bar{a}_k,\bar{b}_k\right]\in\left<\bar{w}\right>R\cap[F,F],
        \]
        and let $a_i$ and $b_i$ be the respective images of $\bar{a}_i$ and $\bar{b}_i$ under $F\rightarrow G$.
        
        Then the Gromov seminorm of $\alpha$ (seen as a rational class) is given by
        \begin{multline*}
            \gromnorm{\alpha}=\sup\left\{\frac{1}{\inftynorm{\psi}}\left(\psi\left(a_1,b_1\right)+\psi\left(a_1b_1,a_1^{-1}\right)+\psi\left(a_1b_1a_1^{-1},b_1^{-1}\right)+\cdots\right.\right.\\
            \left.\vphantom{\frac{1}{\inftynorm{\psi}}}\left.+\psi\left(\left[a_1,b_1\right]\cdots\left[a_{k-1},b_{k-1}\right]a_kb_ka_{k}^{-1},b_k^{-1}\right)\right)\st{}\textrm{$\psi:G^2\rightarrow\mathbb{R}$ a bounded $2$-cocycle}\right\}.
        \end{multline*}
    \end{mthm}
    
    Our third application is a generalisation of a result of Calegari \cite{cal-fnb} on a connection between $\scl$ and the rotation quasimorphism in compact hyperbolic surfaces $S$.
    Calegari proves in particular that, when $\partial S\neq\emptyset$, any immersed admissible surface is extremal for $\scl$, and when such a surface exists, then the rotation quasimorphism is also extremal \cite{cal-fnb}*{Proposition 3.8}.
    He also proves that, in the closed surface case, immersed admissible surfaces exist for all rational chains \cite{cal-fnb}*{Theorem C}, and he hints at the fact that his arguments give extremality of immersed admissible surfaces for the relative Gromov seminorm \cite{cal-fnb}*{Remark 3.18}.
    We make this explicit, using the \emph{bounded Euler class} as the analogue of the rotation quasimorphism in the context of the relative Gromov seminorm:
    
    \begin{mthm}{\ref{thm:eub-extr}}[Extremality of the bounded Euler class]
        Let $\gamma:\coprod S^1\rightarrow S$ be a collection of geodesic loops in a compact hyperbolic surface $S$.
        Let $\alpha\in H_2(S,\gamma;\mathbb{Q})$ be projectively represented by a positive immersion $f:\left(\Sigma,\partial\Sigma\right)\looparrowright\left(S,\gamma\right)$.
        Then
        \[
            \gromnorm{\alpha}=\frac{-2\chi^-(\Sigma)}{n(\Sigma)}=-2\left<\eu_b^\mathbb{R}(S),\alpha\right>,
        \]
        where $\eu_b^\mathbb{R}(S)\in H^2_b\left(S;\mathbb{R}\right)$ is the real bounded Euler class of $S$.
        
        In other words, $f$ is an extremal surface and $-\eu_b^\mathbb{R}(S)$ is an extremal class for $\alpha$.
        In particular, $\gromnorm{\alpha}\in\mathbb{Q}$.
    \end{mthm}

    \subsection*{Outline of the paper}
    
    We start in \S{}\ref{sec:gromov} by introducing the  homology of a space relative to a collection of loops and the associated $\ell^1$-seminorm; we show that it can be given a topological interpretation and deduce Proposition \ref{prop:scl-gromnorm}, which connects it to stable commutator length.
    We then show in \S{}\ref{sec:bavard} how bounded cohomology gives a homological version of Bavard duality, namely Theorem \ref{thm:bav-dual-rel-grom}.
    We give a first application in \S{}\ref{sec:hnn} to the context of graphs of groups, yielding Theorem \ref{thm:isom-grph-gps}.
    In \S{}\ref{sec:hopf}, we prove a relative Hopf formula and obtain Theorem \ref{thm:bavard-hopf}.
    Finally, \S{}\ref{sec:euler} is devoted to the connection between the relative Gromov seminorm and the bounded Euler class, leading to Theorem \ref{thm:eub-extr}.
    
    While \S{}\ref{sec:gromov} and \S{}\ref{sec:bavard} lay the foundations of this paper, \S{}\ref{sec:hnn}, \S{}\ref{sec:hopf}, and \S{}\ref{sec:euler} can be read independently of each other.
    
    \subsection*{Acknowledgements}
    
    I would like to thank my supervisor Henry Wilton for many helpful discussions.
    I am also extremely grateful to Pierre de la Harpe, whose detailed comments on this paper have greatly improved the clarity of the exposition.
    This work was funded by an EPSRC PhD studentship.

\section{The relative Gromov seminorm}\label{sec:gromov}

    The Gromov seminorm will be our measure of complexity for relative homology classes.
    We approach it from two points of view: first as an $\ell^1$-seminorm, then as a measure of the minimal complexity of surfaces representing a given class.
    We'll show that, for rational classes, those two points of view coincide.
    This is well-known for absolute homology \cite{cal-scl}*{\S{}1.2.5}, and we adapt previous arguments to the relative case.

    \subsection{Conjugacy classes of \texorpdfstring{$1$}{1}-chains}\label{subsec:1-chns}
    
    Let $X$ be a path-connected topological space, and let $G=\pi_1X$.
    There is a correspondence between free homotopy classes of loops $S^1\rightarrow X$ and conjugacy classes of elements of $\pi_1X$.
    In this paper, we are interested in free homotopy classes of finite collections of loops $\coprod S^1\rightarrow X$, which can be encoded by certain classes of $1$-chains on $\pi_1X$, as we explain below.
    
    Fix a coefficient ring $R=\mathbb{Z}$ or $\mathbb{Q}$ or $\mathbb{R}$.
    We denote by $C_n\left(G;R\right)$ the group of $n$-chains on $G$ with coefficients in $R$:
    \[
        C_n\left(G;R\right)\coloneqq\bigoplus_{g_1,\dots,g_n\in G}R\left(g_1,\dots, g_n\right).
    \]
    These form a chain complex $C_*\left(G;R\right)$, with boundary maps $d_n:C_n\left(G;R\right)\rightarrow C_{n-1}\left(G;R\right)$ given by
    \begin{multline*}
        d_n\left(w_1,\dots, w_n\right)\coloneqq\left(w_2,\dots, w_n\right)-\left(w_1w_2, w_3,\dots, w_n\right)+\left(w_1, w_2w_3,\dots, w_n\right)\\
        -\dots+(-1)^{n-1}\left(w_1,\dots, w_{n-2}, w_{n-1}w_n\right)+(-1)^n\left(w_1,\dots, w_{n-1}\right).
    \end{multline*}
    The homology of the chain complex $C_*\left(G;R\right)$ is the group homology $H_*\left(G;R\right)$.
    See \cite{weibel}*{\S{}6.5} for more details: our $C_*(G;R)$ is the tensor product of the $\mathbb{Z}G$-module $R$ (with trivial $G$-action) with the bar resolution of $G$.
        
    We will also use the following notations:
    \begin{itemize}
        \item $Z_n\left(G;R\right)\coloneqq\Ker\left(d_{n}\right)\subseteq C_n\left(G;R\right)$ is the group of \emph{$n$-cycles}.
        \item $B_n\left(G;R\right)\coloneqq\Imm\left(d_{n+1}\right)\subseteq Z_n\left(G;R\right)$ is the group of \emph{$n$-boundaries}.
    \end{itemize}
    
    We now focus on $1$-chains, i.e.\ elements of $C_1\left(G;R\right)=\bigoplus_{g\in G}Rg$.
    The \emph{support} of a $1$-chain $c=\sum_{g}\lambda_gg$ (with $\lambda_g\in R$ for $g\in G$) is the finite set $\Supp c\coloneqq\left\{g\in G\st{}\lambda_g\neq0\right\}$.
    
    We consider the sub-$R$-module $K\left(G;R\right)$ of $C_1\left(G;R\right)$ spanned by elements of the form $\left(w-w^{t}\right)$ for $w,t\in G$, where we write $w^t=t^{-1}wt$.
    
    \begin{rmk}\label{rk:incl-kgr}
        $K\left(G;R\right)\subseteq B_1\left(G;R\right)\subseteq Z_1\left(G;R\right)=C_1\left(G;R\right)$.
    \end{rmk}
    \begin{proof}
        The equality $Z_1\left(G;R\right)=C_1\left(G;R\right)$ follows from the fact that $d_1=0$, and the inclusion $B_1\left(G;R\right)\subseteq Z_1\left(G;R\right)$ is because $C_*\left(G;R\right)$ is a chain complex.
        It remains to show that $K\left(G;R\right)\subseteq B_1\left(G;R\right)$.
        This follows from the following computation, for $w,t\in G$:
        \begin{align*}
            d_2\left(\left(t^{-1},wt\right)+\left(w,t\right)-\left(t^{-1},t\right)-\left(1,1\right)\right)
            &=\left(wt-w^t+t^{-1}\right)+\left(t-wt+w\right)
            \\&\phantom{aaa}-\left(t-1+t^{-1}\right)-\left(1-1+1\right)
            \\&=w-w^t.\qedhere
        \end{align*}
    \end{proof}
    
    \begin{defi}
        The $R$-module of \emph{conjugacy classes of chains} on $G$ with coefficients in $R$ is the quotient
        \[
            C_1^{\hmtp}\left(G;R\right)\coloneqq C_1\left(G;R\right)/K\left(G;R\right).
        \]
        We will also denote by $B_1^{\hmtp}\left(G;R\right)$ the image of $B_1\left(G;R\right)$ in $C_1^{\hmtp}\left(G;R\right)$ --- elements of $B_1^{\hmtp}\left(G;R\right)$ are called \emph{conjugacy classes of boundaries}\footnote{Compare with \cite{cal-scl}*{Definition 2.78}, where Calegari introduces a quotient of $B_1^{\hmtp}\left(G;R\right)$.}.
    \end{defi}
    
    We denote by $\pi:C_1\left(G;R\right)\twoheadrightarrow C_1^{\hmtp}\left(G;R\right)$ the projection.
    Given a $1$-chain $c\in C_1\left(G;R\right)$, we write $[c]\coloneqq\pi(c)\in C_1^{\hmtp}\left(G;R\right)$.
    
    Note that there is a new chain complex
    \[
        \cdots\xrightarrow{d_{n+1}}C_n\left(G;R\right)\xrightarrow{d_n}\cdots\xrightarrow{d_3}C_2\left(G;R\right)\xrightarrow{d_2^c\coloneqq\pi\circ d_2}C_1^{\hmtp}\left(G;R\right)\xrightarrow{d_1=0}C_0\left(G;R\right),
    \]
    which we denote by $C_*^{\hmtp}\left(G;R\right)$.
    In fact, this chain complex can be used to compute the first homology of $G$:
    \begin{rmk}\label{rk:compute-homol-homog}
        There is an isomorphism $H_1\left(G;R\right)\cong C_1^{\hmtp}\left(G;R\right)/B_1^{\hmtp}\left(G;R\right)$.
    \end{rmk}
    \begin{proof}
        There is a commutative diagram with exact rows:
            \[
                \vcenter{\hbox{\begin{tikzpicture}[every node/.style={draw=none,fill=none,rectangle},xscale=1.4]
                    \node (A) at (0,-1.25) {$C_2\left(G;R\right)$};
                    \node (Ap) at (0,0) {$C_2(G;R)$};
                    \node (B) at (2,-1.25) {$C_1^{\hmtp}\left(G;R\right)$};
                    \node (Bp) at (2,0) {$C_1\left(G;R\right)$};
                    \node (C) at (4,-1.25) {$H_1\left(C_*^{\hmtp}\left(G;R\right)\right)$};
                    \node (Cp) at (4,0) {$H_1\left(G;R\right)$};
                    \node (D) at (5.5,-1.25) {$0$};
                    \node (Dp) at (5.5,0) {$0$};
                    
                    \node [rotate=-90,xscale=2,yscale=1.2] at (0,-.625) {$=$};
                    
                    \begin{scope}[decoration={
                        markings,
                        mark=at position .9 with {\arrow{>}}}
                        ] 
                        \draw[postaction={decorate},->] (Bp)->(B) node [midway,right] {$\pi$};
                        \draw[postaction={decorate},->] (Cp)->(C) node [midway,right] {$\pi_*$};
                    \end{scope}
                    
                    \draw [->] (A) -> (B) node [midway,above] {$d_2^c$};
                    \draw [->] (B) -> (C);
                    \draw [->] (C) -> (D);
                    
                    \draw [->] (Ap) -> (Bp) node [midway,above] {$d_2$};
                    \draw [->] (Bp) -> (Cp);
                    \draw [->] (Cp) -> (Dp);
                \end{tikzpicture}}}
            \]
        Given $\alpha\in\Ker\pi_*\subseteq H_1\left(G;R\right)$, pick $a\in C_1\left(G;R\right)$ mapping to $\alpha$ under $C_1\left(G;R\right)\rightarrow H_1\left(G;R\right)$.
        Then $\pi(a)\in\Imm\left(d_2^c\right)$, i.e.\ there is $b\in C_2\left(G;R\right)$ such that $a-d_2b\in\Ker\pi=K\left(G;R\right)$.
        But $K\left(G;R\right)\subseteq B_1\left(G;R\right)=\Imm d_2$ (see Remark \ref{rk:incl-kgr}), so $a\in\Imm d_2$ and $\alpha=0$.
        Therefore, $\pi_*:H_1\left(G;R\right)\rightarrow H_1\left(C_*^{\hmtp}\left(G;R\right)\right)$ is an isomorphism, which proves the result since $H_1\left(C_*^{\hmtp}\left(G;R\right)\right)\cong C_1^{\hmtp}\left(G;R\right)/B_1^{\hmtp}\left(G;R\right)$.
    \end{proof}
    
    \subsection{Standard form for conjugacy classes of chains}\label{subsec:std-form-chns}
    
    We now give for each conjugacy class of chains in $C_1^{\hmtp}\left(G;R\right)$ a standard representative in $C_1\left(G;R\right)$ having a natural topological counterpart.
    
    This standard representative will be unique up to reordering and conjugacy, and we will use the following lemma to prove uniqueness:
    
    \begin{lmm}\label{lmm:kgz}
        Let $\kappa\in K\left(G;R\right)$.
        Suppose that there is no pair of distinct conjugate elements in $\Supp\kappa$.
        Then $\kappa=0$.
    \end{lmm}
    \begin{proof}
        By definition, $\kappa$ can be written as a linear combination
        \begin{equation}\label{eq:std-form-kgr}
            \kappa=\sum_{i=1}^r\lambda_i\left(w_i-w_i^{t_i}\right),
        \end{equation}
        with $\lambda_1,\dots,\lambda_r\in R\smallsetminus\{0\}$, $w_1,\dots,w_r\in G$, $t_1,\dots,t_r\in G$, and $w_i\neq w_i^{t_i}$ for all $i$.
        We choose a decomposition (\ref{eq:std-form-kgr}) such that $r$ is minimal.
        Assume for contradiction that $\kappa\neq0$.
        In particular, $r\geq1$ and $\Supp\kappa\neq\emptyset$.
        After reordering, we may assume that at least one of $w_r$ and $w_r^{t_r}$ lies in $\Supp\kappa$.
        However, there is no pair of distinct conjugate elements in $\Supp\kappa$.
        Without loss of generality, we can therefore assume that $w_r\in\Supp\kappa$ and $w_r^{t_r}\not\in\Supp\kappa$.
        Now define
        \begin{align*}
            I_1&\coloneqq\left\{i<r\st{}w_i=w_r^{t_r}\right\},\\
            I_2&\coloneqq\left\{i<r\st{}w_i^{t_i}=w_r^{t_r}\right\}.
        \end{align*}
        The sets $I_1$ and $I_2$ are disjoint since $w_i\neq w_i^{t_i}$ for all $i$, and we also set $I_0\coloneqq\left\{1,\dots,r-1\right\}\smallsetminus\left(I_1\amalg I_2\right)$, so that $I_0,I_1,I_2$ form a partition of $\left\{1,\dots,r-1\right\}$.
        Since the coefficient of $w_r^{t_r}$ in $\kappa$ vanishes, we have
        \[
            \lambda_r=\sum_{i\in I_1}\lambda_i-\sum_{i\in I_2}\lambda_i.
        \]
        Therefore, setting $p_i=\left(w_i-w_i^{t_i}\right)$ for each $i\in\left\{1,\dots,r\right\}$, we can rewrite
        \begin{equation}\label{eq:rewrite-kappa}
            \kappa=\sum_{i\in I_0}\lambda_ip_i+\sum_{i\in I_1}\lambda_i\left(p_i+p_r\right)+\sum_{i\in I_2}\lambda_i\left(p_i-p_r\right).
        \end{equation}
        Now note that:
        \begin{itemize}
            \item For $i\in I_1$, $p_i+p_r=w_r-w_i^{t_i}=w_r-w_r^{t_rt_i}$.
            \item For $i\in I_2$, $p_i-p_r=w_i-w_r=w_i-w_i^{t_it_r^{-1}}$.
        \end{itemize}
        Therefore, (\ref{eq:rewrite-kappa}) is a decomposition of $\kappa$ of the form (\ref{eq:std-form-kgr}) with at most $\left|I_0\right|+\left|I_1\right|+\left|I_2\right|=r-1$ terms.
        This contradicts the minimality of $r$, so $\kappa=0$.\qedhere
    \end{proof}
    
    We can now obtain our standard form:
    
    \begin{lmm}[Standard form for conjugacy classes of chains]\label{lmm:std-form-chns}
        Let $[c]\in C_1^{\hmtp}\left(G;R\right)$.
        \begin{enumerate}
            \item There is a $1$-chain
            \[
                c_0=\sum_{i=1}^d\lambda_iw_i\in C_1\left(G;R\right),
            \]
            such that $\left[c_0\right]=[c]$ in $C_1^{\hmtp}\left(G;R\right)$, where $d\in\mathbb{N}_{\geq0}$, $\lambda_1,\dots,\lambda_d\in R\smallsetminus\{0\}$, and $w_1,\dots,w_d\in G$ are pairwise non-conjugate.
            
            \item Assume that $c'_0=\sum_{i=1}^{d'}\lambda'_iw'_i\in C_1\left(G;R\right)$ also satisfies $\left[c_0\right]=\left[c_0'\right]$, where $w'_1,\dots,w'_{d'}$ are pairwise non-conjugate.
            Then $d=d'$, and there is a permutation $\sigma\in\mathfrak{S}_d$, and elements $t_i\in G$, such that $w'_{\sigma(i)}=w_i^{t_i}$ and $\lambda'_{\sigma(i)}=\lambda_i$ for all $i$.\label{lmm:std-form-chns:2}
        \end{enumerate}
    \end{lmm}
    \begin{proof}
        \begin{enumerate}
            \item Write $c=\sum_{i=1}^d\lambda_iw_i\in C_1\left(G;R\right)$, with $\lambda_1,\dots,\lambda_d\in R\smallsetminus\{0\}$ and $w_1,\dots,w_d\in G$.
            Assume moreover that $d$ is minimal among all representatives $c$ of the class $[c]$.
            If there are $j\neq k$ such that $w_j=w_k^t$ for some $t\in G$, then
            \[
                c\equiv\left(\lambda_j+\lambda_k\right)w_j+\sum_{\substack{1\leq i\leq d\\i\neq j,k}}\lambda_iw_i\mod K\left(G;R\right),
            \]
            which contradicts the minimality of $d$.
            Therefore, no two of the $w_i$'s are conjugate as wanted.
            
            \item We argue by induction on $d+d'$.
            If $d+d'=0$, then $c_0=0=c_0'$.
            Assume that $d+d'\geq 1$ and consider $\kappa\coloneqq c_0-c_0'\in K\left(G;R\right)$.
            If $\kappa=0$, then $c_0=c_0'$ and there is nothing to prove; otherwise, Lemma \ref{lmm:kgz} implies the existence of a pair of distinct conjugate elements in $\Supp\kappa$.
            But $\Supp\kappa\subseteq\left\{w_i\right\}_{1\leq i\leq d}\cup\left\{w'_i\right\}_{1\leq i\leq d'}$.
            By assumption on the $w_i$'s and $w'_i$'s, this implies that one of the $w_i$'s is conjugate to one of the $w'_i$'s.
            After relabelling, we can assume that $w_d$ is conjugate to $w'_{d'}$.
            We then consider
            \[
                c_1\coloneqq\left(\lambda_d-\lambda'_{d'}\right)w_d+\sum_{i<d}\lambda_iw_i\quad\text{and}\quad c'_1\coloneqq\sum_{i<d'}\lambda'_iw'_i.
            \]
            Note that $c_1\equiv c_1'\mod K\left(G;R\right)$, so the induction hypothesis applies to $c_1$ and $c_1'$.
            If $\lambda_d\neq\lambda'_{d'}$, then we deduce that $w_d$ is conjugate to some $w'_i$ with $i<d'$, and therefore $w'_{d'}$ is conjugate to $w'_i$, which is a contradiction.
            Therefore, $\lambda_d=\lambda'_{d'}$, and the result follows from the induction hypothesis applied to $c_1$ and $c_1'$.\qedhere
        \end{enumerate}
    \end{proof}
    
    \begin{rmk}
        \begin{enumerate}
            \item A group element can be seen as an element of $C_1\left(G;\mathbb{Z}\right)$, and conjugacy classes of chains generalise conjugacy classes of group elements, in the sense that, for $w\in G$, Lemma \ref{lmm:std-form-chns} implies that $\pi^{-1}\left([w]\right)\cap G$ is exactly the conjugacy class of $w$ in $G$.\label{rk:homot-clss-chns:1}
            \item The equivalence relation given by $K\left(G;R\right)$ on $1$-chains should be thought of as the algebraic analogue of (free) homotopy.
            This is parallel to the equivalence relation given by $B_1\left(G;R\right)$, which is the algebraic analogue of homology.
            Hence, Remark \ref{rk:incl-kgr} is an algebraic formulation of the fact that homotopic maps also represent the same class in homology.
        \end{enumerate}
        \label{rk:homot-clss-chns}
    \end{rmk}
    
    We now see $G$ as the fundamental group of a path-connected space $X$.
    Pick an integral conjugacy class $[c]\in C_1^{\hmtp}\left(G;\mathbb{Z}\right)$, and let $c_0=\sum_i\lambda_iw_i\in C_1\left(G;\mathbb{Z}\right)$ be a standard representative of $[c]$ given by Lemma \ref{lmm:std-form-chns}.
    For each $i$, pick a loop $\gamma_i:S^1\rightarrow X$ whose free homotopy class in $X$ corresponds to the conjugacy class of $w_i^{\lambda_i}$ in $G$, and consider the map
    \[
        \textstyle\gamma\coloneqq\coprod_i\gamma_i:\coprod_iS^1\rightarrow X.
    \]
    By Lemma \ref{lmm:std-form-chns}\ref{lmm:std-form-chns:2}, the free homotopy class $[\gamma]$ of $\gamma$ only depends on the class $[c]$.
    If $\Xi$ is the set of free homotopy classes of finite unordered collections of pairwise non-homotopic oriented loops $\coprod S^1\rightarrow X$, then this defines a map
    \[
        C_1^{\hmtp}\left(G;\mathbb{Z}\right)\rightarrow\Xi
    \]
    given by $[c]\mapsto\left[\gamma\right]$.
    We will say that the map $\gamma$ \emph{represents} the conjugacy class $[c]$.
    
    Conversely, consider the free homotopy class of a map $\gamma:\coprod S^1\rightarrow X$, with components $\left\{\gamma_i:S^1\rightarrow X\right\}_i$.
    For each $i$, pick an element $w_i\in G$ whose conjugacy class corresponds to the free homotopy class of $\gamma_i$.
    Sending the free homotopy class of $\gamma$ to $\left[\sum_iw_i\right]\in C_1^{\hmtp}\left(G;\mathbb{Z}\right)$ defines a right inverse to the map $C_1^{\hmtp}\left(G;\mathbb{Z}\right)\rightarrow\Xi$ constructed above.
    
    Hence, the map $C_1^{\hmtp}\left(G;\mathbb{Z}\right)\rightarrow\Xi$ is surjective, but note that it is not injective: given $w\in G\smallsetminus\{1\}$ and $\lambda\in\mathbb{Z}\smallsetminus\{1\}$, the conjugacy classes $\left[\lambda w\right]$ and $\left[w^\lambda\right]$ are distinct but are represented by the same (homotopy class of) loop $\gamma:S^1\rightarrow X$.
    
    \begin{rmk}\label{rmk:scl-norm}
        The algebraic definition of stable commutator length ($\scl$) in terms of products of commutators as a function $G\rightarrow\left[0,\infty\right]$ can be shown to extend to $C_1^{\hmtp}\left(G;\mathbb{Z}\right)$, and then by linearity to $C_1^{\hmtp}(G;\mathbb{R})$ --- see \cite{cal-scl}*{\S{}2.6}.
        For us, $\scl$ will be defined by its topological interpretation, given by Proposition \ref{prop:scl-gromnorm} below, and this definition will naturally be given for classes in $C_1^{\hmtp}\left(G;\mathbb{Z}\right)$.
    \end{rmk}
    
    \subsection{Homology of a space relative to a chain}\label{subsec:rel-hom}
    
    Let $\gamma:\coprod S^1\rightarrow X$ be a finite (unordered) collection of loops.
    We denote by $X_\gamma$ the \emph{mapping cylinder} of $\gamma$:
    \[
        \textstyle X_\gamma\coloneqq\left(X\amalg\left(\coprod S^1\times[0,1]\right)\right)/\sim,
    \]
    where $\sim$ is the equivalence relation generated by $(u,0)\sim\gamma(u)$ for $u\in\coprod S^1$.
    There is an embedding $\coprod S^1\hookrightarrow X_\gamma$ via $u\mapsto(u,1)$, and we will identify $\coprod S^1$ with its image under this embedding.
    
    \begin{defi}
        The homology of the pair $(X,\gamma)$ over the coefficient ring $R=\mathbb{Z}$ or $\mathbb{Q}$ or $\mathbb{R}$ is defined as the singular homology of the pair $\left(X_\gamma,\coprod S^1\right)$:
        \[
            \textstyle H_*\left(X,\gamma;R\right)\coloneqq H_*\left(X_\gamma,\coprod S^1;R\right).
        \]
    \end{defi}
        
    We remark that the homotopy type of the pair $\left(X_\gamma,\coprod S^1\right)$ --- and therefore the homology $H_*\left(X,\gamma\right)$ --- only depends on the free homotopy class of $\gamma$.
    
    It is useful to write down the long exact sequence of the pair $\left(X,\gamma\right)$:
    \begin{prop}\label{prop:les}
        Let $X$ be a topological space and $\gamma:\coprod S^1\rightarrow X$.
        Then there is an exact sequence
        \[
            \textstyle 0\rightarrow H_2\left(X;R\right)\rightarrow H_2\left(X,\gamma;R\right)\xrightarrow{\partial}H_1\left(\coprod S^1;R\right)\xrightarrow{\gamma_*}H_1\left(X;R\right).
        \]
    \end{prop}
    \begin{proof}
        This is simply the long exact sequence of the pair $\left(X_\gamma,\coprod S^1\right)$, together with the fact that $X_\gamma$ deformation retracts to $X$ \cite{hatcher}*{p.2}.
    \end{proof}
    
    We now compute $H_*\left(X,\gamma\right)$ in a few special cases:
    \begin{ex}
        \begin{enumerate}
            \item If $\gamma$ is the empty collection of loops, then\label{ex:rel-hom-chn-3}
            \[
                H_*\left(X,\gamma;R\right)\cong H_*\left(X;R\right).
            \]
            \item If $\gamma$ is an embedding $\coprod S^1\hookrightarrow X$, then the pair $\left(X_\gamma,\coprod S^1\right)$ deformation retracts onto $\left(X,\gamma\left(\coprod S^1\right)\right)$, and therefore there is an isomorphism
            \[
                \textstyle H_*\left(X,\gamma;R\right)\cong H_*\left(X,\gamma\left(\coprod S^1\right);R\right).
            \]
            In general, there is still a morphism $H_*\left(X,\gamma\right)\rightarrow H_*\left(X,\gamma\left(\coprod S^1\right)\right)$ given by collapsing the mapping cylinder, but this might not be an isomorphism, as shown for instance by item \ref{ex:rel-hom-chn-2} below.
            \item If $\gamma:S^1\rightarrow X$ is a contractible loop, then the quotient $X_\gamma/S^1$ is homotopy equivalent to $X\vee S^2$, and collapsing the pair gives\label{ex:rel-hom-chn-2}
                \[
                    \textstyle H_*\left(X,\gamma;R\right)\cong H_*\left(X;R\right)\oplus H_*\left(S^2;R\right).
                \]
            \item Suppose that $\gamma:S^1\rightarrow X$ is rationally homologically trivial, in the sense that $\gamma_*:H_1\left(S^1;\mathbb{Q}\right)\rightarrow H_1\left(X;\mathbb{Q}\right)$ vanishes.
            Then the map $\gamma_*:H_1\left(S^1;R\right)\rightarrow H_1\left(X;R\right)$ in the exact sequence of Proposition \ref{prop:les} has kernel $q\left[S^1\right]$ for some $q\in R$, so the image of the boundary map $\partial$ is isomorphic to $R$, which gives a split short exact sequence, and an isomorphism\label{ex:rel-hom-chn-4}
            \[
                H_2\left(X,\gamma;R\right)\cong H_2\left(X;R\right)\oplus R.
            \]
        \end{enumerate}
        \label{ex:rel-hom-chn}
    \end{ex}
    
    Note that there is a natural isomorphism
    \[
        H_*\left(X,\gamma;\mathbb{R}\right)\cong H_*\left(X,\gamma;\mathbb{Q}\right)\otimes_\mathbb{Q}\mathbb{R},
    \]
    allowing us to view $H_*\left(X,\gamma;\mathbb{Q}\right)$ as a subset of $H_*\left(X,\gamma;\mathbb{R}\right)$.
    We will say that $\alpha\in H_*\left(X,\gamma;\mathbb{Z}\right)$ is an \emph{integral class}, while $\alpha\in H_*\left(X,\gamma;\mathbb{Q}\right)$ is \emph{rational} and $\alpha\in H_*\left(X,\gamma;\mathbb{R}\right)$ is \emph{real}.
    
    \begin{defi}\label{def:relhom-gp}
        Let $G$ be a group, and let $X$ be a $K(G,1)$ space.
        Given $[c]\in C_1^{\hmtp}\left(G;\mathbb{Z}\right)$, we set
        \[
            H_*\left(G,[c];R\right)\coloneqq H_*\left(X,\gamma;R\right),
        \]
        where $\gamma:\coprod S^1\rightarrow X$ is a map representing $[c]$ as explained in \S{}\ref{subsec:std-form-chns}.
    \end{defi}
    
    Note that the homotopy type of $X$ is uniquely defined by $G$, and the free homotopy class of $\gamma$ is determined by $[c]$, so the group $H_*\left(G,[c];R\right)$ only depends on $G$ and the class $[c]$.
    
    The case where $[c]=0$ corresponds to $\gamma$ being the empty collection of loops, and so $H_*\left(G,0;R\right)\cong H_*(G;R)$ by Example \ref{ex:rel-hom-chn}\ref{ex:rel-hom-chn-3}.

    \subsection{Rational points in real vector spaces}
    
    The difference between {real} and {rational} classes in $H_2\left(X,\gamma\right)$ will play a role in the sequel, and we make a brief digression to introduce some general terminology related to this.
    
    \begin{defi}
        Let $V$ be a $\mathbb{R}$-vector space. A \emph{rational structure} on $V$ is the choice of an equivalence class of bases of $V$, where two bases are considered equivalent if each vector of one basis has rational coordinates in the second basis.
        Any basis in the equivalence class is called a \emph{rational basis}.
        
        Given a rational structure on $V$, a \emph{rational point} is a vector of $V$ that has rational coordinates in a rational basis.
        The set $V_\mathbb{Q}$ of rational points of $V$ is naturally a $\mathbb{Q}$-vector space, and satisfies $V=V_\mathbb{Q}\otimes_\mathbb{Q}\mathbb{R}$.
        In fact, a rational structure on $V$ can be defined equivalently as the choice of a $\mathbb{Q}$-subspace $V_\mathbb{Q}$ of $V$ such that $V=V_\mathbb{Q}\otimes_\mathbb{Q}\mathbb{R}$.
    \end{defi}
    
    \begin{ex}
        The space $\mathbb{R}^n$ has a rational structure given by the equivalence class of the standard basis, and its set of rational points is $\mathbb{Q}^n$.
    \end{ex}
    
    A \emph{rational subspace} $W$ of $V$ is a $\mathbb{R}$-subspace spanned by rational points.
    It naturally inherits a rational structure from $V$.
    
    If $V$ and $W$ are $\mathbb{R}$-vector spaces equipped with rational structures, a \emph{rational linear map} $f:V\rightarrow W$ is a linear map such that the image of each vector in a rational basis of $V$ has rational coordinates in a rational basis of $W$.
    This implies that the kernel and the image of $f$ are rational subspaces of $V$ and $W$ respectively.
    
    Let $C_*^\mathbb{Q}$ be a chain complex over $\mathbb{Q}$ and let $C_*^\mathbb{R}=C_*^\mathbb{Q}\otimes_\mathbb{Q}\mathbb{R}$.
    Hence, each vector space $C_n^\mathbb{R}$ has a rational structure whose set of rational points is $C_n^\mathbb{Q}$.
    The boundary map $d_n:C_n^\mathbb{R}\rightarrow C_{n-1}^\mathbb{R}$ is rational, and the space $Z_n^\mathbb{R}=\Ker d_n$ of $n$-cycles is a rational subspace.
    In particular, the set of rational points of $Z_n^\mathbb{R}$ is the space $Z_n^\mathbb{Q}$ of $n$-cycles for $C_*^\mathbb{Q}$.
    Moreover, there is an isomorphism
    \[
        H_n\left(C_*^\mathbb{R}\right)\cong H_n\left(C_*^\mathbb{Q}\right)\otimes_\mathbb{Q}\mathbb{R},
    \]
    giving $H_n\left(C_*^\mathbb{R}\right)$ a rational structure whose set of rational points is $H_n\left(C_*^\mathbb{Q}\right)$.
    
    The following lemma says that any real cycle representing a rational homology class can be approximated by a rational cycle:
    
    \begin{lmm}[Rational approximation in homology]\label{lmm:rat-approx-homol}
        Let $C_*^\mathbb{Q}$ be a chain complex over $\mathbb{Q}$ and let $C_*^\mathbb{R}=C_*^\mathbb{Q}\otimes_\mathbb{Q}\mathbb{R}$.
        Let $\normm{\cdot}$ be a norm on $C_*^\mathbb{R}$.
        Consider a real $n$-cycle $a\in Z_n^\mathbb{R}$ whose homology class $[a]$ is rational:
        \[
            [a]\in H_n\left(C_*^\mathbb{Q}\right)\subseteq H_n\left(C_*^\mathbb{R}\right).
        \]
        Then for any $\varepsilon>0$, there exists a rational $n$-cycle $a'\in Z_n^\mathbb{Q}$ such that
        \begin{itemize}
            \item $[a]=\left[a'\right]$ in $H_n\left(C_*^\mathbb{R}\right)$, and
            \item $\normm{a-a'}\leq\varepsilon$.
        \end{itemize}
    \end{lmm}
    \begin{proof}
        We follow an argument of Calegari \cite{cal-scl}*{Remark 1.5}.
        Observe that the natural projection map
        \[
            p:Z_n^\mathbb{R}\rightarrow H_n\left(C_*^\mathbb{R}\right)
        \]
        is rational.
        Hence, since $[a]$ is a rational point of $H_n\left(C_*^\mathbb{R}\right)$, the affine subspace $p^{-1}([a])$ is rational in $Z_n^\mathbb{R}$, so its rational points are contained in $Z_n^\mathbb{Q}$.
        We may assume that $Z_n^\mathbb{R}$ is finite-dimensional by restricting to a finite-dimensional rational subspace containing $a$; hence rational points are dense.
        Since the real $n$-cycle $a$ lies in $p^{-1}([a])$, there is $a'\in p^{-1}([a])$ rational arbitrarily close to $a$ for $\normm{\cdot}$.
        This rational $n$-cycle $a'$ lies in $Z_n^\mathbb{Q}$ and is homologous to $a$ as wanted.
    \end{proof}

    \subsection{The Gromov seminorm as an \texorpdfstring{$\ell^1$}{l1}-seminorm}
    
    We now give a first definition of the Gromov seminorm.
    
    Given $\gamma:\coprod S^1\rightarrow X$, recall that $H_*\left(X_\gamma;\mathbb{R}\right)$ is the homology of the singular chain complex $C_*^\textnormal{sg}\left(X_\gamma;\mathbb{R}\right)$.
    Each $\mathbb{R}$-vector space $C_n^\textnormal{sg}\left(X_\gamma;\mathbb{R}\right)$ can be equipped with the $\ell^1$-norm defined by
    \[
        \normm{\sum_{\sigma}\lambda_\sigma\sigma}_1\coloneqq\sum_\sigma\left|\lambda_\sigma\right|,
    \]
    with $\lambda_\sigma\in\mathbb{R}$ for each singular $n$-simplex $\sigma:\Delta^n\rightarrow X_\gamma$.
    The quotient
    \[
        \textstyle C_n^\textnormal{sg}\left(X_\gamma,\coprod S^1;\mathbb{R}\right)\coloneqq C_n^\textnormal{sg}\left(X_\gamma;\mathbb{R}\right)/C_n^\textnormal{sg}\left(\coprod S^1;\mathbb{R}\right)
    \]
    inherits a quotient seminorm that we also denote by $\normm{\cdot}_1$, and that is defined by
    \[
        \normm{a}_1\coloneqq\inf_{\underline{a}\in a}\normm{\underline{a}}_1,
    \]
    where the infimum is over all absolute $n$-chains $\underline{a}\in C_n^\textnormal{sg}\left(X_\gamma;\mathbb{R}\right)$ representing $a\in C_n^\textnormal{sg}\left(X_\gamma,\coprod S^1;\mathbb{R}\right)$.
    The restriction of $\normm{\cdot}_1$ defines a seminorm on the subspace $Z_n^\textnormal{sg}\left(X_\gamma,\coprod S^1;\mathbb{R}\right)$ of relative $n$-cycles, which descends to a seminorm, still denoted by $\gromnorm{\cdot}$, on homology:
    
    \begin{defi}
        Let $X$ be a topological space and $\gamma:\coprod S^1\rightarrow X$.
        The \emph{relative Gromov seminorm} on $H_n\left(X,\gamma;\mathbb{R}\right)$ is defined by
        \[
            \gromnorm{\alpha}\coloneqq\inf\left\{\gromnorm{a}\st{} a\in Z_n^\textnormal{sg}\left(X_\gamma,\textstyle\coprod S^1;\mathbb{R}\right),\:{[a]}=\alpha\right\}.
        \]
    \end{defi}
    
    \begin{rmk}\label{rmk:defi-grnorm-gp}
        Given a countable group $G$ and an integral conjugacy class $[c]\in C_1^{\hmtp}\left(G;\mathbb{Z}\right)$, the relative homology $H_2\left(G,[c];\mathbb{R}\right)$ is by definition $H_2\left(X,\gamma;\mathbb{R}\right)$, where $X$ is a $K(G,1)$, which can be chosen to be a countable CW-complex since $G$ is countable, and $\gamma:\coprod S^1\rightarrow X$ represents $[c]$ --- see Definition \ref{def:relhom-gp}.
        If $X'$ is another choice of (countable) $K(G,1)$ and $\gamma':\coprod S^1\rightarrow X'$ is another map representing $[c]$, then there is a homotopy equivalence $h:X\xrightarrow{\simeq}X'$ sending $\gamma$ to $h\gamma$, and a free homotopy between $h\gamma$ and $\gamma'$.
        Hence, there are induced homotopy equivalences of pairs
        \[
            \textstyle\left(X_\gamma,\coprod S^1\right)\simeq\left(X'_{h\gamma},\coprod S^1\right)\simeq\left(X'_{\gamma'},\coprod S^1\right).
        \]
        Since $\gromnorm{\cdot}$ is invariant under homotopy equivalence for countable CW-complexes, and in fact under any map inducing an isomorphism of fundamental groups\footnote{This follows from Gromov's Mapping Theorem \cite{frigerio}*{Corollary 5.11}, together with the duality principle between $\ell^1$-homology and bounded cohomology \cite{frigerio}*{Corollary 6.2}.}, the above homotopy equivalences induce isometric isomorphisms
        \[
            \textstyle H_2\left(X,\gamma;\mathbb{R}\right)\cong H_2\left(X',h\gamma;\mathbb{R}\right)\cong H_2\left(X',\gamma';\mathbb{R}\right).
        \]
        Hence, one can extend the definition of the Gromov seminorm to $H_2\left(G,[c];\mathbb{R}\right)$.
    \end{rmk}
    
    The above definitions still make sense if $\mathbb{R}$ is replaced with $\mathbb{Q}$ everywhere.
    Given $\alpha\in H_n\left(X,\gamma;\mathbb{Q}\right)\subseteq H_n\left(X,\gamma;\mathbb{R}\right)$, it is natural to ask whether the Gromov seminorm of $\alpha$ as a {rational class} coincides with its Gromov seminorm as a {real class}.
    The following lemma gives an affirmative answer:
    
    \begin{lmm}[Equality of the rational and real Gromov seminorms]\label{lmm:l1-rat-classes}
        Let $X$ be a topological space and $\gamma:\coprod S^1\rightarrow X$.
        Given a rational class $\alpha\in H_n\left(X,\gamma;\mathbb{Q}\right)$, the Gromov seminorm of $\alpha$ (seen as a real class) can be computed over rational cycles:
        \[
            \gromnorm{\alpha}=\inf\left\{\gromnorm{a}\st{} a\in Z_n^\textnormal{sg}\left(X_\gamma,\textstyle\coprod S^1;\mathbb{Q}\right),\:{[a]}=\alpha\right\}
        \]
    \end{lmm}
    \begin{proof}
        This follows from Lemma \ref{lmm:rat-approx-homol}.
    \end{proof}
    
    In other words, Lemma \ref{lmm:l1-rat-classes} says that the inclusion $H_n\left(X,\gamma;\mathbb{Q}\right)\hookrightarrow H_n\left(X,\gamma;\mathbb{R}\right)$ is an isometric embedding if $H_n\left(X,\gamma;\mathbb{Q}\right)$ and $H_n\left(X,\gamma;\mathbb{R}\right)$ are equipped with the rational and real Gromov seminorms respectively.

    \subsection{Topological interpretation of the Gromov seminorm}\label{subsec:top-int}
    
    Analogously to (and motivated by) the topological interpretation of stable commutator length in terms of surfaces projectively bounding a given loop \cite{cal-scl}*{\S{}2.6}, we now give a topological interpretation of the Gromov seminorm for rational classes in $H_2$.
    This extends the topological interpretation of the absolute Gromov seminorm on $H_2$ \cite{cal-scl}*{\S{}1.2.5}.
    
    An \emph{admissible surface} for $\gamma:\coprod S^1\rightarrow X$ is the data of an oriented compact (possibly disconnected) surface $\Sigma$, and of maps $f:\Sigma\rightarrow X$ and $\partial f:\partial\Sigma\rightarrow\coprod S^1$ making the following diagram commute:
    \begin{equation*}
        \vcenter{\hbox{\begin{tikzpicture}[every node/.style={draw=none,fill=none,rectangle}]
            \node (A) at (0,1.5) {$\partial\Sigma$};
            \node (B) at (2,1.5) {$\Sigma$};
            \node (Ap) at (0,0) {$\coprod S^1$};
            \node (Bp) at (2,0) {$X$};
            
            \draw [->] (A) -> (B) node [midway,above] {$\iota$};
            \draw [->] (Ap) -> (Bp) node [midway,above] {$\gamma$};
            \draw [->] (A) -> (Ap) node [midway,left] {$\partial f$};
            \draw [->] (B) -> (Bp) node [midway,left] {$f$};
        \end{tikzpicture}}}
        \label{eq:comm-diag-adm-surf}
    \end{equation*}
    where $\iota:\partial\Sigma\hookrightarrow\Sigma$ is the inclusion.
    Such an admissible surface will be denoted by $f:\left(\Sigma,\partial\Sigma\right)\rightarrow\left(X,\gamma\right)$.
    
    Let $\Sigma_\iota$ be the mapping cylinder of the inclusion map $\iota:\partial\Sigma\hookrightarrow\Sigma$:
    \[
        \Sigma_\iota\coloneqq\left(\Sigma\amalg\left(\partial\Sigma\times[0,1]\right)\right)/\sim,
    \]
    where $\sim$ is the equivalence relation generated by $(u,0)\sim\iota(u)$ for $u\in\partial\Sigma$.
    Hence, there is a natural map of pairs 
    \[
        \textstyle f_\#:\left(\Sigma_\iota,\partial\Sigma\times\{1\}\right)\rightarrow\left(X_\gamma,\coprod S^1\right)
    \]
    defined by $f$ and $\partial f$ --- see \S{}\ref{subsec:rel-hom} for the definition of the pair $\left(X_\gamma,\coprod S^1\right)$.
    Since the pair $\left(\Sigma_\iota,\partial\Sigma\times\{1\}\right)$ deformation retracts to $\left(\Sigma,\partial\Sigma\right)$, the map $f_\#$ induces a morphism
    \[
        f_*:H_*\left(\Sigma,\partial\Sigma\right)\rightarrow H_*\left(X,\gamma\right).
    \]
    
    In particular, $f$ represents a class $f_*[\Sigma]\in H_2\left(X,\gamma\right)$, where $[\Sigma]\in H_2\left(\Sigma,\partial\Sigma\right)$ is the (integral, rational, or real) fundamental class of $\Sigma$.
    
    The topological complexity of a compact surface $\Sigma$ will be measured by its \emph{reduced Euler characteristic}, defined by $\chi^-(\Sigma)=\sum_K\min\left\{0,\chi(K)\right\}$, where the sum is over all connected components $K$ of $\Sigma$.
    
    \begin{prop}[Topological interpretation of the Gromov seminorm]\label{prop:ell1-gromnorm}
        Let $X$ be a topological space and $\gamma:\coprod S^1\rightarrow X$.
        If $\alpha\in H_2\left(X,\gamma;\mathbb{Q}\right)$ is a rational class, then there is an equality
        \[
            \gromnorm{\alpha}=\inf_{f,\Sigma}\frac{-2\chi^-(\Sigma)}{n(\Sigma)},
        \]
        where the infimum is taken over all admissible surfaces $f:\left(\Sigma,\partial\Sigma\right)\rightarrow\left(X,\gamma\right)$ such that $f_*[\Sigma]=n(\Sigma)\alpha$ for some $n(\Sigma)\in\mathbb{N}_{\geq1}$.
    \end{prop}
    
    \begin{proof}
        First consider an admissible surface $f:\left(\Sigma,\partial\Sigma\right)\rightarrow\left(X,\gamma\right)$ with $f_*[\Sigma]=n(\Sigma)\alpha$.
        Then we can estimate
        \[
            \gromnorm{\alpha}=\frac{\gromnorm{f_*[\Sigma]}}{n(\Sigma)}\leq\frac{\normm{[\Sigma]}_1}{n(\Sigma)}.
        \]
        But the $\ell^1$-seminorm of $[\Sigma]$ is known as the \emph{simplicial volume} of $\Sigma$, and it is equal to $-2\chi^-(\Sigma)$ \cite{frigerio}*{Corollary 7.5}.
        This proves the inequality $\left(\leq\right)$ of the proposition.
        
        For the reverse inequality, we follow the same line of reasoning as in Calegari's proof that $\scl$ is not greater than the Gersten boundary norm \cite{cal-scl}*{Lemma 2.69}, which is based on an argument of Bavard \cite{bavard}*{Proposition 3.2}.
        Let $a\in Z_2\left(X_\gamma,\coprod S^1;\mathbb{Q}\right)$ be a rational relative $2$-cycle representing $\alpha$, and let $a_0\in C_2\left(X_\gamma;\mathbb{Q}\right)$ be a $2$-chain mapping to $a$.
        By Lemma \ref{lmm:l1-rat-classes}, the infimum of $\gromnorm{a_0}$ over such $a_0$ is equal to $\gromnorm{\alpha}$.
                
        Since $a_0$ is rational, there exists $q\in\mathbb{N}_{\geq1}$ such that $qa_0$ is integral; we can write $qa_0=\sum_j\varepsilon_j\sigma_j$, with $\varepsilon_j\in\{\pm1\}$ and $\sigma_j:\Delta^2\rightarrow X_\gamma$ a singular $2$-simplex.
        We can assume that no singular $2$-simplex appears twice with opposite signs in the above expression, so that
        \[
            \normm{qa_0}_1=\sum_j\left|\varepsilon_j\right|.
        \]
        The fact that $a$ is a relative $2$-cycle means that $da_0$ has support contained in $\coprod S^1$.
        Therefore, we can construct a partial pairing on the edges of the simplices $\sigma_j$ such that paired edges have the same image in $X_\gamma$, and non-paired edges all map to $\coprod S^1$.
        We then construct a $2$-dimensional simplicial complex $\Sigma$ by taking a collection $\left\{\Delta^2_j\right\}_j$ of $2$-simplices and gluing them along this pairing.
        The simplicial complex $\Sigma$ thus constructed is a surface with boundary, and the singular simplices $\sigma_j$ define a map $f:\Sigma\rightarrow X_\gamma$ by $f_{\left|\Delta_j^2\right.}=\sigma_j$, with $f\left(\partial\Sigma\right)\subseteq\coprod S_1$.
        After homotoping $f(\Sigma)$ into the image of $X$ in $X_\gamma$, and $f\left(\partial\Sigma\right)$ into $\gamma\left(\coprod S^1\right)$, this induces an admissible surface $f:\left(\Sigma,\partial\Sigma\right)\rightarrow\left(X,\gamma\right)$, and $f_*[\Sigma]=q\alpha$ in $H_2\left(X,\gamma;\mathbb{R}\right)$.
        As above, $-2\chi^-(\Sigma)$ is the simplicial volume $\gromnorm{[\Sigma]}$ of $\Sigma$ \cite{frigerio}*{Corollary 7.5}; on the other hand, our triangulation of $\Sigma$ by the simplices $\Delta^2_j$ gives an upper bound on the simplicial volume:
        \[
            \gromnorm{a_0}=\frac{\normm{qa_0}_1}{q}=\frac{1}{q}\sum_j\left|\varepsilon_j\right|\geq\frac{\gromnorm{[\Sigma]}}{q}=\frac{-2\chi^-(\Sigma)}{q}.
        \]
        By taking the infimum over $a_0$ representing $\alpha$, we obtain the inequality $\left(\geq\right)$.
    \end{proof}
    
    The topological interpretation of $\gromnorm{\cdot}$ connects it to stable commutator length:
    
    \begin{mprop}{A}[Gromov seminorm and $\scl$]\label{prop:scl-gromnorm}
        Let $X$ be a path-connected topological space and let $[c]\in C_1^{\hmtp}\left(\pi_1X;\mathbb{Z}\right)$ be an integral conjugacy class represented by a map $\gamma:\coprod S^1\rightarrow X$. Then
        \[
            \scl_{\pi_1X}\left([c]\right)=\frac{1}{4}\inf\left\{\gromnorm{\alpha}\st{}\alpha\in H_2\left(X,\gamma;\mathbb{Q}\right),\:\partial\alpha=\left[\textstyle\coprod S^1\right]\right\},
        \]
        where $\partial:H_2(X,\gamma;\mathbb{Q})\rightarrow H_1\left(\coprod S^1;\mathbb{Q}\right)$ is the boundary map in the long exact sequence of the pair $\left(X,\gamma\right)$ (see Proposition \ref{prop:les}).
    \end{mprop}
    \begin{proof}
        This follows from the topological interpretations of $\gromnorm{\cdot}$ (Proposition \ref{prop:ell1-gromnorm}) and $\scl$ \cite{cal-scl}*{Proposition 2.74}.
    \end{proof}
    
    We refer to Calegari's book \cite{cal-scl}*{\S{}2.6} for the usual algebraic definition of $\scl$ on chains.
    For our purpose, Proposition \ref{prop:scl-gromnorm} can serve as a definition.
    
    \subsection{Simplicity and incompressibility for admissible surfaces}
    
    We will need admissible surfaces with additional properties:
    
    \begin{defi}
        Let $X$ be a path-connected topological space and $\gamma:\coprod S^1\rightarrow X$.
        We say that an admissible surface $f:\left(\Sigma,\partial\Sigma\right)\rightarrow(X,\gamma)$ is
        \begin{itemize}
            \item \emph{Incompressible} if every simple closed curve in $\Sigma$ with nullhomotopic image in $X$ is nullhomotopic in $\Sigma$,
            \item \emph{Simple} if there are no two boundary components of $\Sigma$ whose images under $f$ represent powers of the same conjugacy class in $\pi_1X$.
        \end{itemize}
    \end{defi}
    
    \begin{lmm}[Simplicity and incompressibility]\label{lmm:simple-incompr-adm-srf}
        Let $X$ be a topological space and $\gamma:\coprod S^1\rightarrow X$.
        Then for every rational class $\alpha\in H_2\left(X,\gamma;\mathbb{Q}\right)$ and for every $\varepsilon>0$, there is a simple, incompressible, admissible surface $f:\left(\Sigma,\partial\Sigma\right)\rightarrow\left(X,\gamma\right)$ such that $f_*[\Sigma]=n(\Sigma)\alpha$ for some $n(\Sigma)\in\mathbb{N}_{\geq1}$, and
        \begin{equation}\label{eq:estim-gromnorm}
            \gromnorm{\alpha}\leq\frac{-2\chi^-(\Sigma)}{n(\Sigma)}\leq\gromnorm{\alpha}+\varepsilon.
        \end{equation}
    \end{lmm}
    \begin{proof}
        Proposition \ref{prop:ell1-gromnorm} implies the existence of an admissible surface $f:\left(\Sigma,\partial\Sigma\right)\rightarrow\left(X,\gamma\right)$ satisfying (\ref{eq:estim-gromnorm}) with $f_*[\Sigma]=n(\Sigma)\alpha$ for some $n(\Sigma)\in\mathbb{N}_{\geq1}$.
        
        If $f$ is not simple, then we can find two boundary components $\partial_1$ and $\partial_2$ of $\Sigma$ whose image under $f$ represent powers of the same conjugacy class in $\pi_1X$.
        Hence we can glue a $1$-handle $H$ between $\partial_1$ and $\partial_2$, with $H$ mapping to a path connecting the respective basepoints of $f\circ\partial_1$ and $f\circ\partial_2$.
        This does not change $f_*[\Sigma]$ but increases $-\chi^-(\Sigma)$ by $1$.
        In order to keep control of $-\chi^-(\Sigma)/n(\Sigma)$, we perform this operation only after replacing $\Sigma$ with a finite cover of large degree $N$ that preserves the number of boundary components.
        Hence, the quantity $-\chi^-(\Sigma)/n(\Sigma)$ will only increase by $1/N$ (this is a simple case of an asymptotic promotion argument, adapted from \cite{cal-scl}*{Proposition 2.10} --- see \cite{chen}*{\S{}4} and \cite{m:isom-embed}*{\S{}4.d} for similar arguments).
        Since this operation decreases the number of boundary components of $\Sigma$ by $1$, we will obtain a simple admissible surface after finitely many iterations.
        
        Now if $f$ is compressible, then there is a simple closed curve $\beta$ in $\Sigma$ which is not nullhomotopic but such that $f\circ\beta$ is.
        In this case, one can cut $\Sigma$ along $\beta$ and glue two discs onto the resulting boundary components; the map $f$ extends onto the new discs since $f\circ\beta$ is assumed to be nullhomotopic.
        This does not change $f_*[\Sigma]$ and makes $-\chi^-(\Sigma)$ decrease, so that (\ref{eq:estim-gromnorm}) still holds, and moreover the property of $f$ being simple is preserved.
        After performing this operation a finite number of times, we therefore obtain that $f$ is simple and incompressible.
    \end{proof}

\section{Bavard duality for the relative Gromov seminorm}\label{sec:bavard}

    Bavard \cite{bavard} proved that the dual space of the $\scl$-seminorm on $C_1^{\hmtp}\left(G;\mathbb{R}\right)$ can be interpreted in terms of quasimorphisms.
    This can be thought of as a kind of $\ell^1$--$\ell^\infty$ duality, and has had a wide range of applications in giving lower bounds for $\scl$ \cites{calegari-fujiwara,bbf,heuer-raags,fft,cfl}.
    We start with some background on classical Bavard duality and bounded cohomology, and then we explain how a result analogous to Bavard's Theorem can be obtained for the relative Gromov seminorm.
    
    \subsection{Bavard duality for scl}
    
    A \emph{quasimorphism} on a group $G$ is a map $\phi:G\rightarrow\mathbb{R}$ such that
    \[
        \sup_{g,h\in G}\left|\phi(gh)-\phi(g)-\phi(h)\right|<\infty.
    \]
    The above supremum is then called the \emph{defect} of $\phi$ and denoted by $D(\phi)$.
    We say in addition that $\phi$ is \emph{homogeneous} if $\phi\left(g^n\right)=n\phi(g)$ for all $g\in G$ and $n\in\mathbb{Z}$.
    
    We denote by $Q(G)$ the $\mathbb{R}$-vector space of homogeneous quasimorphisms on $G$.
    The defect defines a seminorm $D:Q(G)\rightarrow\left[0,\infty\right)$, which vanishes exactly on the subspace $\Hom\left(G,\mathbb{R}\right)\subseteq Q(G)$ consisting of homomorphisms to $\mathbb{R}$.
    In particular, the defect descends to a genuine norm on the quotient $Q(G)/\Hom\left(G,\mathbb{R}\right)$.
    
    If $\phi:G\rightarrow\mathbb{R}$ is a homogeneous quasimorphism, then $\phi$ extends to a $\mathbb{Z}$-linear map $C_1\left(G;\mathbb{Z}\right)\rightarrow\mathbb{R}$.
    The extension satisfies
    \[
        \left|\phi\left(w-w^t\right)\right|=\frac{1}{n}\left|\phi\left(w^n-t^{-1}w^nt\right)\right|\leq\frac{2D(\phi)}{n}\xrightarrow[n\to\infty]{}0
    \]
    for all $w,t\in G$.
    It follows that $\phi$ vanishes on the sub-$\mathbb{Z}$-module $K\left(G;\mathbb{Z}\right)$ of $C_1\left(G;\mathbb{Z}\right)$ spanned by elements of the form $\left(w-w^t\right)$, as in \S{}\ref{subsec:1-chns}.
    Therefore, $\phi$ descends to a $\mathbb{Z}$-linear map $C_1^{\hmtp}\left(G;\mathbb{Z}\right)\rightarrow\mathbb{R}$, which then extends to a $\mathbb{R}$-linear map $C_1^{\hmtp}\left(G;\mathbb{R}\right)\rightarrow\mathbb{R}$.
    
    Classical Bavard duality says that the (semi)normed vector space $\left(Q(G),D\right)$\footnote{Or $Q(G)/\Hom\left(G,\mathbb{R}\right)$, where $D$ defines a honest norm.} is dual to $\left(C_1^{\hmtp}\left(G;\mathbb{R}\right),\scl\right)$:
    
    \begin{thm}[Bavard \cite{bavard}]\label{thm:bavard}
        Let $G$ be a countable group and $[c]\in C_1^{\hmtp}\left(G;\mathbb{R}\right)$.
        Then there is an equality
        \[
            \scl_G\left([c]\right)=\sup\left\{\frac{\phi\left([c]\right)}{2D(\phi)}\st{}\phi\in Q(G)\smallsetminus\Hom\left(G;\mathbb{R}\right)\right\}.
        \]
    \end{thm}
    
    \subsection{Bounded cohomology of topological spaces}\label{subsec:hb-top}
    
    Our analogue of Bavard duality for the relative Gromov seminorm will be based on bounded cohomology, of which we recall the definition here.
    We refer the reader to Frigerio's book \cite{frigerio} for a much more detailed treatment.
    
    Let $X$ be a topological space.
    Recall that the singular cohomology of $X$ (with real coefficients) is the cohomology of the singular cochain complex $C^*_\textnormal{sg}\left(X;\mathbb{R}\right)$ given by
    \[
        C^n_\textnormal{sg}\left(X;\mathbb{R}\right)\coloneqq\Hom_R\left(C_n^\textnormal{sg}\left(X;\mathbb{R}\right),R\right).
    \]
    Since the $n$-th chain group $C_n^\textnormal{sg}\left(X;\mathbb{R}\right)$ is the free $R$-module on the set $\mathcal{S}_n$ of singular $n$-simplices $\sigma:\Delta^n\rightarrow X$, the $n$-th cochain group $C^n_\textnormal{sg}\left(X;\mathbb{R}\right)$ can equivalently be defined as the set of all maps $\mathcal{S}_n\rightarrow \mathbb{R}$.
    The \emph{$\ell^\infty$-norm} of a cochain $\psi\in C^n_\textnormal{sg}\left(X;\mathbb{R}\right)=\mathbb{R}^{\mathcal{S}_n}$ is
    \[
        \inftynorm{\psi}\coloneqq\sup_{\sigma\in\mathcal{S}_n}\left|\psi(\sigma)\right|\in\left[0,+\infty\right].
    \]
    
    Now the \emph{bounded cochain complex} of $X$ with coefficients in $R$ is the sub-cochain complex $C^*_{b}\left(X;\mathbb{R}\right)$ of $C^*_\textnormal{sg}\left(X;\mathbb{R}\right)$ consisting of all \emph{bounded} maps $\mathcal{S}_n\rightarrow\mathbb{R}$:
    \[
        C^n_{b}\left(X;\mathbb{R}\right)\coloneqq\left\{\psi\in C^n_\textnormal{sg}\left(X;\mathbb{R}\right)\st{}\inftynorm{\psi}<\infty\right\}\subseteq C^n_\textnormal{sg}\left(X;\mathbb{R}\right),
    \]
    with coboundary induced by that of $C^*_\textnormal{sg}\left(X;\mathbb{R}\right)$.
    The \emph{bounded cohomology} of $X$ is the cohomology of this cochain complex:
    \[
        H^*_{b}\left(X;\mathbb{R}\right)\coloneqq H^*\left(C^*_{b}\left(X;\mathbb{R}\right)\right).
    \]
    
    The $\ell^\infty$-norm descends to a seminorm --- still denoted $\inftynorm{\cdot}$ --- on $H^*_b\left(X;\mathbb{R}\right)$.
    
    For us, bounded cohomology will always be understood to be with real coefficients, and we might omit $\mathbb{R}$ from the notation.
    
    It turns out that $\inftynorm{\cdot}$ defines a genuine norm in degree $2$ if $X$ is a countable CW-complex:
    
    \begin{thm}[Matsumoto--Morita--Ivanov \cites{mm,ivanov:h2b}]\label{thm:mmi}
        Let $X$ be a countable CW-complex.
        Then $\inftynorm{u}>0$ for every $u\in H^2_b\left(X;\mathbb{R}\right)\smallsetminus\{0\}$.
    \end{thm}
    
    Duality between the $\ell^\infty$-norm on bounded cohomology and the $\ell^1$-norm on singular homology plays a central role in this paper.
    It comes from the natural pairing
    \[
        \left<-,-\right>:C^*_b\left(X;\mathbb{R}\right)\times C_*^\textnormal{sg}\left(X;\mathbb{R}\right)\rightarrow R
    \]
    which is the restriction of the duality pairing $C^*_\textnormal{sg}\left(X;\mathbb{R}\right)\times C_*^\textnormal{sg}\left(X;\mathbb{R}\right)\rightarrow R$ given by $\left<\psi,c\right>\coloneqq\psi(c)$ for $\psi\in C^n_\textnormal{sg}\left(X;\mathbb{R}\right)$ and $c\in C_n^\textnormal{sg}\left(X;\mathbb{R}\right)$.
    This descends to a pairing
    \[
        \left<-,-\right>:H^*_b\left(X;\mathbb{R}\right)\times H_*\left(X;\mathbb{R}\right)\rightarrow\mathbb{R},
    \]
    which is called the \emph{Kronecker product}.
    The Banach space $H_b^n\left(X;\mathbb{R}\right)$ is dual to the seminormed space $H_n\left(X;\mathbb{R}\right)$ under this pairing:
    
    \begin{prop}[$\ell^1$--$\ell^\infty$ duality in bounded cohomology \cite{frigerio}*{Lemma 6.1}]
        Let $X$ be a topological space and $\alpha\in H_n\left(X;\mathbb{R}\right)$.
        Then the $\ell^1$-seminorm of $\alpha$ satisfies
        \[
            \normm{\alpha}_1=\sup\left\{\frac{\left<u,\alpha\right>}{\inftynorm{u}}\st{}u\in H^n_b(X;\mathbb{R}),\:\inftynorm{u}>0\right\}.
        \]
    \end{prop}
    
    \subsection{Bounded cohomology of groups}\label{subsec:hb-gp}
    
    It follows from Gromov's Mapping Theorem \cite{frigerio}*{Corollary 5.11} that, for any continuous map $f:X\rightarrow Y$ between countable path-connected CW-complexes inducing an isomorphism on fundamental groups, the induced map $f^*:H^*_b(Y;\mathbb{R})\rightarrow H^*_b(X;\mathbb{R})$ is an isometric isomorphism (\emph{isometric} means that it preserves the $\ell^\infty$-seminorm).
    
    Hence, given a countable group $G$, one can define the \emph{bounded cohomology} of $G$ to be the bounded cohomology of any countable path-connected CW-complex $X$ with $\pi_1X=G$:
    \[
        H^*_b(G;\mathbb{R})\coloneqq H^*_b(X;\mathbb{R}).
    \]
    Such a space $X$ always exists --- for instance, one can take $X$ to be a (potentially infinite) presentation complex of $G$.
    Since there are isometric isomorphisms $H^*_b(X;\mathbb{R})\cong H^*_b\left(X';\mathbb{R}\right)$ for any two choices of $X,X'$ as above, there is a well-defined $\ell^\infty$-seminorm on $H^*_b(G;\mathbb{R})$.
    
    But the bounded cohomology of a group can be given a more algebraic interpretation as follows.
    The \emph{bar cochain complex} of $G$ with real coefficients is defined by
    \[
        C^n\left(G;\mathbb{R}\right)\coloneqq\mathbb{R}^{G^n},
    \]
    where $\mathbb{R}^{G^n}$ is the space of all maps $G^n\rightarrow\mathbb{R}$. Coboundary maps $d^n:C^{n-1}\left(G;\mathbb{R}\right)\rightarrow C^{n}\left(G;\mathbb{R}\right)$ are given by
    \begin{multline*}
        d^n\psi\left(g_1,\dots,g_n\right)\coloneqq\psi\left(g_2,\dots, g_n\right)-\psi\left(g_1g_2, g_3,\dots, g_n\right)+\psi\left(g_1, g_2g_3,\dots, g_n\right)\\
        -\cdots+(-1)^{n-1}\psi\left(g_1,\dots, g_{n-2}, g_{n-1}g_n\right)+(-1)^n\psi\left(g_1,\dots, g_{n-1}\right).
    \end{multline*}
    This is the dual of the chain complex $C_*\left(G;\mathbb{R}\right)$ introduced in \S{}\ref{subsec:1-chns}.
    
    Given a cochain $\psi\in C^n\left(G;\mathbb{R}\right)$, its $\ell^\infty$-norm is
    \[
        \inftynorm{\psi}\coloneqq\sup_{\left(g_1,\dots,g_n\right)\in G^n}\left|\psi\left(g_1,\dots,g_n\right)\right|\in\left[0,+\infty\right].
    \]
    The \emph{bounded cochain complex} of $G$ is
    \[
        C^*_b\left(G;\mathbb{R}\right)\coloneqq\left\{\psi\in C^n\left(G;\mathbb{R}\right)\st{}\inftynorm{\psi}<\infty\right\}.
    \]
    It turns out that the bounded cohomology of $G$ can be defined as the cohomology of $C^*_b\left(G;\mathbb{R}\right)$, and we now explain how to write an explicit isomorphism between this cohomology and the bounded cohomology of a space $X$ with fundamental group $G$ as defined in \S{}\ref{subsec:hb-top}.
    
    Let $X$ be a countable path-connected CW-complex with a fixed basepoint $\omega$ and such that $\pi_1\left(X,\omega\right)=G$.
    Each element $g$ of $G$ can be represented by a loop $\gamma_g:S^1\rightarrow X$ based at $\omega$, which can also be described as a map $\sigma_g:\Delta^1\rightarrow X$, where $\Delta^1$ is a $1$-simplex (i.e.\ a segment), and $\sigma_g$ maps both endpoints of $\Delta^1$ to $\omega$.
    For all $g_1,g_2\in G$, the concatenation $\sigma_{g_1}\cdot\sigma_{g_2}$ is homotopic (with fixed endpoints) to $\sigma_{g_1g_2}$, and one can construct a map $\sigma_{g_1,g_2}:\Delta^2\rightarrow X$ (where $\Delta^2$ is a $2$-simplex) such that the restrictions of $\sigma_{g_1,g_2}$ to its three faces are $\sigma_{g_2}$, $\sigma_{g_1g_2}^{-1}$, and $\sigma_{g_1}$ (where $\sigma^{-1}$ is the singular simplex $\sigma$ with reversed orientation), as in Figure \ref{fig:constr-che}.
        \begin{figure}[htb]
            \centering
            \begin{tikzpicture}[every node/.style={rectangle,draw=none,fill=none},scale=1.2]
                \foreach \i in {0,1,2}{\coordinate (c\i) at (\i*360/3:4.5em);}
                
                \draw [draw=none,pattern={Lines[angle=45,distance=7pt]},pattern color=gray] (c0) -- (c1) -- (c2) -- cycle;
                
                \foreach \i in {0,1,2}{\node [draw,circle,fill,inner sep=1pt] (x\i) at (c\i) {};}
                
                \draw [thick] (x0) -- (x1) coordinate [midway] (mid);
                \draw [thick,->] (x1) -- (mid) node [above right] {${\sigma_{g_1g_2}}$};
                
                \draw [thick] (x1) -- (x2) coordinate [midway] (mid);
                \draw [thick,->] (x1) -- (mid) node [left] {${\sigma_{g_1}}$};
                
                \draw [thick] (x2) -- (x0) coordinate [midway] (mid);
                \draw [thick,->] (x2) -- (mid) node [below right] {${\sigma_{g_2}}$};
                
                \draw node [rectangle,fill=white,draw=none,inner sep=1pt,rounded corners=2pt] at (-.1,0) {$\sigma_{g_1,g_2}$};
                
                \draw node at (.7,0) {\Huge$\circlearrowleft$};
            \end{tikzpicture}
            \caption{Construction of the chain map $h_*:C_*\left(G;\mathbb{R}\right)\rightarrow C_*^\textnormal{sg}\left(X;\mathbb{R}\right)$.}
            \label{fig:constr-che}
        \end{figure}
    We can iterate this construction and choose, for each $n$-tuple $\left(g_1,\dots,g_n\right)\in G^n$, a singular simplex $\sigma_{g_1,\dots,g_n}$ whose restriction to its $i$-th face is $\sigma_{g_1,\dots,g_ig_{i+1},\dots,g_n}^{\varepsilon_i}$ (respectively $\sigma_{g_2,\dots,g_n}^{\varepsilon_0}$ for $i=0$ and $\sigma_{g_1,\dots,g_{n-1}}^{\varepsilon_n}$ for $i=n$), with $\varepsilon_i=(-1)^i$.
    
    If $X$ is a $K(G,1)$ space, then the map $h:\left(g_1,\dots,g_n\right)\mapsto\sigma_{g_1,\dots,g_n}$ induces a chain homotopy equivalence $h_*:C_*\left(G;\mathbb{R}\right)\xrightarrow{\sim}C_*^\textnormal{sg}\left(X;\mathbb{R}\right)$ \cite{brown}*{\S{}I.4} and therefore a cochain homotopy equivalence
    \[
        h^*:C^*_\textnormal{sg}\left(X;\mathbb{R}\right)\xrightarrow{\sim}C^*\left(G;\mathbb{R}\right),
    \]
    which induces an isomorphism
    \[
        H^*\left(C^*\left(G;\mathbb{R}\right)\right)\cong H^*\left(X;\mathbb{R}\right).
    \]
    The image under $h^*$ of the bounded cochain complex $C^*_b\left(X;\mathbb{R}\right)$ is $C^*_b\left(G;\mathbb{R}\right)$.
    It follows that $h$ also induces a cochain homotopy equivalence $h^*:C^*_b\left(X;\mathbb{R}\right)\xrightarrow{\sim}C^*_b\left(G;\mathbb{R}\right)$, inducing an isomorphism
    \[
        H^*\left(C^*_b\left(G;\mathbb{R}\right)\right)\cong H^*_b\left(X;\mathbb{R}\right),
    \]
    which is in fact an isometric isomorphism \cite{frigerio}*{Theorem 5.9}.
    We will denote this cohomology by $H^*_b\left(G;\mathbb{R}\right)$ and interpret it using both points of view.
    
    \begin{rmk}\label{rk:qm-h2b}
        There is a connection between quasimorphisms and bounded cohomology: a quasimorphism $\phi:G\rightarrow\mathbb{R}$ can be seen as an element of $C^1\left(G;\mathbb{R}\right)$, and its coboundary $d^2\phi$ is given by
        \[
            d^2\phi\left(g,h\right)=\phi(g)-\phi(gh)+\phi(h).
        \]
        Hence, the quasimorphism condition means exactly that $d^2\phi$ is a \emph{bounded} cochain, and in fact a bounded cocycle.
        Therefore, it defines a class $\left[d^2\phi\right]\in H^2_b\left(G;\mathbb{R}\right)$.
        This gives a morphism $\left[d^2-\right]:Q(G)\rightarrow H^2_b\left(G;\mathbb{R}\right)$ whose kernel is the subspace $\Hom\left(G,\mathbb{R}\right)$ of $Q(G)$ consisting of homomorphisms to $\mathbb{R}$.
        In fact, this extends to an exact sequence \cite{cal-scl}*{Theorem 2.50}
        \[
            0\rightarrow\Hom\left(G,\mathbb{R}\right)\rightarrow Q(G)\xrightarrow{\left[d^2-\right]}H^2_b\left(G;\mathbb{R}\right)\rightarrow H^2\left(G;\mathbb{R}\right),
        \]
        where $H^2_b\left(G;\mathbb{R}\right)\rightarrow H^2\left(G;\mathbb{R}\right)$ is the map induced by the inclusion $C^*_b\left(G;\mathbb{R}\right)\hookrightarrow C^*\left(G;\mathbb{R}\right)$.
    \end{rmk}
    
    \subsection{Bavard duality for the relative Gromov seminorm}
    
    Our aim is now to use bounded cohomology in order to obtain a statement analogous to Bavard duality (Theorem \ref{thm:bavard}) for the relative Gromov seminorm on $H_2\left(X,\gamma\right)$, where $\gamma:\coprod S^1\rightarrow X$ is a collection of loops in a path-connected topological space $X$.
    
    There is a notion of \emph{relative bounded cohomology}, giving rise to a long exact sequence as in the case of singular cohomology --- we refer the reader to Frigerio's book \cite{frigerio}*{\S{}5.7} for a definition.
    The duality principle \cite{frigerio}*{Lemma 6.1} then implies that $H^2_b\left(X,\gamma\right)\coloneqq H^2_b\left(X_\gamma,\coprod S^1\right)$ (with the $\ell^\infty$-seminorm) is the dual of $H_2\left(X,\gamma\right)$ (with the $\ell^1$-seminorm).
    But since $\pi_1S^1=\mathbb{Z}$ is amenable, $H^*_b\left(\coprod S^1\right)$ vanishes \cite{frigerio}*{Theorem 3.6} and the long exact sequence of $\left(X_\gamma,\coprod S^1\right)$ gives a natural isomorphism
    \[
        H^2_b\left(X,\gamma;\mathbb{R}\right)\cong H^2_b\left(X;\mathbb{R}\right).
    \]
    This isomorphism, together with the Kronecker product $H^2_b\left(X,\gamma\right)\times H_2\left(X,\gamma\right)\rightarrow\mathbb{R}$ (which is the relative analogue of the absolute Kronecker product introduced in \S{}\ref{subsec:hb-top}) defines a pairing
    \[
        \left<\cdot,\cdot\right>:H^2_b\left(X;\mathbb{R}\right)\times H_2\left(X,\gamma;\mathbb{R}\right)\rightarrow\mathbb{R}.
    \]
    
    It turns out that the $\ell^\infty$-seminorm (which is a norm in degree $2$ by Theorem \ref{thm:mmi}) is dual to the $\ell^1$-seminorm under this pairing:
    
    \begin{mthm}{B}[Bavard duality for the relative Gromov seminorm]\label{thm:bav-dual-rel-grom}
        Let $X$ be a countable CW-complex and $\gamma:\coprod S^1\rightarrow X$.
        Given a real class $\alpha\in H_2(X,\gamma;\mathbb{R})$, the relative Gromov seminorm of $\alpha$ is given by
        \[
            \gromnorm{\alpha}=\sup\left\{\frac{\left<u,\alpha\right>}{\inftynorm{u}}\st{}u\in H^2_b\left(X;\mathbb{R};\mathbb{R}\right)\smallsetminus\{0\}\right\}.
        \]
    \end{mthm}
    \begin{proof}
        Duality between $\ell^1$-homology and bounded cohomology \cite{frigerio}*{Lemma 6.1} implies that
        \[
            \gromnorm{\alpha}=\sup\left\{\frac{\left<u,\alpha\right>}{\inftynorm{u}}\st{}u\in H^2_b\left(X,\gamma;\mathbb{R}\right),\:\inftynorm{u}\neq0\right\}.
        \]
        In addition, a result proved independently by Bucher et al.\ \cite{bbfplus}*{Theorem 1.2} and by Kim and Kuessner \cite{kim-kuessner}*{Theorem 1.2} implies that the isomorphism $H^2_b\left(X,\gamma\right)\cong H^2_b\left(X\right)$ is isometric for the $\ell^\infty$-norm.
        Together with the fact that $\inftynorm{u}=0$ only if $u=0$ in $H^2_b(X)$ (Theorem \ref{thm:mmi}), this implies the result.
    \end{proof}
    
    We'll say that a class $u\in H^2_b\left(X;\mathbb{R}\right)$ is \emph{extremal} for $\alpha\in H_2\left(X,\gamma;\mathbb{R}\right)$ if it realises the supremum in Theorem \ref{thm:bav-dual-rel-grom}.
    Note that extremal classes exist for all $\alpha\in H_2\left(X,\gamma;\mathbb{R}\right)$ by the Hahn--Banach Theorem.
    
\section{An application to scl in graphs of groups}\label{sec:hnn}
    
    The aim of this section is to give an example in the context of graphs of groups where Theorem \ref{thm:bav-dual-rel-grom} can be used to understand the relative Gromov seminorm, which then yields computations of stable commutator length.
    
    \subsection{Failure of isometric embedding for scl}
    
    One of the fundamental facts in the theory of graphs of groups is that a vertex group embeds into the fundamental group of the graph of groups.
    It is natural at first to try to make this inclusion map $\scl$-preserving, but that unfortunately does not work in general, even if edge groups are amenable (and hence have vanishing stable commutator length \cite{cal-scl}*{Proposition 2.65}):
        
    \begin{ex}\label{ex:hnn-non-isom}
        Let $S$ be a closed genus-$3$ surface, and let $\beta$ be a non-separating simple closed curve in $S$.
        Consider the HNN-splitting $\pi_1S=\pi_1T\ast_\mathbb{Z}$ obtained by cutting $S$ along $\beta$, where $T$ is a genus-$2$ surface with two boundary components, and the HNN-extension identifies the two boundary components of $T$ --- see Figure \ref{fig:ex:hnn-non-isom}.
        \begin{figure}[htb]
            \centering
            \begin{subfigure}[c]{15em}
                \centering
                \begin{tikzpicture}[every node/.style={rectangle,draw=none,fill=none},xscale=.4, yscale=.53]                
                    \draw [thick] (6.25,2) to (0,2) (0,2) to[bend right=45] (-2,0) (-2,0) to[bend right=45] (0,-2) (0,-2) to (6.25,-2);
                    \draw [thick] (6.25,-2) to (8,-2) (8,-2) to[bend right=45] (10,0) (10,0) to[bend right=45] (8,2) (8,2) to (6.25,2);
                    
                    \draw [draw=none,fill=white] (0,0) ellipse[x radius=.73cm,y radius=.5cm];
                    \draw [thick] (-0.73,0) to[bend left=37] (0,.5) (0,.5) to[bend left=37] (0.73,0)
                        (-0.8,.3) to[bend right=45] (0,-.5) (0,-.5) to[bend right=45] (0.8,.3);
                    
                    \draw [xshift=4cm,draw=none,fill=white] (0,0) ellipse[x radius=.73cm,y radius=.5cm];
                    \draw [xshift=4cm,thick] (-0.73,0) to[bend left=37] (0,.5) (0,.5) to[bend left=37] (0.73,0)
                        (-0.8,.3) to[bend right=45] (0,-.5) (0,-.5) to[bend right=45] (0.8,.3);
                    
                    \draw [xshift=8cm,thick] (-0.73,0) to[bend left=37] (0,.5) (0,.5) to[bend left=37] (0.73,0)
                        (-0.8,.3) to[bend right=45] (0,-.5) (0,-.5) to[bend right=45] (0.8,.3);
                    
                    \draw [ultra thick,Plum] (5.75,-2) to[bend left=20] (5.75,2);
                    \draw [ultra thick,Plum,dashed] (5.75,-2) to[bend right=20] (5.75,2);
                    \draw [ultra thick,Plum,->] (5.35,-.001) -- (5.35,.001);
                    
                    \draw [ultra thick,NavyBlue] (8.73,0) to[bend left=20] (10,0);
                    \draw [ultra thick,NavyBlue,dashed] (8.73,0) to[bend right=20] (10,0);
                    
                    \draw node at (2,0) {\large$S$};
                    \draw node [Plum] at (5,.5) {$\gamma$};
                    \draw node [NavyBlue] at (9.4,.45) {$\beta$};
                \end{tikzpicture}
            \end{subfigure}
            {\LARGE $\phantom{a}=\phantom{a}$}
            \begin{subfigure}[c]{15em}
                \centering
                \begin{tikzpicture}[every node/.style={rectangle,draw=none,fill=none},xscale=.4, yscale=.53]                
                    \draw [thick] (6.25,2) to (0,2) (0,2) to[bend right=45] (-2,0) (-2,0) to[bend right=45] (0,-2) (0,-2) to (6.25,-2);
                    \draw [thick] (6.25,-2) to (8,-2) (8,-2)
                        (8,2) (8,2) to (6.25,2);
                    
                    \draw [draw=none,fill=white] (0,0) ellipse[x radius=.73cm,y radius=.5cm];
                    \draw [thick] (-0.73,0) to[bend left=37] (0,.5) (0,.5) to[bend left=37] (0.73,0)
                        (-0.8,.3) to[bend right=45] (0,-.5) (0,-.5) to[bend right=45] (0.8,.3);
                    
                    \draw [xshift=4cm,draw=none,fill=white] (0,0) ellipse[x radius=.73cm,y radius=.5cm];
                    \draw [xshift=4cm,thick] (-0.73,0) to[bend left=37] (0,.5) (0,.5) to[bend left=37] (0.73,0)
                        (-0.8,.3) to[bend right=45] (0,-.5) (0,-.5) to[bend right=45] (0.8,.3);
                    
                    \draw [xshift=8cm,thick] (-0.65,0) to[bend left=45] (0,.65)
                        (-0.65,0) to[bend right=45] (0,-.65);
                        
                    \draw [NavyBlue,thick] (8,2) to [bend right=25] (8,.65)
                        (8,2) to [bend left=25] (8,.65)
                        (8,-2) to [bend right=25] (8,-.65)
                        (8,-2) to [bend left=25] (8,-.65);
                    
                    \draw [ultra thick,Plum] (5.75,-2) to[bend left=20] (5.75,2);
                    \draw [ultra thick,Plum,dashed] (5.75,-2) to[bend right=20] (5.75,2);
                    \draw [ultra thick,Plum,->] (5.35,-.001) -- (5.35,.001);
                    
                    \draw node [draw,thick,NavyBlue,circle,fill=none,inner sep=7pt] at (10,0) {};
                    
                    \draw [thick,NavyBlue,<->] (8.5,1.25) to [bend left=45] (10,.85);
                    \draw [thick,NavyBlue,<->] (8.5,-1.25) to [bend right=45] (10,-.85);
                    
                    \draw node at (2,0) {\large$T$};
                    \draw node [Plum] at (5,.5) {$\gamma$};
                \end{tikzpicture}
            \end{subfigure}
            \caption{HNN-splitting of a closed genus-$3$ surface.}
            \label{fig:ex:hnn-non-isom}
        \end{figure}        
        Then the embedding $\pi_1T\hookrightarrow\pi_1S$ is \emph{not} $\scl$-preserving.
        To see this, consider the loop $\gamma$ represented in the picture.
        Note that $T$ is a surface with non-empty boundary, and $\gamma$ bounds an immersed (and in fact embedded) genus-$2$ surface with one boundary component in $T$, so a result of Calegari \cite{cal-scl}*{Lemma 4.62} (which also follows from Theorem \ref{thm:eub-extr} below) implies that $\scl_{\pi_1T}\left([\gamma]\right)=\frac{3}{2}$.
        In $S$ however, $\gamma$ bounds a genus-$1$ surface with one boundary component, so $\scl_{\pi_1S}\left([\gamma]\right)\leq\frac{1}{2}$.
        This shows that the morphism
        \[
            \pi_1T\hookrightarrow\pi_1T\ast_\mathbb{Z}
        \]
        is not $\scl$-preserving
    \end{ex}
    
    \subsection{Isometric embedding for the relative Gromov seminorm}
    
    Example \ref{ex:hnn-non-isom} shows that the inclusion map of a vertex group in a graph of groups can fail to be $\scl$-preserving.
    However, using Theorem \ref{thm:bav-dual-rel-grom}, we are able to translate an isometric embedding result of Bucher et al.\ \cite{bbfplus} in bounded cohomology into the fact that the inclusion map preserves the relative Gromov seminorm if edge groups are amenable:
    
    \begin{mthm}{C}[$\ell^1$-isometric embedding of vertex groups in graphs of groups]\label{thm:isom-grph-gps}
        Let $\mathcal{G}$ be a graph of groups whose underlying graph $\Gamma$ is finite, with countable vertex groups $\left\{G_v\right\}_{v\in V(\Gamma)}$, and amenable edge groups $\left\{G_e\right\}_{e\in E(\Gamma)}$.
        Fix a vertex $v$ and consider the inclusion map $i_v:G_v\hookrightarrow\pi_1\mathcal{G}$.
        Then for each class $[c]\in C_1^{\hmtp}\left(G_v;\mathbb{Z}\right)$, the embedding
        \[
            {i_v}_*:H_2\left(G_v,[c];\mathbb{R}\right)\hookrightarrow H_2\left(\pi_1\mathcal{G},\left[i_v(c)\right];\mathbb{R}\right).
        \]
        is isometric for $\gromnorm{\cdot}$.
    \end{mthm}
    \begin{proof}
        By \cite{bbfplus}*{Theorem 1.1}, there is an isometric embedding
        \[
            \textstyle\Theta:\bigoplus_v H^2_b\left(G_v;\mathbb{R}\right)\hookrightarrow H^2_b\left(\pi_1\mathcal{G};\mathbb{R}\right),
        \]
        which is a right inverse to
        \[
            \textstyle\bigoplus_vi_v^*:H^2_b\left(\pi_1\mathcal{G};\mathbb{R}\right)\twoheadrightarrow\bigoplus_v H^2_b\left(G_v;\mathbb{R}\right).
        \]
        Now let $[c]\in C_1^{\hmtp}\left(G_v;\mathbb{Z}\right)$ and $\alpha\in H_2\left(G_v,[c];\mathbb{R}\right)$.
        Bavard duality for $\gromnorm{\cdot}$ (Theorem \ref{thm:bav-dual-rel-grom}) implies that
        \[
            \gromnorm{\alpha}=\sup\left\{\frac{\left<u,\alpha\right>}{\inftynorm{u}}\st{}u\in H^2_b\left(G_v;\mathbb{R}\right)\right\}.
        \]
        Since $i_v^*\Theta u=u$ for all $u\in H^2_b\left(G_v;\mathbb{R}\right)$, and since $\Theta$ preserves $\inftynorm{\cdot}$, it follows that
        \begin{align*}
            \gromnorm{\alpha}
            &=\sup\left\{\frac{\left<i_v^*\Theta u,\alpha\right>}{\inftynorm{\Theta u}}\st{}u\in H^2_b\left(G_v;\mathbb{R}\right)\right\}
            =\sup\left\{\frac{\left<\Theta u,{i_v}_*\alpha\right>}{\inftynorm{\Theta u}}\st{}u\in H^2_b\left(G_v;\mathbb{R}\right)\right\}\\
            &\leq\sup\left\{\frac{\left<u',{i_v}_*\alpha\right>}{\inftynorm{u'}}\st{}u'\in H^2_b\left(\pi_1\mathcal{G};\mathbb{R}\right)\right\}
            =\gromnorm{{i_v}_*\alpha}.
        \end{align*}
        This proves that $\gromnorm{\alpha}\leq\gromnorm{{i_v}_*\alpha}$, and the reverse inequality follows from the fact that group homomorphisms are $\gromnorm{\cdot}$-non-increasing.
    \end{proof}
    
    With an extra homological condition, we can deduce an isometric embedding result for $\scl$:
    
    \begin{cor}[$\scl$-isometric embedding of vertex groups in graphs of groups]\label{cor:scl-isom-grph-gps}
        Let $\mathcal{G}$ be a graph of groups whose underlying graph $\Gamma$ is finite, with countable vertex groups $\left\{G_v\right\}_{v\in V(\Gamma)}$, and amenable edge groups $\left\{G_e\right\}_{e\in E(\Gamma)}$.
        Fix a vertex $v$ and assume that the inclusion-induced map ${i_v}_*:H_2\left(G_v;\mathbb{Q}\right)\rightarrow H_2\left(\pi_1\mathcal{G};\mathbb{Q}\right)$ is surjective.
        Then for every $[c]\in B_1^{\hmtp}\left(G_v;\mathbb{Z}\right)$,
        \[
            \scl_{\pi_1\mathcal{G}}\left(\left[{i_v}(c)\right]\right)=\scl_{G_v}\left([c]\right).
        \]
    \end{cor}
    \begin{proof}
        Fix a $K\left(G_v,1\right)$ space $X_v$ for each vertex $v$ and a $K\left(G_e,1\right)$ space $X_e$ for each edge $e$, and form the corresponding graph of spaces $\mathcal{X}$, which is a $K\left(\pi_1\mathcal{G},1\right)$.
        Let $j_v:X_v\hookrightarrow\mathcal{X}$ be the inclusion map, so that $i_v={j_v}_*:G_v\rightarrow\pi_1\mathcal{G}$.
        
        Given a map $\gamma:\coprod S^1\rightarrow X_v$ representing a conjugacy class $[c]\in B_1^{\hmtp}\left(G_v;\mathbb{Z}\right)$, the map $\gamma_*:H_1\left(\coprod S^1\right)\rightarrow H_1\left(X_v\right)$ vanishes, and Proposition \ref{prop:les} gives a commutative diagram with exact rows (with omitted $\mathbb{Q}$-coefficients):
            \[
                \vcenter{\hbox{\begin{tikzpicture}[every node/.style={draw=none,fill=none,rectangle},xscale=1.1]
                    \node (Z) at (-1.5,-1.25) {$0$};
                    \node (Zp) at (-1.5,0) {$0$};
                    \node (A) at (0,-1.25) {$H_2\left(\mathcal{X}\right)$};
                    \node (Ap) at (0,0) {$H_2\left(X_v\right)$};
                    \node (B) at (2.25,-1.25) {$H_2\left(\mathcal{X},j_v\gamma\right)$};
                    \node (Bp) at (2.25,0) {$H_2\left(X_v,\gamma\right)$};
                    \node (C) at (4.5,-1.25) {$H_1\left(\coprod S_1\right)$};
                    \node (Cp) at (4.5,0) {$H_1\left(\coprod S^1\right)$};
                    \node (D) at (6,-1.25) {$0$};
                    \node (Dp) at (6,0) {$0$};
                    \node [rotate=90,xscale=2,yscale=1.2] at (4.5,-.625) {$=$};
                    
                    \draw [<-] (A) -> (Ap) node [midway,right] {${j_v}_*$};
                    \draw [<-] (B) -> (Bp) node [midway,right] {${j_v}_*$};
                    \draw [->] (Z) -> (A);
                    \draw [->] (Zp) -> (Ap);
                    \draw [->] (A) -> (B);
                    \draw [->] (Ap) -> (Bp);
                    \draw [->] (B) -> (C) node [midway,above] {$\partial$};
                    \draw [->] (Bp) -> (Cp) node [midway,above] {$\partial$};
                    \draw [->] (C) -> (D);
                    \draw [->] (Cp) -> (Dp);
                \end{tikzpicture}}}
            \]
        Now the map ${j_v}_*:H_2\left(X_v\right)\rightarrow H_2\left(\mathcal{X}\right)$ is surjective by assumption, so the Five Lemma implies that ${j_v}_*:H_2\left(X_v,\gamma\right)\rightarrow H_2\left(\mathcal{X},j_v\gamma\right)$ is surjective.
        Hence, given a class $\beta\in H_2\left(\mathcal{X},j_v\gamma;\mathbb{Q}\right)$ with $\partial\beta=\left[\coprod S^1\right]$, there exists $\alpha\in H_2\left(X_v,\gamma;\mathbb{Q}\right)$ with ${j_v}_*\alpha=\beta$.
        Since the diagram commutes, we have $\partial\alpha=\partial\beta=\left[\coprod S^1\right]$.
        Therefore, Proposition \ref{prop:scl-gromnorm} and Theorem \ref{thm:isom-grph-gps} yield
        \[
            \scl_{G_v}\left([c]\right)\leq\frac{1}{4}\gromnorm{\alpha}=\frac{1}{4}\gromnorm{{j_v}_*\alpha}=\frac{1}{4}\gromnorm{\beta}.
        \]
        Taking the infimum over $\beta$ gives $\scl_{G_v}\left([c]\right)\leq\scl_{\pi_1\mathcal{G}}\left(\left[i_v(c)\right]\right)$.
        The reverse inequality follows from the general fact that group homomorphisms are $\scl$-non-increasing.
    \end{proof}
    
    Note that Theorem \ref{thm:isom-grph-gps} and Corollary \ref{cor:scl-isom-grph-gps} recover the main theorems of \cite{m:isom-embed} on isometric embeddings of surfaces for the relative Gromov seminorm: indeed, given an oriented, compact, connected surface $S$ and a $\pi_1$-injective subsurface $T$, the fundamental group $\pi_1S$ splits as a graph of groups, with vertex groups given by $\pi_1T$ and the fundamental groups of all connected components of $S\smallsetminus T$, and with edge groups isomorphic to $\mathbb{Z}$ --- and corresponding to cutting $S$ along simple closed curves.
    
    \subsection{Spectral gaps in HNN-extensions}\label{subsec:spc-gap-hnn}
    
    Corollary \ref{cor:scl-isom-grph-gps} can be used for instance to estimate the spectral gap for $\scl$ of certain HNN-extensions:
        
    \begin{ex}[Dyck's surface]\label{ex:dyck}
        Let $\varDelta$ be the non-orientable surface given by the side-pairing of Figure \ref{fig:dyck-pairing}, so that $\pi_1\varDelta=\left<a,b,c\st{}\left[a,b\right]=c^2\right>$.
        \begin{figure}[htb]
                \centering
                \begin{tikzpicture}[every node/.style={rectangle,draw=none,fill=none},scale=.65]
                    \foreach \i in {0,1,...,5}{\coordinate (c\i) at (\i*360/6:4.5em);}
                    
                    \draw [draw=none,pattern={Lines[angle=45,distance=7pt]},pattern color=gray] (c0) -- (c1) -- (c2) -- (c3) -- (c4) -- (c5) -- cycle;
                    
                    \foreach \i in {0,1,...,5}{\node [draw,circle,fill,inner sep=1pt] (x\i) at (c\i) {};}
                    
                    \draw [red,thick] (x0) -- (x1) coordinate [midway] (mid);
                    \draw [red,thick,->] (x0) -- (mid) node [above right] {$a$};
                    
                    \draw [ForestGreen,thick] (x1) -- (x2) coordinate [midway] (mid);
                    \draw [ForestGreen,thick,->] (x1) -- (mid) node [above] {$b$};
                    
                    \draw [red,thick] (x2) -- (x3) coordinate [midway] (mid);
                    \draw [red,thick,->] (x3) -- (mid) node [above left] {$a$};
                    
                    \draw [ForestGreen,thick] (x3) -- (x4) coordinate [midway] (mid);
                    \draw [ForestGreen,thick,->] (x4) -- (mid) node [below left] {$b$};
                    
                    \draw [NavyBlue,thick] (x4) -- (x5) coordinate [midway] (mid);
                    \draw [NavyBlue,thick,->] (x5) -- (mid) node [below] {$c$};
                    
                    \draw [NavyBlue,thick] (x5) -- (x0) coordinate [midway] (mid);
                    \draw [NavyBlue,thick,->] (x0) -- (mid) node [below right] {$c$};
                \end{tikzpicture}
            \caption{Side-pairing for Dyck's surface $\varDelta$.}
            \label{fig:dyck-pairing}
        \end{figure}
        Dyck's Theorem \cite{dyck} asserts that $\varDelta\cong\mathbb{RP}^2\#\mathbb{RP}^2\#\mathbb{RP}^2$.
        
        Then for all $g\in\pi_1\varDelta\smallsetminus\{1\}$, there is an inequality
        \[
            \scl_{\pi_1\varDelta}\left([g]\right)\geq\frac{1}{4}.
        \]
        Moreover, this bound is sharp: $\scl_{\pi_1\varDelta}\left([c]\right)=\frac{1}{4}$.
    \end{ex}
    \begin{proof}
        The group $\pi_1\varDelta$ splits as an HNN-extension $\pi_1\varDelta=G\ast_{\mathbb{Z}}$, where
        \[
            G=\left<a_1,a_2,c\st{}a_1a_2=c^2\right>,
        \]
        and the HNN-extension is given by the isomorphism $\left<a_1\right>\cong\left<a_2\right>$ sending $a_1$ to $a_2$.
        
        Note that $G$ is a free group, and both $\left\{a_1,c\right\}$ and $\left\{a_2,c\right\}$ are free bases of $G$.
        It follows that $\left<a_1\right>$ and $\left<a_2\right>$ are \emph{left relatively convex} in $G$, meaning that the $G$-sets $G/\left<a_1\right>$ and $G/\left<a_2\right>$ admit $G$-invariant orders.
        This will allow us to apply results of Chen and Heuer \cite{chen-heuer} on spectral gaps in graphs of groups.
        
        For $g\in\pi_1\varDelta\smallsetminus\{1\}$, there are two cases:
        \begin{itemize}
            \item If $g$ is hyperbolic for the HNN-splitting $G\ast_\mathbb{Z}$, then since $\left<a_1\right>$ and $\left<a_2\right>$ are left relatively convex in $G$, it follows from \cite{chen-heuer}*{Theorem 5.19} that $\scl_{\pi_1\varDelta}\left([g]\right)\geq\frac{1}{2}$.
            \item If $g$ is elliptic, then $g$ is conjugate to a non-trivial element $g_0$ of $G$.
            Since $G$ is a free group, the Duncan--Howie Theorem \cite{duncan-howie} implies that $\scl_G\left(\left[g_0\right]\right)\geq\frac{1}{2}$.
            Now Corollary \ref{cor:scl-isom-grph-gps} implies that $\scl_{\pi_1\varDelta}\left([g]\right)=\scl_G\left(\left[g_0\right]\right)\geq\frac{1}{2}$ whenever some multiple of $\left[g_0\right]$ lies in $B_1^{\hmtp}\left(G;\mathbb{Z}\right)$.
        \end{itemize}
        It now follows that any $g\in\pi_1\varDelta$ for which $\scl_{\pi_1\varDelta}\left([g]\right)<\frac{1}{2}$ must be conjugate to some $g_0\in G$ whose image in $H_1\left(G;\mathbb{Q}\right)$ lies in $\Ker\left(H_1\left(G;\mathbb{Q}\right)\rightarrow H_1\left(\pi_1\varDelta;\mathbb{Q}\right)\right)$ (since $\left[g_0\right]$ has finite $\scl$ if and only if some multiple of $\left[g_0\right]$ lies in $B_1^{\hmtp}\left(G;\mathbb{Z}\right)$, if and only if the image of $\left[g_0\right]$ in $H_1\left(G;\mathbb{Q}\right)$ vanishes --- see Remark \ref{rk:compute-homol-homog}).
        But we have
        \[
            H_1\left(G;\mathbb{Q}\right)=\mathbb{Q}a_1\oplus\mathbb{Q}c
            \xrightarrow{\substack{a_1\mapsto a\\\phantom{_1}c\mapsto 0}}
            \mathbb{Q}a\oplus\mathbb{Q}b=H_1\left(\pi_1\varDelta;\mathbb{Q}\right),
        \]
        so $\Ker\left(H_1\left(G;\mathbb{Q}\right)\rightarrow H_1\left(\pi_1\varDelta;\mathbb{Q}\right)\right)=\mathbb{Q}c$.
        Therefore, any element $g$ of $\pi_1\varDelta$ with $\scl_{\pi_1\varDelta}\left([g]\right)<\frac{1}{2}$ must be conjugate into $\left<c\right>$.
        
        Now $\scl_{\pi_1\varDelta}\left([c]\right)=\frac{1}{2}\scl_{\pi_1\varDelta}\left(\left[[a,b]\right]\right)$ since $c^2=\left[a,b\right]$.
        But $\scl_{\pi_1\varDelta}\left(\left[[a,b]\right]\right)\geq\frac{1}{2}$ since $[a,b]$ is hyperbolic in $G\ast_\mathbb{Z}$, and it is a general fact that the $\scl$ of a commutator is at most $\frac{1}{2}$ (since any commutator is bounded by a genus-$1$ surface with one boundary component, see also \cite{culler}*{Example 2.6}), so $\scl_{\pi_1\varDelta}\left([c]\right)=\frac{1}{4}$ as wanted.
    \end{proof}
        
    In particular, the Duncan--Howie Theorem \cite{duncan-howie} implies that $\pi_1\varDelta$ is not residually free (this also follows from a result of Lyndon \cites{lyndon,lyndon-schuetzenberger}).
    Moreover, it follows from a theorem of Heuer \cite{heuer-raags} (see also \cite{m:ang-str}) that $\pi_1\varDelta$ is not a subgroup of any right-angled Artin group, and thus not special\footnote{To be more precise, $\pi_1\varDelta$ is not \emph{$A$-special} in the terminology of \cite{haglund-wise}.} in the sense of Haglund and Wise \cite{haglund-wise}.
    Note however that $\pi_1\varDelta$ is virtually special since its orientation double cover is an orientable closed surface of genus $2$.
    
\section{An algebraic interpretation à la Hopf}\label{sec:hopf}
    
    We now prove a relative version of the Hopf formula, and explain how this can be used to provide a purely algebraic interpretation of Theorem \ref{thm:bav-dual-rel-grom}.
    We focus on the special case of the homology of a group relative to the conjugacy class of an element (rather than that of a chain).
    An analogous Hopf formula could be given in the general case, but the notation would become cumbersome.
    
    \subsection{A relative Hopf formula}
    
    Recall that the classical Hopf formula computes $H_2(G)$ when $G$ is a group given by a presentation (see \cite{brown}*{Theorem II.5.3}):
    \begin{thm}[Hopf formula \cite{hopf}]\label{thm:hopf}
        Let $F$ be a free group, $R\trianglelefteq F$, and $G=F/R$.
        Then there is an isomorphism
        \[
            H_2\left(G;\mathbb{Z}\right)\cong R\cap[F,F]/[F,R].
        \]
    \end{thm}
    
    With the same setup as in Theorem \ref{thm:hopf}, our goal is to compute $H_2\left(G,[w];\mathbb{Z}\right)$, where $[w]\in C_1^{\hmtp}\left(G;\mathbb{Z}\right)$ is an integral conjugacy class represented by an element $w\in G$ (see Remark \ref{rk:homot-clss-chns}\ref{rk:homot-clss-chns:1} for the relation between conjugacy classes of chains and of elements).
    This is provided by the following theorem; our proof is topological and inspired by \cite{cal-scl}*{\S{}1.1.6} and \cite{putman}.
    
    \begin{thm}[Relative Hopf formula]\label{thm:rel-hopf}
        Let $F$ be a free group, $R\trianglelefteq F$, and $G=F/R$.
        Let $w$ be an infinite-order element of $G$, and let $\bar{w}\in F$ be a preimage of $w$ under $F\xrightarrow{p}F/R$.
        Then there is an isomorphism
        \[
            H_2\left(G,[w];\mathbb{Z}\right){\cong}\left<\bar{w}\right>R\cap[F,F]/[F,R].
        \]
    \end{thm}
    \begin{proof}
        Let $X$ be a $K(G,1)$ with a fixed basepoint $x_0$ and let $\gamma:S^1\rightarrow X$ be a loop based at $x_0$ representing $w$.
        Then $H_2\left(G,[w]\right)=H_2\left(X,\gamma\right)$ (see Definition \ref{def:relhom-gp}), and we construct a morphism
        \[
            \Phi:\left<\bar{w}\right>R\cap[F,F]\rightarrow H_2\left(X,\gamma;\mathbb{Z}\right)
        \]
        as follows.
        Let $\bar{g}\in\left<\bar{w}\right>R\cap\left[F,F\right]$.
        Since $\bar{g}\in[F,F]$, one can write
        \[
            \bar{g}=\left[\bar{a}_1,\bar{b}_1\right]\cdots\left[\bar{a}_k,\bar{b}_k\right],
        \]
        with $\bar{a}_i,\bar{b}_i\in F$.
        Set $a_i=p\left(\bar{a}_i\right)\in G$, $b_i=p\left(\bar{b}_i\right)\in G$ and $g=p\left(\bar{g}\right)\in G$.
        The assumption that $\bar{g}\in\left<\bar{w}\right>R$ in $F$ means that $g\in\left<w\right>$ in $G$, so one can write $g=w^n$ for some $n\in\mathbb{Z}$.
        Moreover, since $w$ has infinite order, the integer $n$ is uniquely determined by $\bar{g}$.
        Let $\Sigma_{k,1}$ be an oriented genus-$k$ surface with one boundary component.
        The surface $\Sigma_{k,1}$ has a cell structure with one $0$-cell $\bullet$, $(2k+1)$ $1$-cells with labels ${\alpha}_1,{\beta}_1,\dots,{\alpha}_k,{\beta_k},{\delta}$, and one $2$-cell glued along the word ${\delta}^{-1}\left[{\alpha}_1,{\beta}_1\right]\cdots\left[{\alpha}_k,{\beta}_k\right]$ --- see Figure \ref{fig:cell-struct-sigma}.
        \begin{figure}[htb]
            \centering
            \begin{subfigure}[c]{10em}
                \centering
                \begin{tikzpicture}[every node/.style={rectangle,draw=none,fill=none},scale=.65]
                    \foreach \i in {0,1,...,8}{\coordinate (c\i) at (\i*360/9:4.5em);}
                    
                    \draw [draw=none,pattern={Lines[angle=45,distance=7pt]},pattern color=gray] (c0) -- (c1) -- (c2) -- (c3) -- (c4) -- (c5) -- (c6) -- (c7) -- (c8) -- cycle;
                    
                    \foreach \i in {0,1,...,8}{\node [draw,circle,fill,inner sep=1pt] (x\i) at (c\i) {};}
                    
                    \draw [thick] (x0) -- (x8) coordinate [midway] (mid);
                    \draw [thick,->] (x0) -- (mid) node [right] {${\delta}$};
                    
                    \draw [red,thick] (x0) -- (x1) coordinate [midway] (mid);
                    \draw [red,thick,->] (x0) -- (mid) node [right] {${\alpha}_1$};
                    
                    \draw [red,thick] (x2) -- (x3) coordinate [midway] (mid);
                    \draw [red,thick,->] (x3) -- (mid) node [above] {${\alpha}_1$};
                    
                    \draw [red,thick] (x4) -- (x5) coordinate [midway] (mid);
                    \draw [red,thick,->] (x4) -- (mid) node [left] {${\alpha}_2$};
                    
                    \draw [red,thick] (x6) -- (x7) coordinate [midway] (mid);
                    \draw [red,thick,->] (x7) -- (mid) node [below] {${\alpha}_2$};
                    
                    \draw [ForestGreen,thick] (x1) -- (x2) coordinate [midway] (mid);
                    \draw [ForestGreen,thick,->] (x1) -- (mid) node [above right] {${\beta}_1$};
                    
                    \draw [ForestGreen,thick] (x3) -- (x4) coordinate [midway] (mid);
                    \draw [ForestGreen,thick,->] (x4) -- (mid) node [above left] {${\beta}_1$};
                    
                    \draw [ForestGreen,thick] (x5) -- (x6) coordinate [midway] (mid);
                    \draw [ForestGreen,thick,->] (x5) -- (mid) node [below left] {${\beta}_2$};
                    
                    \draw [ForestGreen,thick] (x7) -- (x8) coordinate [midway] (mid);
                    \draw [ForestGreen,thick,->] (x8) -- (mid) node [below right] {${\beta}_2$};
                    
                    \draw node at (0,0) {\Huge$\circlearrowleft$};
                \end{tikzpicture}
            \end{subfigure}
            \quad{\LARGE$=$}\quad
            \begin{subfigure}[c]{18em}
                \centering
                \begin{tikzpicture}[every node/.style={rectangle,draw=none,fill=none},scale=.65,rotate=180]
                    \draw [thick] (-2.5,0) to[bend right=45] (-.5,-2)
                        (-.5,-2) to (4.5,-2)
                        (4.5,-2) to[bend right=45] (6.5,0)
                        (6.5,0) to[bend right=45] (4.5,2)
                        (4.5,2) to (3.5,2)
                        (3.5,2) to [bend left=45] (3,2.5)
                        (1,2.5) to [bend left=45] (.5,2)
                        (.5,2) to (-.5,2)
                        (-.5,2) to [bend right=45] (-2.5,0);
                    
                    \draw [thick] (3,2.5) to[bend left=15] (1,2.5)
                        (3,2.5) to[bend right=15] (1,2.5) node [right] {${\delta}$};
                    \draw [thick,<-] (1.999,2.65) -- (2.001,2.65);
                    
                    \node [draw,circle,fill,inner sep=1.5pt] (x0) at (2,2.35) {};
                    
                    \draw [red,thick] (x0) to[bend right=25] (4,1.5)
                        (5.5,0) to [bend left=45] (4,-1.5)
                        (4,-1.5) to [bend left=45] (2.5,0)
                        (2.5,0) to[bend right=5] (x0);
                    \draw [red,thick,->] (4,1.5) to[bend left=45] (5.5,0)
                        node [left] {\small$\alpha_1$};
                        
                    \draw [red,thick] (x0) to[bend left=25] (0,1.5)
                        (0,1.5) to[bend right=45] (-1.5,0)
                        (0,-1.5) to [bend right=45] (1.5,0)
                        (1.5,0) to[bend left=5] (x0);
                    \draw [red,thick,->] (0,-1.5) to [bend left=45] (-1.5,0)
                        node [above left] {\small$\alpha_2$};
                        
                    \draw [ForestGreen,thick] (x0) -- (3.27,0) coordinate [midway] (midb11);
                    \draw [ForestGreen,thick,-<] (x0) to (midb11);
                    \draw [ForestGreen,thick] (2,-2) -- (x0) coordinate [midway] (midb12);
                    \draw [ForestGreen,dashed] (3.27,0) to[bend left=25] (2,-2);
                    \draw [ForestGreen,thick,-<] (2,-2) to (midb12)
                        node [above left,xshift=.35em] {\small$\beta_1$};
                        
                    \draw [ForestGreen,thick] (x0) -- (.73,0) coordinate [midway] (midb2);
                    \draw [ForestGreen,thick,-<] (x0) to (midb2);
                    \draw [ForestGreen,dashed] (.73,0) to[bend left=25] (0,-2);
                    \draw [ForestGreen,thick] (0,-2) to[bend left=45] (-2,0)
                        node [above right,xshift=-.35em] {\small$\beta_2$};
                    \draw [ForestGreen,thick,<-] (-2,0) to[bend left=40] (0,1.8);
                    \draw [ForestGreen,thick] (0,1.8) to[bend right=10] (x0);
                    
                    \draw [thick] (-0.73,0) to[bend left=37] (0,.5) (0,.5) to[bend left=37] (0.73,0)
                        (-0.8,.3) to[bend right=45] (0,-.5) (0,-.5) to[bend right=45] (0.8,.3);
                    
                    \draw [xshift=4cm,thick] (-0.73,0) to[bend left=37] (0,.5) (0,.5) to[bend left=37] (0.73,0)
                        (-0.8,.3) to[bend right=45] (0,-.5) (0,-.5) to[bend right=45] (0.8,.3);
                \end{tikzpicture}
            \end{subfigure}
            \caption{The cell structure on $\Sigma_{k,1}$ (with $k=2$).}
            \label{fig:cell-struct-sigma}
        \end{figure}
        
        First pick a degree-$n$ map $\partial f:\partial\Sigma_{k,1}\rightarrow S^1$.        
        Then define a map $f^{(1)}:\Sigma_{k,1}^{(1)}\rightarrow X$ on the $1$-skeleton of $\Sigma_{k,1}$ by sending $\bullet$ to the basepoint $x_0$ of $X$, each $1$-cell ${\alpha}_i$ to a loop representing $a_i$ in $\pi_1\left(X,x_0\right)=G$, each ${\beta}_i$ to a loop representing $b_i$, and define $f^{(1)}$ on $\delta$ by $f^{(1)}_{\left|\delta\right.}=\gamma\circ\partial f$; in particular, $\delta$ is mapped to a loop representing $g=w^n$.
        Since $g=\left[a_1,b_1\right]\cdots\left[a_k,b_k\right]$ in $G=\pi_1\left(X,x_0\right)$, the map $f^{(1)}:\Sigma_{k,1}^{(1)}\rightarrow X$ can be extended over the $2$-cell of $\Sigma_{k,1}$ to $f:\Sigma_{k,1}\rightarrow X$.
        Now the data of $f$ and $\partial f$ define an admissible surface
        \[
            f:\left(\Sigma_{k,1},\partial\Sigma_{k,1}\right)\rightarrow\left(X,\gamma\right).
        \]
        Note moreover that the homotopy types of $f$ and $\partial f$ are uniquely determined by the choice of an expression $\left[\bar{a}_1,\bar{b}_1\right]\cdots\left[\bar{a}_k,\bar{b}_k\right]$; this would fail if $w$ had torsion because in that case, the integer $n$ is not unique --- see Remark \ref{rk:rel-hopf}\ref{rk:rel-hopf-2} below.
        
        Now we define $\Phi(\bar{g})$ by
        \[
            \Phi(\bar{g})\coloneqq f_*\left[\Sigma_{k,1}\right]\in H_2\left(X,\gamma;\mathbb{Z}\right),
        \]
        where $\left[\Sigma_{k,1}\right]\in H_2\left(\Sigma_{k,1},\partial\Sigma_{k,1};\mathbb{Z}\right)$ is the integral fundamental class of $\Sigma_{k,1}$ (note that $\Sigma_{k,1}$ was chosen with an orientation).
        
        The construction of $\Phi\left(\bar{g}\right)$ explained above depends \emph{a priori} on the choice of an expression $\bar{g}=\left[\bar{a}_1,\bar{b}_1\right]\cdots\left[\bar{a}_k,\bar{b}_k\right]$, which might not be unique.
        For now, we see $\Phi$ as a map defined on the monoid $\Theta$ of all formal expressions $\left[\bar{a}_1,\bar{b}_1\right]\cdots\left[\bar{a}_k,\bar{b}_k\right]$ whose image in $F$ lie in $\left<\bar{w}\right>R$, and we'll show that this induces a well-defined map on $\left<\bar{w}\right>R\cap[F,F]$.
        
        \begin{claim}
            The map $\Phi:\Theta\rightarrow H_2\left(X,\gamma;\mathbb{Z}\right)$ is a monoid homomorphism.
        \end{claim}
        \begin{proof}[Proof of the claim]
            Consider two formal expressions $\theta=\left[\bar{a}_1,\bar{b}_1\right]\cdots\left[\bar{a}_k,\bar{b}_k\right]$ and $\theta'=\left[\bar{a}'_1,\bar{b}'_1\right]\cdots\left[\bar{a}'_\ell,\bar{b}'_\ell\right]$ in $\Theta$.
            As explained above, this gives rise to admissible surfaces $f:\left(\Sigma_{k,1},\partial\Sigma_{k,1}\right)\rightarrow(X,\gamma)$ and $f':\left(\Sigma_{\ell,1},\partial\Sigma_{\ell,1}\right)\rightarrow(X,\gamma)$, and we have $\Phi(\theta)=f_*\left[\Sigma_{k,1}\right]$ and $\Phi\left(\theta'\right)=f'_*\left[\Sigma_{\ell,1}\right]$.
            Consider the wedge sum
            \[
                \Sigma_\vee\coloneqq\Sigma_{k,1}\vee\Sigma_{\ell,1}.
            \]
            The maps $f$ and $f'$ naturally induce $f_\vee:\Sigma_\vee\rightarrow X$, and the fundamental classes of $\Sigma_{k,1}$ and $\Sigma_{\ell,1}$ sum to a class $\left[\Sigma_\vee\right]\in H_2\left(\Sigma_\vee,\partial\Sigma_\vee;\mathbb{Z}\right)$, where we define $\partial\Sigma_\vee=\partial\Sigma_{k,1}\vee\partial\Sigma_{\ell,1}\subseteq\Sigma_\vee$.
            Hence,
            \[
                \Phi\left(\theta_1\right)+\Phi\left(\theta_2\right)=\left(f_\vee\right)_*\left[\Sigma_\vee\right].
            \]
            Now there is a homotopy equivalence $\left(\Sigma_\vee,\partial\Sigma_\vee\right)\simeq\left(\Sigma_{k+\ell,1},\partial\Sigma_{k+\ell,1}\right)$, as illustrated in Figure \ref{fig:hopf-monoid-hom}.
            \begin{figure}[htb]
                \centering
                \begin{subfigure}[c]{15em}
                    \centering
                    \begin{tikzpicture}[every node/.style={rectangle,draw=none,fill=none},scale=.5]
                        \foreach \i in {0,1,...,8}{\coordinate (c\i) at ({2*cos(\i*360/9)},{2*sin(\i*360/9)});}
                        
                        \draw [draw=none,pattern={Lines[angle=45,distance=7pt]},pattern color=gray] (c0) -- (c1) -- (c2) -- (c3) -- (c4) -- (c5) -- (c6) -- (c7) -- (c8) -- cycle;
                        
                        \foreach \i in {0,1,...,8}{\node [draw,circle,fill,inner sep=1pt] (x\i) at (c\i) {};}
                        
                        \draw [thick] (x0) -- (x8) coordinate [midway] (mid);
                        \draw [thick,->] (x0) -- (mid) node [left] {$\delta'$};
                        
                        \draw [red,thick] (x0) -- (x1) coordinate [midway] (mid);
                        \draw [red,thick,->] (x0) -- (mid) node [left] {$\alpha'_1$};
                        
                        \draw [red,thick] (x2) -- (x3) coordinate [midway] (mid);
                        \draw [red,thick,->] (x3) -- (mid) node [above] {$\alpha'_1$};
                        
                        \draw [red,thick] (x4) -- (x5) coordinate [midway] (mid);
                        \draw [red,thick,->] (x4) -- (mid) node [left] {$\alpha'_2$};
                        
                        \draw [red,thick] (x6) -- (x7) coordinate [midway] (mid);
                        \draw [red,thick,->] (x7) -- (mid) node [below] {$\alpha'_2$};
                        
                        \draw [ForestGreen,thick] (x1) -- (x2) coordinate [midway] (mid);
                        \draw [ForestGreen,thick,->] (x1) -- (mid) node [above] {$\beta'_1$};
                        
                        \draw [ForestGreen,thick] (x3) -- (x4) coordinate [midway] (mid);
                        \draw [ForestGreen,thick,->] (x4) -- (mid) node [above left] {$\beta'_1$};
                        
                        \draw [ForestGreen,thick] (x5) -- (x6) coordinate [midway] (mid);
                        \draw [ForestGreen,thick,->] (x5) -- (mid) node [below left] {$\beta'_2$};
                        
                        \draw [ForestGreen,thick] (x7) -- (x8) coordinate [midway] (mid);
                        \draw [ForestGreen,thick,->] (x8) -- (mid) node [below] {$\beta'_2$};
                        
                        \draw node at (0,0) {\Huge$\circlearrowleft$};
                        
                        \foreach \i in {0,1,...,8}{\coordinate (c\i) at ({4+2*cos(180-\i*360/9)},{2*sin(180-\i*360/9)});}
                        
                        \draw [draw=none,pattern={Lines[angle=45,distance=7pt]},pattern color=gray] (c0) -- (c1) -- (c2) -- (c3) -- (c4) -- (c5) -- (c6) -- (c7) -- (c8) -- cycle;
                        
                        \foreach \i in {0,1,...,8}{\node [draw,circle,fill,inner sep=1pt] (x\i) at (c\i) {};}
                        
                        \draw [thick] (x0) -- (x8) coordinate [midway] (mid);
                        \draw [thick,->] (x8) -- (mid) node [right] {$\delta$};
                        
                        \draw [ForestGreen,thick] (x0) -- (x1) coordinate [midway] (mid);
                        \draw [ForestGreen,thick,->] (x0) -- (mid) node [right] {$\beta_2$};
                        
                        \draw [ForestGreen,thick] (x2) -- (x3) coordinate [midway] (mid);
                        \draw [ForestGreen,thick,->] (x3) -- (mid) node [above] {$\beta_2$};
                        
                        \draw [ForestGreen,thick] (x4) -- (x5) coordinate [midway] (mid);
                        \draw [ForestGreen,thick,->] (x4) -- (mid) node [right] {$\beta_1$};
                        
                        \draw [ForestGreen,thick] (x6) -- (x7) coordinate [midway] (mid);
                        \draw [ForestGreen,thick,->] (x7) -- (mid) node [below] {$\beta_1$};
                        
                        \draw [red,thick] (x1) -- (x2) coordinate [midway] (mid);
                        \draw [red,thick,->] (x1) -- (mid) node [above] {$\alpha_2$};
                        
                        \draw [red,thick] (x3) -- (x4) coordinate [midway] (mid);
                        \draw [red,thick,->] (x4) -- (mid) node [above right] {$\alpha_2$};
                        
                        \draw [red,thick] (x5) -- (x6) coordinate [midway] (mid);
                        \draw [red,thick,->] (x5) -- (mid) node [below right] {$\alpha_1$};
                        
                        \draw [red,thick] (x7) -- (x8) coordinate [midway] (mid);
                        \draw [red,thick,->] (x8) -- (mid) node [below] {$\alpha_1$};
                        
                        \draw node at (4,0) {\Huge$\circlearrowleft$};
                    \end{tikzpicture}
                \end{subfigure}
                \quad{\LARGE$\simeq$}\quad
                \begin{subfigure}[c]{15em}
                    \centering
                    \begin{tikzpicture}[every node/.style={rectangle,draw=none,fill=none},scale=.5]
                        \foreach \i in {0,1,...,9}{\coordinate (c\i) at ({2*cos((\i+.5)*360/10)},{2*sin((\i+.5)*360/10)});}
                        
                        \draw [draw=none,pattern={Lines[angle=45,distance=7pt]},pattern color=gray] (c0) -- (c1) -- (c2) -- (c3) -- (c4) -- (c5) -- (c6) -- (c7) -- (c8) -- (c9) -- cycle;
                        
                        \foreach \i in {0,1,...,9}{\node [draw,circle,fill,inner sep=1pt] (x\i) at (c\i) {};}
                        
                        \draw [thick] (x8) -- (x9) coordinate [midway] (mid);
                        \draw [thick,->] (x9) -- (mid) node [left] {$\delta'$};
                        
                        \draw [red,thick] (x0) -- (x1) coordinate [midway] (mid);
                        \draw [red,thick,->] (x0) -- (mid) node [left] {$\alpha'_1$};
                        
                        \draw [red,thick] (x2) -- (x3) coordinate [midway] (mid);
                        \draw [red,thick,->] (x3) -- (mid) node [above] {$\alpha'_1$};
                        
                        \draw [red,thick] (x4) -- (x5) coordinate [midway] (mid);
                        \draw [red,thick,->] (x4) -- (mid) node [left] {$\alpha'_2$};
                        
                        \draw [red,thick] (x6) -- (x7) coordinate [midway] (mid);
                        \draw [red,thick,->] (x7) -- (mid) node [below] {$\alpha'_2$};
                        
                        \draw [ForestGreen,thick] (x1) -- (x2) coordinate [midway] (mid);
                        \draw [ForestGreen,thick,->] (x1) -- (mid) node [above] {$\beta'_1$};
                        
                        \draw [ForestGreen,thick] (x3) -- (x4) coordinate [midway] (mid);
                        \draw [ForestGreen,thick,->] (x4) -- (mid) node [above left] {$\beta'_1$};
                        
                        \draw [ForestGreen,thick] (x5) -- (x6) coordinate [midway] (mid);
                        \draw [ForestGreen,thick,->] (x5) -- (mid) node [below left] {$\beta'_2$};
                        
                        \draw [ForestGreen,thick] (x7) -- (x8) coordinate [midway] (mid);
                        \draw [ForestGreen,thick,->] (x8) -- (mid) node [below] {$\beta'_2$};
                        
                        \draw node at (0,0) {\Huge$\circlearrowleft$};
                        \foreach \i in {0,1,...,9}{\coordinate (c\i) at ({4*cos(18)+2*cos(180-(\i+.5)*360/10)},{2*sin(180-(\i+.5)*360/10)});}
                        
                        \draw [draw=none,pattern={Lines[angle=45,distance=7pt]},pattern color=gray] (c0) -- (c1) -- (c2) -- (c3) -- (c4) -- (c5) -- (c6) -- (c7) -- (c8) -- (c9) -- cycle;
                        
                        \foreach \i in {0,1,...,9}{\node [draw,circle,fill,inner sep=1pt] (x\i) at (c\i) {};}
                        
                        \draw [thick] (x8) -- (x9) coordinate [midway] (mid);
                        \draw [thick,->] (x8) -- (mid) node [right] {$\delta$};
                        
                        \draw [ForestGreen,thick] (x0) -- (x1) coordinate [midway] (mid);
                        \draw [ForestGreen,thick,->] (x0) -- (mid) node [right] {$\beta_2$};
                        
                        \draw [ForestGreen,thick] (x2) -- (x3) coordinate [midway] (mid);
                        \draw [ForestGreen,thick,->] (x3) -- (mid) node [above] {$\beta_2$};
                        
                        \draw [ForestGreen,thick] (x4) -- (x5) coordinate [midway] (mid);
                        \draw [ForestGreen,thick,->] (x4) -- (mid) node [right] {$\beta_1$};
                        
                        \draw [ForestGreen,thick] (x6) -- (x7) coordinate [midway] (mid);
                        \draw [ForestGreen,thick,->] (x7) -- (mid) node [below] {$\beta_1$};
                        
                        \draw [red,thick] (x1) -- (x2) coordinate [midway] (mid);
                        \draw [red,thick,->] (x1) -- (mid) node [above] {$\alpha_2$};
                        
                        \draw [red,thick] (x3) -- (x4) coordinate [midway] (mid);
                        \draw [red,thick,->] (x4) -- (mid) node [above right] {$\alpha_2$};
                        
                        \draw [red,thick] (x5) -- (x6) coordinate [midway] (mid);
                        \draw [red,thick,->] (x5) -- (mid) node [below right] {$\alpha_1$};
                        
                        \draw [red,thick] (x7) -- (x8) coordinate [midway] (mid);
                        \draw [red,thick,->] (x8) -- (mid) node [below] {$\alpha_1$};
                        
                        \draw node at (4,0) {\Huge$\circlearrowleft$};
                    \end{tikzpicture}
                \end{subfigure}
                \caption{The homotopy equivalence $\Sigma_{k,1}\vee\Sigma_{\ell,1}\simeq\Sigma_{k+\ell,1}$ (here $k=\ell=2$).}
                \label{fig:hopf-monoid-hom}
            \end{figure}
            This yields an admissible surface $\left(\Sigma_{k+\ell,1},\partial\Sigma_{k+\ell,1}\right)\rightarrow(X,\gamma)$ representing the class $\Phi\left(\theta\right)+\Phi\left(\theta'\right)$ in $H_2\left(X,\gamma;\mathbb{Z}\right)$.
            But note that this admissible surface is exactly the one obtained when the above construction is applied to $\theta\theta'$.
            This proves that $\Phi\left(\theta\right)+\Phi\left(\theta'\right)=\Phi\left(\theta\theta'\right)$, so $\Phi$ is a monoid homomorphism.
        \end{proof}
        
        Using the claim, we now prove that $\Phi$ induces a well-defined map on $\left<\bar{w}\right>R\cap[F,F]$.
        Consider two formal expressions $\theta,\theta'\in\Theta$ defining the same element of $\left<\bar{w}\right>R\cap[F,F]$.
        Write $\theta=\left[\bar{a}_1,\bar{b}_1\right]\cdots\left[\bar{a}_k,\bar{b}_k\right]$, and consider its formal inverse $\theta^{-1}=\left[\bar{b}_k,\bar{a}_k\right]\cdots\left[\bar{b}_1,\bar{a}_1\right]\in\Theta$ (which, despite our choice of notation, is not an inverse of $\theta$ in the monoid $\Theta$!).
        Then the formal expression $\theta^{-1}\theta'$ represents the trivial element of $F$.
        This means that the above construction for the formal expression $\theta^{-1}\theta'$ can actually be performed when the $K(G,1)$ space $X$ is replaced with a $K(F,1)$ space $X_F$.
        In other words, the admissible surface $f:\left(\Sigma_{k,1},\partial\Sigma_{k,1}\right)\rightarrow\left(X,\gamma\right)$ associated to $\theta^{-1}\theta'$ factors through the map $X_F\rightarrow X$ induced by $F\rightarrow G$.
        Moreover, the image of $\partial\Sigma_{k,1}$ is nullhomotopic in $X_F$, from which it follows that
        \[
            f_*\left[\Sigma_{k,1}\right]\in H_2\left(X_F;\mathbb{Z}\right)\subseteq H_2\left(X_F,\bar{\gamma};\mathbb{Z}\right),
        \]
        where $\bar{\gamma}:S^1\rightarrow X_F$ is a representative of $\bar{w}\in F$.
        But $H_2\left(X_F;\mathbb{Z}\right)=H_2\left(F;\mathbb{Z}\right)=0$ since $F$ is a free group, so $\left[\Sigma_{k,1}\right]$ maps to zero in $H_2\left(X_F,\bar{\gamma};\mathbb{Z}\right)$, and hence also in $H_2\left(X,\gamma;\mathbb{Z}\right)$.
        Therefore, it follows from the claim that
        \[
            0=\Phi\left(\theta^{-1}\theta'\right)=\Phi\left(\theta^{-1}\right)+\Phi\left(\theta'\right),
        \]
        and it is clear from the construction that $\Phi\left(\theta^{-1}\right)=-\Phi\left(\theta\right)$, so $\Phi\left(\theta\right)=\Phi\left(\theta'\right)$ as wanted.
        This proves that $\Phi$ induces a well-defined map
        \[
            \Phi:\left<\bar{w}\right>R\cap[F,F]\rightarrow H_2\left(X,\gamma;\mathbb{Z}\right),
        \]
        which is a group homomorphism by the claim.
        
        The homomorphism $\Phi$ is surjective since every element of $H_2\left(X,\gamma;\mathbb{Z}\right)$ can be represented by a map from an orientable compact connected surface with one boundary component --- this follows from Proposition \ref{prop:ell1-gromnorm} and Lemma \ref{lmm:simple-incompr-adm-srf}.
        
        It remains to show the following:
        \begin{claim}
            $\Ker\Phi=\left[F,R\right]$.
        \end{claim}
        \begin{proof}[Proof of the claim]
            To prove that $\left[F,R\right]\subseteq\Ker\Phi$, it suffices to show that for every $\bar{g}\in F$ and $\bar{r}\in R$, we have $\left[\bar{g},\bar{r}\right]\in\Ker\Phi$.
            But $\Phi\left(\left[\bar{g},\bar{r}\right]\right)=f_*\left[\Sigma_{1,1}\right]$, where $\Sigma_{1,1}$ is a torus with one boundary component, with equator mapping to $\bar{g}$ and meridian mapping to $\bar{r}$.
            Since the image of $\bar{r}$ in $G$ is trivial, we may cut $\Sigma_{1,1}$ along the meridian and fill in the two resulting discs, obtaining a map $f_1:\left(D^2,\partial D^2\right)\rightarrow\left(X,\gamma\right)$.
            We can glue $f_1$ to itself with reversed orientation along $\partial D^2$ to obtain $f_2:S^2\rightarrow X$.
            But $X$ is assumed to be a $K(G,1)$ so it is aspherical, and $f_2$ is nullhomotopic.
            Therefore, $f_1$ is also nullhomotopic, and $f_*\left[\Sigma_{1,1}\right]=\left(f_1\right)_*\left[D^2\right]=0$.
            This proves that $\Phi\left(\left[\bar{g},\bar{r}\right]\right)=0$, so $\left[F,R\right]\subseteq\Ker\Phi$.
            
            Conversely, let $\bar{g}\in\Ker\Phi$.
            Let $f:\left(\Sigma,\partial\Sigma\right)\rightarrow(X,\gamma)$ be an admissible surface associated to an expression of $\bar{g}$ as a product of commutators by the above construction, with $\Sigma=\Sigma_{k,1}$.
            The assumption that $\bar{g}\in\Ker\Phi$ means that $f_*\left[\Sigma\right]=0$, so the map $f_*:H_2\left(\Sigma,\partial\Sigma;\mathbb{Z}\right)\rightarrow H_2\left(X,\gamma;\mathbb{Z}\right)$ is zero.
            Long exact sequences of pairs give a commutative diagram with exact rows (with omitted $\mathbb{Z}$-coefficients):
            \[
                \vcenter{\hbox{\begin{tikzpicture}[every node/.style={draw=none,fill=none,rectangle}]
                    \node (Z) at (-2.25,-1.25) {$0$};
                    \node (A) at (0,-1.25) {$H_2\left(X\right)$};
                    \node (Ap) at (0,0) {$0$};
                    \node (B) at (2.25,-1.25) {$H_2\left(X,\gamma\right)$};
                    \node (Bp) at (2.25,0) {$H_2\left(\Sigma,\partial\Sigma\right)$};
                    \node (C) at (4.5,-1.25) {$H_1\left(S_1\right)$};
                    \node (Cp) at (4.5,0) {$H_1\left(\partial\Sigma\right)$};
                    \node (D) at (6.75,-1.25) {$H_1(X)$};
                    \node (Dp) at (6.75,0) {$H_1(\Sigma)$};
                    \node (E) at (9,-1.25) {$\cdots$};
                    \node (Ep) at (9,0) {$\cdots$};
                    
                    \draw [->] (Z) -> (A);
                    \draw [->] (A) -> (B);
                    \draw [->] (Ap) -> (Bp);
                    \draw [<-] (B) -> (Bp) node [midway,right] {$0$};
                    \draw [->] (B) -> (C) node [midway,above] {$\partial$};
                    \draw [->] (Bp) -> (Cp) node [midway,above] {$\partial$};
                    \draw [<-] (C) -> (Cp) node [midway,right] {$f_*$};
                    \draw [->] (C) -> (D);
                    \draw [->] (Cp) -> (Dp);
                    \draw [<-] (D) -> (Dp) node [midway,right] {$f_*$};
                    \draw [->] (D) -> (E);
                    \draw [->] (Dp) -> (Ep);
                \end{tikzpicture}}}
            \]
            If $f_*:H_1\left(\partial\Sigma\right)\rightarrow H_1\left(S^1\right)$ were nonzero, then since $H_1\left(\partial\Sigma\right)\cong H_1\left(S^1\right)\cong\mathbb{Z}$, the map $f_*:H_1\left(\partial\Sigma\right)\rightarrow H_1\left(S^1\right)$ would in fact be injective.
            But $f_*\circ\partial=0$, so the map $\partial:H_2\left(\Sigma,\partial\Sigma\right)\rightarrow H_1\left(\partial\Sigma\right)$ would be zero, implying by exactness that $H_1\left(\partial\Sigma\right)=0$ since $H_1\left(\partial\Sigma\right)\rightarrow H_1\left(\Sigma\right)$ is zero.
            This is a contradiction, and therefore the map
            \[
                f_*:H_1\left(\partial\Sigma\right)\rightarrow H_1\left(S^1\right)
            \]
            is zero.
            Therefore, the restriction of $f$ to $\partial\Sigma$ is nullhomotopic, which implies in particular that the image of $\bar{g}$ in $G$ is trivial, i.e.\ $\bar{g}\in R\cap[F,F]$.
            Therefore, we are reduced to the setting of the classical Hopf formula (Theorem \ref{thm:hopf}), i.e.\ $\bar{g}\in R\cap[F,F]$ and $\Phi\left(\bar{g}\right)=1$ in $H_2\left(X;\mathbb{Z}\right)$.
            Since $\Phi$ coincides with the morphism giving the classical Hopf formula (see for instance \cite{putman}), it follows that $\bar{g}\in[F,R]$.
        \end{proof}
        
        We have constructed a surjective group homomorphism $\Phi:\left<\bar{w}\right>R\cap[F,F]\rightarrow H_2\left(X,\gamma;\mathbb{Z}\right)$ with $\Ker\Phi=\left[F,R\right]$, so $\Phi$ induces the desired isomorphism.
    \end{proof}
    
    \begin{rmk}
        \begin{enumerate}
            \item In the proof of Theorem \ref{thm:rel-hopf}, the assumption that $X$ is a $K(G,1)$ is essential.
            This is why --- contrary to the rest of this paper --- we state the theorem in terms of the relative homology of groups, rather than topological spaces.
            \item Theorem \ref{thm:rel-hopf} becomes false if $w$ has finite order $q$ in $G$.
            Indeed, the (absolute) Hopf formula (Theorem \ref{thm:hopf}) says that $H_2\left(G;\mathbb{Z}\right)$ is isomorphic to $R\cap\left[F,F\right]/\left[F,R\right]$, which has finite index in the right-hand side $\left<\bar{w}\right>R\cap\left[F,F\right]/\left[F,R\right]$ of Theorem \ref{thm:rel-hopf} when $w$ has finite order.
            But we know from Example \ref{ex:rel-hom-chn}\ref{ex:rel-hom-chn-4} that
            \[
                H_2\left(G,[w];\mathbb{Z}\right)\cong H_2\left(G;\mathbb{Z}\right)\oplus\mathbb{Z},
            \]
            so $H_2\left(G;\mathbb{Z}\right)$ must have infinite index in $H_2\left(G,[w];\mathbb{Z}\right)$!
            
            Let us explain briefly where the missing homology classes are.
            Pick $\gamma:S^1\rightarrow X$ a loop at $x_0\in X$ representing $w$.
            Then the homology class corresponding to the integer $n\in\mathbb{Z}$ in the $\mathbb{Z}$-summand of $H_2\left(G,\left[w\right];\mathbb{Z}\right)$ is represented by an admissible surface $f:\left(D^2,\partial D^2\right)\rightarrow\left(X,\gamma\right)$, where $D^2$ is the disc, $\partial f:\partial D^2\rightarrow S^1$ is a map of degree $nq$, and $f:D^2\rightarrow X$ is an extension of $\gamma\circ\partial f$ to the disc (which exists since $\gamma^{nq}$ is nullhomotopic).
            
            Note that the underlying maps $f:D^2\rightarrow X$ of the admissible surfaces representing these `missing classes' are nullhomotopic since the disc is contractible.
            This underlines the importance of defining an admissible surface as the data of both maps $f$ \emph{and} $\partial f$.\label{rk:rel-hopf-2}
            
            \item One can recover the classical Hopf formula (Theorem \ref{thm:hopf}) from our proof by observing that our isomorphism
            \[
                \left<\bar{w}\right>R\cap[F,F]/\left[F,R\right]\xrightarrow{\cong}H_2\left(G,\left[w\right];\mathbb{Z}\right)
            \]
            sends $R\cap[F,F]/\left[F,R\right]$ to $H_2\left(G;\mathbb{Z}\right)\subseteq H_2\left(G,\left[w\right];\mathbb{Z}\right)$.
            In other words, our construction maps an element $\bar{g}\in\left<\bar{w}\right>R\cap[F,F]$ to an absolute homology class if and only if $\bar{g}$ has trivial image in $G$.
        \end{enumerate}
        \label{rk:rel-hopf}
    \end{rmk}
    
    \subsection{Bavard duality through the lens of the Hopf formula}
    
    We next explain how to obtain an algebraic restatement of Theorem \ref{thm:bav-dual-rel-grom} using the relative Hopf formula (Theorem \ref{thm:rel-hopf}).
    
    Recall from \S{}\ref{subsec:hb-gp} the definition of the bounded cochain complex $C^*_b\left(G;\mathbb{R}\right)$.
    We denote by $Z^2_b\left(G;\mathbb{R}\right)$ the subspace of $C^2_b\left(G;\mathbb{R}\right)$ consisting of \emph{bounded $2$-cocycles} on $G$, i.e.\ bounded maps $\psi:G^2\rightarrow\mathbb{R}$ such that
    \[
        \psi\left(g_2,g_3\right)-\psi\left(g_1g_2,g_3\right)+\psi\left(g_1,g_2g_3\right)-\psi\left(g_1,g_2\right)=0
    \]
    for all $g_1,g_2,g_3\in G$.
    
    \begin{mthm}{D}[Bavard duality via the Hopf formula]\label{thm:bavard-hopf}
        Let $F$ be a (countable) free group, $R\trianglelefteq F$, and $G=F/R$.
        Let $w\in G$ and let $\bar{w}\in F$ be a preimage of $w$ under $F\xrightarrow{p}G$.
        
        Let $\alpha\in H_2\left(G,\left[w\right];\mathbb{Z}\right)$ and let
        \[
            \left[\bar{a}_1,\bar{b}_1\right]\cdots\left[\bar{a}_k,\bar{b}_k\right]\in\left<\bar{w}\right>R\cap[F,F],
        \]
        be a representative of $\Psi(\alpha)$, where $\Psi:H_2\left(G,\left[w\right];\mathbb{Z}\right)\xrightarrow{\cong}\left<\bar{w}\right>R\cap\left[F,F\right]/[F,R]$ is the isomorphism of Theorem \ref{thm:rel-hopf}.
        Set $a_i=p\left(\bar{a}_i\right)\in G$ and $b_i=p\left(\bar{b}_i\right)\in G$.
        
        Then
        \begin{multline*}
            \gromnorm{\iota\alpha}=\sup\left\{\frac{1}{\inftynorm{\psi}}\left(\psi\left(a_1,b_1\right)+\psi\left(a_1b_1,a_1^{-1}\right)+\psi\left(a_1b_1a_1^{-1},b_1^{-1}\right)\right.\right.\\
            +\psi\left(\left[a_1,b_1\right],a_2\right)+\psi\left(\left[a_1,b_1\right]a_2,b_2\right)+\psi\left(\left[a_1,b_1\right]a_2b_2,a_2^{-1}\right)+\cdots\\
            \left.\vphantom{\frac{1}{\inftynorm{\psi}}}\left.+\psi\left(\left[a_1,b_1\right]\cdots\left[a_{k-1},b_{k-1}\right]a_kb_ka_{k}^{-1},b_k^{-1}\right)\right)\st{}\psi\in Z^2_b\left(G;\mathbb{R}\right)\smallsetminus\{0\}\right\},
        \end{multline*}
        where $\iota:H_2\left(G,\left[w\right];\mathbb{Z}\right)\rightarrow H_2\left(G,\left[w\right];\mathbb{Q}\right)$ is the change-of-coefficients map.
    \end{mthm}
    \begin{proof}
        Let $X$ be a $K(G,1)$ and let $\gamma:S^1\rightarrow X$ represent $w$.
        Recall that the isomorphism $\Psi:H_2\left(G,\left[w\right];\mathbb{Z}\right)\xrightarrow{\cong}\left<\bar{w}\right>R\cap\left[F,F\right]/[F,R]$ was constructed in the proof of Theorem \ref{thm:rel-hopf} by starting with a product of $k$ commutators in $\left<\bar{w}\right>R\cap\left[F,F\right]$, labelling the edges in a cellular decomposition of the compact surface $\Sigma_{k,1}$ with those commutators, mapping $\Sigma_{k,1}$ to $X$ and considering the image of the fundamental class $\left[\Sigma_{k,1}\right]$ in $H_2\left(X,\gamma\right)=H_2\left(G,\left[w\right]\right)$.
        We will now be a bit more specific about the choice of the map $\Sigma_{k,1}\rightarrow X$.
        We start by picking singular simplices $\sigma_{g_1,\dots,g_n}:\Delta^n\rightarrow X$ for each $n$-uple $\left(g_1,\dots,g_n\right)\in G^n$ as in \S{}\ref{subsec:hb-gp} (see in particular Figure \ref{fig:constr-che}), so that the map $\left(g_1,\dots,g_n\right)\mapsto\sigma_{g_1,\dots,g_n}$ induces a chain homotopy equivalence $C_*\left(G;\mathbb{R}\right)\xrightarrow{\sim}C_*^\textnormal{sg}\left(X;\mathbb{R}\right)$.
        Take a one-vertex triangulation of $\Sigma_{k,1}$ as in Figure \ref{fig:trig-sigma}.
        \begin{figure}[htb]
            \centering
            \begin{tikzpicture}[every node/.style={rectangle,draw=none,fill=none},scale=.85]
                \foreach \i in {0,1,...,8}{\coordinate (c\i) at (\i*360/9:4.5em);}
                
                \draw [draw=none,pattern={Lines[angle=45,distance=7pt]},pattern color=gray] (c0) -- (c1) -- (c2) -- (c3) -- (c4) -- (c5) -- (c6) -- (c7) -- (c8) -- cycle;
                
                \foreach \i in {0,1,...,8}{\node [draw,circle,fill,inner sep=1pt] (x\i) at (c\i) {};}
                
                \draw [thick] (x0) -- (x8) coordinate [midway] (mid);
                \draw [thick,->] (x0) -- (mid) node [right] {${\delta}$};
                
                \draw [red,thick] (x0) -- (x1) coordinate [midway] (mid);
                \draw [red,thick,->] (x0) -- (mid) node [right] {${\alpha}_1$};
                
                \draw [red,thick] (x2) -- (x3) coordinate [midway] (mid);
                \draw [red,thick,->] (x3) -- (mid) node [above] {${\alpha}_1$};
                
                \draw [red,thick] (x4) -- (x5) coordinate [midway] (mid);
                \draw [red,thick,->] (x4) -- (mid) node [left] {${\alpha}_2$};
                
                \draw [red,thick] (x6) -- (x7) coordinate [midway] (mid);
                \draw [red,thick,->] (x7) -- (mid) node [below] {${\alpha}_2$};
                
                \draw [ForestGreen,thick] (x1) -- (x2) coordinate [midway] (mid);
                \draw [ForestGreen,thick,->] (x1) -- (mid) node [above right] {${\beta}_1$};
                
                \draw [ForestGreen,thick] (x3) -- (x4) coordinate [midway] (mid);
                \draw [ForestGreen,thick,->] (x4) -- (mid) node [above left] {${\beta}_1$};
                
                \draw [ForestGreen,thick] (x5) -- (x6) coordinate [midway] (mid);
                \draw [ForestGreen,thick,->] (x5) -- (mid) node [below left] {${\beta}_2$};
                
                \draw [ForestGreen,thick] (x7) -- (x8) coordinate [midway] (mid);
                \draw [ForestGreen,thick,->] (x8) -- (mid) node [below right] {${\beta}_2$};
                
                \foreach \i in {2,...,7}{\draw [thick] (x0) -- (x\i);}
            \end{tikzpicture}
            \caption{One-vertex triangulation of $\Sigma_{k,1}$.}
            \label{fig:trig-sigma}
        \end{figure}
        We can construct the map $f:\Sigma_{k,1}\rightarrow X$ explicitly by sending each triangle of $\Sigma_{k,1}$ to the correct singular $2$-simplex among the $\sigma_{g_1,g_2}$'s.
        We obtain in particular that
        \begin{multline}\label{eq:alpha-sing-spxs}
            \alpha=f_*\left[\Sigma_{k,1}\right]=\left[\sigma_{a_1,b_1}+\sigma_{a_1b_1,a_1^{-1}}+\sigma_{a_1b_1a_1^{-1},b_1^{-1}}+\sigma_{\left[a_1,b_1\right],a_2}\right.\\
            \left.+\sigma_{\left[a_1,b_1\right]a_2,b_2}+\cdots+\sigma_{\left[a_1,b_1\right]\cdots\left[a_{k-1},b_{k-1}\right]a_kb_ka_k^{-1},b_k^{-1}}\right]\in H_2\left(X,\gamma;\mathbb{Z}\right).
        \end{multline}
        
        Now Bavard duality for $H_2\left(X,\gamma\right)$ (Theorem \ref{thm:bav-dual-rel-grom}) gives
        \[
            \gromnorm{\iota\alpha}=\sup\left\{\frac{\left<u,\alpha\right>}{\inftynorm{u}}\st{}u\in H^2_b\left(X;\mathbb{R}\right)\smallsetminus\{0\}\right\}.
        \]
        
        Pick some $u\in H^2_b\left(X;\mathbb{R}\right)\cong H^2_b\left(G;\mathbb{R}\right)$ and let $\psi\in Z^2_b\left(G;\mathbb{R}\right)$ be a $2$-cocycle such that $u=\left[\psi\right]$.
        The chain homotopy equivalence $\left(g_1,\dots,g_n\right)\mapsto\sigma_{g_1,\dots,g_n}$ tells one how to evaluate $\psi$ on singular (relative) $2$-cycles spanned by the $\sigma_{g_1,g_2}$'s in $C_2\left(X;\mathbb{R}\right)$: there is an equality
        \[
            \left<\psi,\sigma_{g_1,g_2}\right>=\psi\left(g_1,g_2\right).
        \]
        Therefore (\ref{eq:alpha-sing-spxs}) implies that the Kronecker product $\left<u,\alpha\right>$ is given by
        \begin{multline*}
            \left<u,\alpha\right>=\psi\left(a_1,b_1\right)+\psi\left(a_1b_1,a_1^{-1}\right)+\psi\left(a_1b_1a_1^{-1},b_1^{-1}\right)+\psi\left(\left[a_1,b_1\right],a_2\right)\\
            +\psi\left(\left[a_1,b_1\right]a_2,b_2\right)+\cdots+\psi\left(\left[a_1,b_1\right]\cdots\left[a_{k-1},b_{k-1}\right]a_kb_ka_{k}^{-1},b_k^{-1}\right).
        \end{multline*}
        The result follows, remembering that $\inftynorm{u}=\inf\left\{\inftynorm{\psi}\st{}[\psi]=u\right\}$.
    \end{proof}
    
    \begin{rmk}
        A quasimorphism $\phi\in Q(G)$ defines a bounded $2$-cocycle $d^2\phi\in Z^2_b(G;\mathbb{R})$ as explained in Remark \ref{rk:qm-h2b}.
        Given an integral class $\alpha\in H_2\left(G,\left[w\right];\mathbb{Z}\right)$ with $\partial\alpha=n\left[S^1\right]$, one can use the formula of Theorem \ref{thm:bavard-hopf} to obtain
        \[
            \left<\left[d^2\phi\right],\alpha\right>=n\cdot\phi\left(w\right).
        \]
        Using the lower bound on $\gromnorm{\cdot}$ given by Theorem \ref{thm:bavard-hopf} together with the connection between $\scl$ and $\gromnorm{\cdot}$ (Proposition \ref{prop:scl-gromnorm}), it follows that
        \[
            \scl\left([w]\right)\geq\frac{1}{2}\sup_{\phi\in Q(G)}\frac{\phi(w)}{2\inftynorm{d^2\phi}}.
        \]
        On the other hand, classical Bavard duality (Theorem \ref{thm:bavard}) says that
        \[
            \scl\left([w]\right)=\sup_{\phi\in Q(G)}\frac{\phi(w)}{2\inftynorm{d^2\phi}}.
        \]
        Feeding quasimorphisms into Theorem \ref{thm:bavard-hopf} has yielded a non-optimal lower bound on $\scl$.
        The reason for this is the difference between a cocycle $\psi\in Z^2_b\left(G;\mathbb{R}\right)$ and its class $[\psi]\in H^2_b\left(G;\mathbb{R}\right)$: given $\phi\in Q(G)$, there are inequalities \cite{cal-scl}*{Lemma 2.58}
        \[
            \frac{1}{2}\inftynorm{d^2\phi}\leq\inftynorm{\left[d^2\phi\right]}\leq\inftynorm{d^2\phi},
        \]
        and $\inftynorm{\left[d^2\phi\right]}$ might not be realised by the coboundary of a quasimorphism.
    \end{rmk}

\section{The bounded Euler class}\label{sec:euler}

    Calegari \cite{cal-fnb} exhibited a connection between the rotation quasimorphism, area, and stable commutator length in fundamental groups of compact hyperbolic surfaces with non-empty boundary.
    We explain how this generalises to a statement about the relative Gromov seminorm in possibly closed hyperbolic surface groups.
    
    \subsection{Equivariant (bounded) cohomology}
    
    To define the bounded Euler class, we will use the language of equivariant cohomology.
    
    Given a set $X$ with an action of a group $G$, a \emph{degree-$n$ homogeneous $G$-cochain} (with real coefficients) is a map $\psi:X^{n+1}\rightarrow\mathbb{R}$ which is invariant under the diagonal action of $G$ on $X^{n+1}$, in the sense that
    \[
        \psi\left(x_0,\dots,x_n\right)=\psi\left(gx_0,\dots,gx_n\right)
    \]
    for all $x_0,\dots,x_n\in X$ and $g\in G$.
    We denote by $C^n\left(G\curvearrowright X;\mathbb{R}\right)$ the $\mathbb{R}$-vector space of such cochains; they form a cochain complex $C^*\left(G\curvearrowright X;\mathbb{R}\right)$ with coboundary maps $d^n:C^{n-1}\left(G\curvearrowright X;\mathbb{R}\right)\rightarrow C^{n}\left(G\curvearrowright X;\mathbb{R}\right)$ given by
    \[
        d^n\psi\left(x_0,\dots,x_n\right)\coloneqq\sum_{i=0}^n(-1)^i\psi\left(x_0,\dots,\hat{x}_i,\dots,x_n\right),
    \]
    where the hat denotes omission.
    The cohomology of $C^*\left(G\curvearrowright X;\mathbb{R}\right)$ is denoted by $H^*\left(G\curvearrowright X;\mathbb{R}\right)$.
    
    \begin{rmk}
        \begin{enumerate}
            \item If $G$ acts on itself by (left) multiplication, then there is an isomorphism
            \[
                H^*\left(G\curvearrowright G;\mathbb{R}\right)\cong H^*\left(G;\mathbb{R}\right),
            \]
            which is induced by the map $\theta:C^*\left(G;\mathbb{R}\right)\rightarrow C^*\left(G\curvearrowright G;\mathbb{R}\right)$ given by
            \[
                \left(\theta\psi\right)\left(g_0,\dots,g_n\right)\coloneqq\psi\left(g_0^{-1}g_1,g_1^{-1}g_2,\dots,g_{n-1}^{-1}g_n\right)
            \]
            for $\psi\in C^n\left(G;\mathbb{R}\right)$ (see \S{}\ref{subsec:hb-gp} for our definition of $C^*\left(G;\mathbb{R}\right)$).\label{rk:equiv-cohom-1}
            
            \item Given a choice of basepoint $x$ in a $G$-set $X$, there is a morphism\label{rk:equiv-cohom-2}
            \[
                H^*\left(G\curvearrowright X;\mathbb{R}\right)\rightarrow H^*\left(G\curvearrowright G;\mathbb{R}\right)
            \]
            induced by the map $\varpi_{x}:C^*\left(G\curvearrowright X;\mathbb{R}\right)\rightarrow C^*\left(G\curvearrowright G;\mathbb{R}\right)$ given by
            \[
                \left(\varpi_{x}\psi\right)\left(g_0,\dots,g_n\right)=\psi\left(g_0x,\dots,g_nx\right)
            \]
            for $\psi\in C^n\left(G\curvearrowright X;\mathbb{R}\right)$.
            In fact, this morphism is independent of the choice of $x$.
            Combining this with \ref{rk:equiv-cohom-1}, we obtain a morphism
            \[
                \eta:H^*\left(G\curvearrowright X;\mathbb{R}\right)\rightarrow H^*\left(G;\mathbb{R}\right).
            \]
        \end{enumerate}
        \label{rk:equiv-cohom}
    \end{rmk}
    
    Similar to the definition of bounded cohomology for spaces and groups (see \S{}\ref{subsec:hb-top} and \S{}\ref{subsec:hb-gp}), there is a bounded version of equivariant cohomology: the complex $C^*_b\left(G\curvearrowright X;\mathbb{R}\right)$ of bounded homogeneous $G$-cochains is the subcomplex of $C^*\left(G\curvearrowright X;\mathbb{R}\right)$ consisting of \emph{bounded} $G$-equivariant maps $X^{n+1}\rightarrow\mathbb{R}$.
    The corresponding cohomology is denoted by $H^*_b\left(G\curvearrowright X;\mathbb{R}\right)$, and there is a morphism
    \[
        H^*_b\left(G\curvearrowright X;\mathbb{R}\right)\rightarrow H^*_b\left(G;\mathbb{R}\right)
    \]
    as in Remark \ref{rk:equiv-cohom}\ref{rk:equiv-cohom-2}.
    We refer the reader to \cite{bfh}*{\S{}3.1} for more details on equivariant cohomology.
    
    \subsection{Bounded Euler class of a circle action}\label{subsec:eub}
    
    A choice of hyperbolic structure on a connected surface $S$ defines an action of $\pi_1S$ on the hyperbolic plane $\mathbb{H}^2$.
    This induces an action on the boundary of $\mathbb{H}^2$, which is homeomorphic to the circle $S^1$.
    In general, the dynamics of an action of a group $G$ on the circle is encoded by the bounded Euler class, which is a class in $H^2_b\left(G\right)$ that was introduced by Ghys \cite{ghys:mexico} as a generalisation of Poincaré's rotation number \cites{poincare1,poincare2}.
    
    The bounded Euler class has several equivalent definitions \cite{bfh}, and for our purpose, it will be helpful to define it from the point of view of the orientation cocycle.
    
    Consider the action of the group $\Homeo^+\left(S^1\right)$ of orientation-preserving homeomorphisms of the circle on $S^1$.    
    The \emph{orientation cocycle} is the bounded $2$-cochain $\Or\in C_b^2\left(\Homeo^+\left(S^1\right)\curvearrowright S^1;\mathbb{R}\right)$ given by
    \[
        \Or\left(x,y,z\right)=\begin{cases}+1\quad&\textrm{if the triple $(x,y,z)\in\left(S^1\right)^3$ is positively oriented}\\-1\quad&\textrm{if the triple $(x,y,z)\in\left(S^1\right)^3$ is negatively oriented}\\0\quad&\textrm{if the triple $(x,y,z)\in\left(S^1\right)^3$ is degenerate}\end{cases}.
    \]
    This turns out to be a cocycle, so it defines $\left[\Or\right]\in H^2_b\left(\Homeo^+\left(S^1\right)\curvearrowright S^1;\mathbb{R}\right)$.
    \begin{defi}
        The \emph{universal real bounded Euler class for circle actions} is
        \[
            \eu_b^\mathbb{R}=-\frac{1}{2}\eta\left[\Or\right]\in H^2_b\left(\Homeo^+\left(S^1\right);\mathbb{R}\right),
        \]
        where $\eta:H^2_b\left(\Homeo^+\left(S^1\right)\curvearrowright S^1\right)\rightarrow H^2_b\left(\Homeo^+\left(S^1\right)\right)$ is the morphism described in Remark \ref{rk:equiv-cohom}\ref{rk:equiv-cohom-2}.
    \end{defi}
    
    Given an action $\rho:G\rightarrow\Homeo^+\left(S^1\right)$ of a group on the circle, the (real) \emph{bounded Euler class} of the action is
    \[
        \eu_b^\mathbb{R}(\rho)=\rho^*\eu_b^\mathbb{R}\in H^2_b\left(G;\mathbb{R}\right).
    \]
    This measures how far $\rho$ is from being a rotation action on $S^1$ \cite{frigerio}*{Corollary 10.27}.
    By definition, $\inftynorm{\eu_b^\mathbb{R}(\rho)}\leq\inftynorm{\eu_b^\mathbb{R}}\leq\frac{1}{2}\inftynorm{\Or}=\frac{1}{2}$.
    See \cites{ghys:em,bfh} for more details on the bounded Euler class.
    
    \subsection{Area of a relative \texorpdfstring{$2$}{2}-class}
    
    In \cite{cal-fnb}, Calegari defines a notion of area for a homologically trivial $\gamma:\coprod S^1\rightarrow S$ in an oriented, connected, hyperbolic surface $S$ with non-empty boundary.
    In his definition, it is crucial that $S$ has non-empty boundary because then $H_2(S)=0$, so the map $\partial:H_2\left(S,\gamma\right)\rightarrow H_1\left(\coprod S^1\right)$ is injective and there is a unique class $\alpha\in\partial^{-1}\left(\left[\coprod S^1\right]\right)$.
    We now explain how to generalise Calegari's notion of area to the case where $S$ is closed by defining the area of a class in $H_2(S,\gamma;\mathbb{R})$.
    
    Let $S$ be a hyperbolic surface with (possibly empty) geodesic boundary.
    Let $\gamma:\coprod S^1\rightarrow S$ be a collection of geodesic loops in $S$, and let $\alpha\in H_2\left(S,\gamma;\mathbb{R}\right)$.
    By definition, $H_2(S,\gamma;\mathbb{R})=H_2\left(S_\gamma,\coprod S^1;\mathbb{R}\right)$.
    The mapping cylinder $S_\gamma$ has no geometric structure allowing us to measure areas, but there is a map of pairs $\left(S_\gamma,\coprod S^1\right)\rightarrow\left(S,\gamma\left(\coprod S^1\right)\right)$ defined by collapsing the cylinder.
    This induces a morphism
    \[
        \textstyle H_2\left(S,\gamma;\mathbb{R}\right)\rightarrow H_2\left(S,\gamma\left(\coprod S^1\right);\mathbb{R}\right),
    \]
    and we'll measure the area of $\alpha$ in the image.
    We pick a cell structure on $S$ such that
    \begin{itemize}
        \item The $0$-skeleton of $S$ contains all multiple points of $\gamma$ (i.e. all points $p\in S$ for which there are $s\neq t\in\coprod S^1$ such that $p=\gamma(s)=\gamma(t)$),
        \item The $1$-skeleton of $S$ contains $\gamma\left(\coprod S^1\right)$, and
        \item Each $2$-cell is positively oriented (for the orientation inherited by $S$).
    \end{itemize}
    One can choose a cellular relative $2$-cycle $c$ representing the image of $\alpha$ in the homology $H_2\left(S,\gamma\left(\coprod S^1\right);\mathbb{R}\right)$, and $c$ is in fact unique because both $C_3^\textnormal{cell}\left(S\right)$ and $C_2^\textnormal{cell}\left(\gamma\left(\coprod S^1\right)\right)$ are zero.
    
    \begin{defi}
        Let $\gamma:\coprod S^1\rightarrow S$ be a collection of geodesic loops in a hyperbolic surface $S$.
        Given $\alpha\in H_2\left(S,\gamma;\mathbb{R}\right)$, the \emph{area} of $\alpha$ is defined by
        \[
            \area(\alpha)=\sum_\sigma\lambda_\sigma\area\left(\sigma\right),
        \]
        where $\sum_\sigma\lambda_\sigma\sigma\in Z_2^\textnormal{cell}\left(S,\gamma\left(\coprod S^1\right);\mathbb{R}\right)$ (with $\lambda_\sigma\in\mathbb{R}$ for each $2$-cell $\sigma$) is the unique cellular relative $2$-cycle representing the image of $\alpha$ in $H_2\left(S,\gamma\left(\coprod S^1\right);\mathbb{R}\right)$.
    \end{defi}
    
    \begin{rmk}\label{rk:isom-adm-surf-area}
        Let $f:\left(\Sigma,\partial\Sigma\right)\rightarrow\left(S,\gamma\right)$ be an admissible surface.
        Assume that $\Sigma$ is equipped with a hyperbolic structure with respect to which the map $f:\Sigma\rightarrow S$ is a local isometric embedding.
        Then there is an equality
        \[
            \area\left(f_*[\Sigma]\right)=\area(\Sigma),
        \]
        where $f_*[\Sigma]$ is seen as a class in $H_2\left(S,\gamma;\mathbb{R}\right)$.
    \end{rmk}

    \subsection{Pleated surfaces}
    
    In order to obtain good estimates on the Gromov seminorm for a hyperbolic surface $S$, it will be helpful to measure it with special admissible surfaces, called pleated surfaces.
    The heuristics behind pleated surfaces is the following: if $\Sigma$ is an orientable compact connected surface, then its simplicial volume is given by $\gromnorm{[\Sigma]}=-2\chi^-(\Sigma)$; however, there is no triangulation of $\Sigma$ realising this equality.
    Instead, the simplicial volume is realised by an \emph{ideal triangulation}.
    The idea is therefore to endow admissible surfaces $\Sigma$ with ideal triangulations that are compatible with the hyperbolic structure on $S$.
    
    Pleated surfaces, which were introduced by Thurston \cite{thurston}*{\S{}8.8}, will achieve this.
    
    A \emph{geodesic lamination} $\Lambda$ in a hyperbolic surface $\Sigma$ is a closed subset of $\Sigma$ which decomposes as a disjoint union of complete embedded geodesics.
    Each such geodesic is called a \emph{leaf} of $\Lambda$.
    
    \begin{defi}
        Let $M$ be a hyperbolic manifold.
        A \emph{pleated surface} in $M$ is a map $f:\Sigma\rightarrow M$, where $\Sigma$ is a finite-area hyperbolic surface, such that
        \begin{enumerate}
            \item $f$ sends each arc in $\Sigma$ to an arc of the same length in $M$,
            \item There is a geodesic lamination $\Lambda\subseteq\Sigma$ such that $f$ sends each leaf of $\Lambda$ to a geodesic of $M$, and $f$ is totally geodesic (i.e.\ sends every geodesic to a geodesic) on $\Sigma\smallsetminus\Lambda$, and
            \item If $\Sigma$ is non-compact, then $f$ sends each small neighbourhood of each cusp of $\Sigma$ to a small neighbourhood of a cusp of $M$.
        \end{enumerate}
        The geodesic lamination $\Lambda$ is called a \emph{pleating locus} for $f$.
    \end{defi}
    For a more detailed introduction to pleated surfaces in hyperbolic manifolds, we refer the reader to \cites{futer-schleimer,bonahon,notesonnotes}.
    
    We now show, following Calegari \cite{cal-scl}*{\S{}3.1.3}, how to obtain pleated admissible surfaces.
    The fundamental tool for this is Thurston's \emph{spinning construction}:
    \begin{lmm}[Thurston \cite{thurston}*{\S{}8.10}]\label{lmm:spinning}
        Let $P$ be a pair of pants (i.e.\ a compact hyperbolic surface of genus $0$ with three boundary components) and let $M$ be a compact hyperbolic surface or a closed hyperbolic manifold.
        Given a map $f:P\rightarrow M$, either
        \begin{enumerate}
            \item The image of $\pi_1P$ under $f_*$ is a cyclic subgroup of $\pi_1M$, or
            \item The map $f$ can be homotoped to a pleated surface.
        \end{enumerate}
    \end{lmm}
    \begin{proof}
        Consider a lift $\tilde{f}:\tilde{P}\rightarrow\tilde{M}$ of $f$ to universal covers.
        Note that $\tilde{M}$ is a convex subset of the hyperbolic $n$-space $\mathbb{H}^n$, and $\tilde{P}$ is a convex subset of $\mathbb{H}^2$.
        Pick a geodesic triangle $\Delta$ in $P$ with one vertex on each boundary component.
        This lifts to a geodesic triangle $\tilde{\Delta}$ in a fundamental domain of $\tilde{P}\subseteq\mathbb{H}^2$.
        
        Now the spinning construction consists in dragging vertices of $\tilde{\Delta}$ along the lifts of $\partial P$ to $\mathbb{H}^2$, and moving them to the boundary $\partial\mathbb{H}^2$.
        See Figure \ref{fig:spin}.
        \begin{figure}[htb]
            \centering
            \def\scale{.7}
            \begin{subfigure}[c]{14em}
                \centering
                \begin{tikzpicture}[every node/.style={draw,circle,fill,inner sep=1pt},scale=\scale]
                    \def\ra{3}
                    \def\th{60}
                    
                    \def\ang{-75}
                    \def\rb{6}
                    \def\di{sqrt(\ra*\ra+\rb*\rb)}
                    \def\alpha{atan(\ra/\rb)}
                    \draw [thick,red, domain=94.25:114.5, variable=\s]
                        plot ({\di*cos(\ang)+\rb*cos(\s)},{\di*sin(\ang)+\rb*sin(\s)});

                    \def\ang{45}
                    \def\rb{9.4}
                    \def\di{sqrt(\ra*\ra+\rb*\rb)}
                    \def\alpha{atan(\ra/\rb)}
                    \draw [thick,red, domain=218.25:232.5, variable=\s]
                        plot ({\di*cos(\ang)+\rb*cos(\s)},{\di*sin(\ang)+\rb*sin(\s)});
                    
                    \def\ang{172.1}
                    \def\rb{8.5}
                    \def\di{sqrt(\ra*\ra+\rb*\rb)}
                    \def\alpha{atan(\ra/\rb)}
                    \draw [thick,red, domain=-.5:-15.35, variable=\s]
                        plot ({\di*cos(\ang)+\rb*cos(\s)},{\di*sin(\ang)+\rb*sin(\s)});
                        
                    \def\rb{(\ra*sqrt(((1-cos(\th))^2+sin(\th)^2)/(2-(1-cos(\th))^2-sin(\th)^2)))}
                    \def\di{sqrt(\ra*\ra+\rb*\rb)}
                    \def\alpha{atan(\ra/\rb)}
                
                    \node[circle,draw=gray!60,ultra thick,fill=none,minimum size=2*\ra*\scale cm] () at (0,0) {};
                    
                    \draw node[rectangle,draw=none,fill=none] at (2.75,-2.75) {$\mathbb{H}^2$};
                    \draw node[red,rectangle,draw=none,fill=none] at (0,0) {$\tilde{\Delta}$};
                    
                    \foreach \i in {1,3,5}{
                        \draw [black, domain=0:1, variable=\t]
                            plot ({\di*cos(\th*\i)+\rb*cos(180+(2*\t-1)*\alpha+\th*\i)}, {\di*sin(\th*\i)+\rb*sin(180+(2*\t-1)*\alpha+\th*\i)});
                    }
                    
                    \foreach \i in {0,2,4}{
                        \coordinate (a\i0) at ({\di*cos(\th*\i)+\rb*cos(180+(-1)*\alpha+\th*\i)},{\di*sin(\th*\i)+\rb*sin(180+(-1)*\alpha+\th*\i)});
                        \fill [gray!60, domain=0:1, variable=\t]
                            plot ({\di*cos(\th*\i)+\rb*cos(180+(2*\t-1)*\alpha+\th*\i)}, {\di*sin(\th*\i)+\rb*sin(180+(2*\t-1)*\alpha+\th*\i)})
                            -- plot ({\ra*cos((2*\t-1)*(90-\alpha)+\th*\i)}, {\ra*sin((2*\t-1)*(90-\alpha)+\th*\i)})
                            -- cycle;
                    }
                    
                    \def\ang{-75}
                    \def\rb{6}
                    \def\di{sqrt(\ra*\ra+\rb*\rb)}
                    \def\alpha{atan(\ra/\rb)}
                    \draw node [red,label=right:{\textcolor{red}{$\tilde{v}_1$}}] (p1) at ({\di*cos(\ang)+\rb*cos(94.25)},{\di*sin(\ang)+\rb*sin(94.25)}) {};
                    \coordinate [xshift=-.25cm] (end1) at (a00);
                    \draw [thick,dashed,red,->] (p1)+(-.25,.25) to [bend left=20] (end1);
                    \def\ang{45}
                    \def\rb{9.4}
                    \def\di{sqrt(\ra*\ra+\rb*\rb)}
                    \def\alpha{atan(\ra/\rb)}
                    \draw node [red,label=above:{\textcolor{red}{$\tilde{v}_2$}}] (p2) at ({\di*cos(\ang)+\rb*cos(218)},{\di*sin(\ang)+\rb*sin(218)}) {};
                    \coordinate [xshift=.1cm,yshift=-.15cm] (end2) at (a20);
                    \draw [thick,dashed,red,->] (p2)+(-.15,-.25) to [bend left=20] (end2);
                    \def\ang{172.1}
                    \def\rb{8.5}
                    \def\di{sqrt(\ra*\ra+\rb*\rb)}
                    \def\alpha{atan(\ra/\rb)}
                    \draw node [red,label=below left:{\textcolor{red}{$\tilde{v}_3$}}] (p3) at ({\di*cos(\ang)+\rb*cos(-15.35)},{\di*sin(\ang)+\rb*sin(-15.35)}) {};
                    \coordinate [xshift=.1cm,yshift=.1cm] (end3) at (a40);
                    \draw [thick,dashed,red,->] (p3)+(.25,0) to [bend left=20] (end3);
                    
                    \foreach \i in {1,2,3}{
                        \draw node[rectangle,draw=none,fill=none] at ({2.25*cos(2*\th*(\i-1)+\th/2)},{2.25*sin(2*\th*(\i-1)+\th/2)}) {$\tilde{\partial}_\i$};
                    }
                \end{tikzpicture}
            \end{subfigure}
            \quad{\Huge$\rightsquigarrow$}\quad
            \begin{subfigure}[c]{14em}
                \centering
                \begin{tikzpicture}[every node/.style={draw,circle,fill,inner sep=1pt},scale=\scale]
                    \def\ra{3}
                    \def\th{60}
                    
                    \def\rb{(\ra*sqrt(((1-cos(\th))^2+sin(\th)^2)/(2-(1-cos(\th))^2-sin(\th)^2)))}
                    \def\di{sqrt(\ra*\ra+\rb*\rb)}
                    \def\alpha{atan(\ra/\rb)}
                
                    \node[circle,draw=gray!60,ultra thick,fill=none,minimum size=2*\ra*\scale cm] () at (0,0) {};
                    
                    \draw node[rectangle,draw=none,fill=none] at (2.75,-2.75) {$\mathbb{H}^2$};
                    
                    \foreach \i in {1,3,5}{
                        \draw [black, domain=0:1, variable=\t]
                            plot ({\di*cos(\th*\i)+\rb*cos(180+(2*\t-1)*\alpha+\th*\i)}, {\di*sin(\th*\i)+\rb*sin(180+(2*\t-1)*\alpha+\th*\i)});
                    }
                    
                    \def\ang{-15}
                    \def\rb{5.2}
                    \def\di{sqrt(\ra*\ra+\rb*\rb)}
                    \def\alpha{atan(\ra/\rb)}
                    \draw [thick,red, domain=0:1, variable=\t]
                        plot ({\di*cos(\ang)+\rb*cos((1-\t)*(180+\ang-\alpha)+\t*(180+\ang+\alpha))},{\di*sin(\ang)+\rb*sin((1-\t)*(180+\ang-\alpha)+\t*(180+\ang+\alpha))});
                    
                    \def\ang{105}
                    \def\rb{5.2}
                    \def\di{sqrt(\ra*\ra+\rb*\rb)}
                    \def\alpha{atan(\ra/\rb)}
                    \draw [thick,red, domain=0:1, variable=\t]
                        plot ({\di*cos(\ang)+\rb*cos((1-\t)*(180+\ang-\alpha)+\t*(180+\ang+\alpha))},{\di*sin(\ang)+\rb*sin((1-\t)*(180+\ang-\alpha)+\t*(180+\ang+\alpha))});
                    
                    \def\ang{225}
                    \def\rb{5.2}
                    \def\di{sqrt(\ra*\ra+\rb*\rb)}
                    \def\alpha{atan(\ra/\rb)}
                    \draw [thick,red, domain=0:1, variable=\t]
                        plot ({\di*cos(\ang)+\rb*cos((1-\t)*(180+\ang-\alpha)+\t*(180+\ang+\alpha))},{\di*sin(\ang)+\rb*sin((1-\t)*(180+\ang-\alpha)+\t*(180+\ang+\alpha))});
                    
                    \def\rb{(\ra*sqrt(((1-cos(\th))^2+sin(\th)^2)/(2-(1-cos(\th))^2-sin(\th)^2)))}
                    \def\di{sqrt(\ra*\ra+\rb*\rb)}
                    \def\alpha{atan(\ra/\rb)}
                    
                    \foreach \i [evaluate={\theta=\th*\i}] in {0,2,4}{
                        \fill [gray!60, domain=0:1, variable=\t]
                            plot ({\di*cos(\th*\i)+\rb*cos(180+(2*\t-1)*\alpha+\th*\i)}, {\di*sin(\th*\i)+\rb*sin(180+(2*\t-1)*\alpha+\th*\i)})
                            -- plot ({\ra*cos((2*\t-1)*(90-\alpha)+\th*\i)}, {\ra*sin((2*\t-1)*(90-\alpha)+\th*\i)})
                            -- cycle;
                    }
                \end{tikzpicture}
            \end{subfigure}
            \caption{Thurston's spinning construction.}
            \label{fig:spin}
        \end{figure}
        This construction is called \emph{spinning} because, in $P$, the triangle $\Delta$ has been spun around the boundary components of $P$.
        In this way, one obtains a geodesic lamination $\Lambda$ on $P$ with three leaves, whose complement consists of two open ideal triangles.
        
        There are two cases:
        \begin{enumerate}
            \item If $f(\Lambda)$ is degenerate (i.e.\ the images of the three leaves of $\Lambda$ have the same axis in $\tilde{M}$), then $f_*\left(\pi_1P\right)$ generates a cyclic subgroup of $\pi_1M$.
            
            \item Otherwise, construct a map $f':P\rightarrow M$ homotopic to $f$ as follows.
            For each boundary component $\partial_i$ of $P$, we define $f'\left(\partial_i\right)$ to be the unique closed geodesic in the homotopy class of $f\left(\partial_i\right)$.
            Each leaf $\lambda_i$ of $\Lambda$ is mapped under $f$ to a quasi-geodesic in $M$, which can be straightened to a geodesic $\gamma_i$.
            Set $f'\left(\lambda_i\right)=\gamma_i$.
            Finally, each component of $P\smallsetminus\Lambda$ is an open ideal triangle, and since its image is nondegenerate in $M$, there is a unique totally geodesic extension of $f'$ to this triangle.\qedhere
        \end{enumerate}
    \end{proof}
    
    Using Thurston's spinning construction, we can obtain pleated admissible surfaces.
    This is an adaptation of a lemma of Calegari \cite{cal-scl}*{Lemma 3.7}:
    
    \begin{lmm}[Pleated admissible surfaces]\label{lmm:pleated-adm-srf}
        Let $M$ be a compact hyperbolic surface or a closed hyperbolic manifold.
        Let $\gamma:\coprod S^1\rightarrow M$ be a collection of geodesic loops in $M$, no two components of which have the same image in $M$.
        Then for every rational class $\alpha\in H_2\left(M,\gamma;\mathbb{Q}\right)$ and for every $\varepsilon>0$, there is a pleated admissible surface $f:\left(\Sigma,\partial\Sigma\right)\rightarrow\left(M,\gamma\right)$ such that $f_*[\Sigma]=n(\Sigma)\alpha$ for some $n(\Sigma)\in\mathbb{N}_{\geq1}$, and
        \begin{equation}\label{eq:estim-gromnorm-2}
            \gromnorm{\alpha}\leq\frac{-2\chi^-(\Sigma)}{n(\Sigma)}\leq\gromnorm{\alpha}+\varepsilon.
        \end{equation}
    \end{lmm}
    \begin{proof}
        By Lemma \ref{lmm:simple-incompr-adm-srf}, there is a simple, incompressible, admissible surface $f:\left(\Sigma,\partial\Sigma\right)\rightarrow\left(M,\gamma\right)$ satisfying (\ref{eq:estim-gromnorm-2}), with $f_*\left[\Sigma\right]=n(\Sigma)\alpha$ for some $n(\Sigma)\in\mathbb{N}_{\geq1}$.
        Now take a pants decomposition $\left\{P_i\right\}_i$ of $\Sigma$, as in Figure \ref{fig:pants-dec}.
        The idea is to apply the spinning construction (Lemma \ref{lmm:spinning}) to each $P_i$.
        We can perform the construction separately on each connected component of $\Sigma$; to simplify notations, we therefore assume that $\Sigma$ is connected.
        \begin{figure}[htb]
            \centering
            \begin{tikzpicture}[scale=0.75,rotate=-90]
                \pic[
                transform shape,
                name=a,
                tqft,
                cobordism edge/.style={draw},
                incoming boundary components=0,
                outgoing boundary components=2,
                offset=-.5,
                thick,
                ];
                \pic[
                transform shape,
                name=b,
                tqft,
                cobordism edge/.style={draw},
                incoming boundary components=2,
                outgoing boundary components=1,
                incoming upper boundary component 1/.style={draw,thick,RoyalPurple},
                incoming lower boundary component 1/.style={draw,thick,dotted,RoyalPurple},
                offset=.5,
                at=(a-outgoing boundary),
                anchor=incoming boundary,
                thick,
                ];
                \pic[
                transform shape,
                name=c,
                tqft,
                cobordism edge/.style={draw},
                incoming upper boundary component 1/.style={draw,thick},
                incoming lower boundary component 1/.style={draw,thick,dotted},
                incoming boundary components=1,
                outgoing boundary components=2,
                offset=-.5,
                at=(b-outgoing boundary),
                thick,
                ];
                \pic[
                transform shape,
                name=d,
                tqft,
                cobordism edge/.style={draw},
                incoming boundary components=2,
                outgoing boundary components=1,
                incoming upper boundary component 1/.style={draw,thick,red},
                incoming lower boundary component 1/.style={draw,thick,dotted,red},
                incoming upper boundary component 2/.style={draw,thick,red},
                incoming lower boundary component 2/.style={draw,thick,dotted,red},
                offset=.5,
                at=(c-outgoing boundary),
                anchor=incoming boundary,
                thick,
                ];
                \pic[
                transform shape,
                name=e,
                tqft,
                cobordism edge/.style={draw},
                incoming upper boundary component 1/.style={draw,thick},
                incoming lower boundary component 1/.style={draw,thick,dotted},
                incoming boundary components=1,
                outgoing boundary components=2,
                offset=-.5,
                at=(d-outgoing boundary),
                thick,
                ];
                \pic[
                transform shape,
                name=f,
                tqft,
                cobordism edge/.style={draw},
                incoming boundary components=2,
                outgoing boundary components=1,
                incoming upper boundary component 1/.style={draw,thick},
                incoming lower boundary component 1/.style={draw,thick,dotted},
                incoming upper boundary component 2/.style={draw,thick},
                incoming lower boundary component 2/.style={draw,thick,dotted},
                offset=.5,
                at=(e-outgoing boundary),
                anchor=incoming boundary,
                thick,
                ];
                \pic[
                transform shape,
                name=g,
                tqft,
                cobordism edge/.style={draw},
                incoming upper boundary component 1/.style={draw,thick},
                incoming lower boundary component 1/.style={draw,thick,dotted},
                outgoing upper boundary component 2/.style={draw,thick},
                outgoing lower boundary component 2/.style={draw,thick},
                incoming boundary components=1,
                outgoing boundary components=2,
                offset=-.5,
                at=(f-outgoing boundary),
                thick,
                ];
                \pic[
                transform shape,
                name=h,
                tqft,
                cobordism edge/.style={draw},
                incoming boundary components=1,
                outgoing boundary components=2,
                incoming upper boundary component 1/.style={draw,thick},
                incoming lower boundary component 1/.style={draw,thick,dotted},
                outgoing upper boundary component 1/.style={draw,thick},
                outgoing lower boundary component 1/.style={draw,thick},
                outgoing upper boundary component 2/.style={draw,thick},
                outgoing lower boundary component 2/.style={draw,thick},
                offset=-.5,
                at=(g-outgoing boundary 1),
                anchor=incoming boundary,
                thick,
                ];
                
                \draw node[draw=none] at (-1,-4) {$\Sigma$};
                \foreach \i [evaluate={\x=(-1-2*(7-\i))}] in {1,...,6}{\draw node[draw=none] at (0,\x) {$P_{\i}$};}
                \draw node[draw=none] at (-1,-15) {$P_0$};
                \draw node[red,draw=none] at (1.7,-6) {{\small$\partial_-$}};
                \draw node[red,draw=none] at (-1.7,-6) {{\small$\partial_+$}};
                \draw node[RoyalPurple,draw=none] at (-1.7,-2) {{\small$\partial_1$}};
                
                \draw [ForestGreen,thick] (0,-6) ellipse[x radius=2.75em,y radius=2em];
                \draw node[ForestGreen,draw=none] at (.9,-6.65) {{\small$\alpha$}};
                \draw [Peach,thick] (0,-2) ellipse[x radius=2.75em,y radius=2em];
                \draw node[Peach,draw=none] at (.9,-2.65) {{\small$\beta$}};
            \end{tikzpicture}
            \caption{Pants decomposition of $\Sigma$.}
            \label{fig:pants-dec}
        \end{figure}
        There are three types of components in the pants decomposition:
        \begin{enumerate}
            \item Pairs of pants that are part of a twice-punctured torus (e.g. $P_2,\dots,P_5$ in Figure \ref{fig:pants-dec}),\label{pants-dec-type-2}
            \item Pairs of pants that are glued to themselves to form a once-punctured torus (e.g. $P_6$ in Figure \ref{fig:pants-dec}),\label{pants-dec-type-1}
            \item Pairs of pants that are not of type \ref{pants-dec-type-2} or \ref{pants-dec-type-1} (e.g. $P_0,P_1$ in Figure \ref{fig:pants-dec}).\label{pants-dec-type-3}
        \end{enumerate}
        
        Fix a pants component $P_i$.
        We want to apply Lemma \ref{lmm:spinning} to the restriction $f_{\left|P_i\right.}:P_i\rightarrow M$; we need to ensure that $f_*\left(\pi_1P_i\right)$ is non-cyclic.
        We distinguish three cases, based on the type of $P_i$.
        
        We first show that, if $P_i$ of type \ref{pants-dec-type-3}, then $f_*\left(\pi_1P_i\right)$ cannot be cyclic.
        Recall from \S{}\ref{subsec:1-chns} that $\gamma$ represents $[c]=\left[\sum_jw_j\right]\in C_1^{\hmtp}\left(\pi_1M;\mathbb{Z}\right)$, where no two of the $w_j$'s are conjugate or generate a cyclic subgroup of $\pi_1M$ (by assumption on $\gamma$).
        Since $f$ is an admissible surface, each boundary component of $\Sigma$ maps to a power of some $w_j$, and simplicity implies that no two boundary components of $\Sigma$ map to powers of the same $w_j$.
        We can assume that the components $\left\{P_i\right\}_i$ of type \ref{pants-dec-type-3} are ordered as $\left\{P_0,\dots,P_k\right\}$, in such a way that $P_0$ has two boundary components on $\partial\Sigma$, and each $P_i$ is glued to $P_{i-1}$ along one boundary component and has one boundary component on $\partial\Sigma$ (this is consistent with the notations of Figure \ref{fig:pants-dec}, where $k=1$).
        With these notations, we can order the $w_j$'s in such a way that $f_*\left(\pi_1P_0\right)=\left<w_0,w_1\right>$, and each $P_i$ has one boundary component glued to $P_{i-1}$ and whose image represents an element of $\left<w_0,\dots,w_i\right>$, and one boundary component lying on $\partial\Sigma$, and whose image represents a power of $w_{i+1}$.
        In particular, it follows that $f_*\left(\pi_1P_i\right)$ is not cyclic for any $P_i$ of type \ref{pants-dec-type-3}.
        
        Now assume that $P_i$ is of type \ref{pants-dec-type-2}.
        Two of the boundary components $\partial_+$ and $\partial_-$ of $P_i$ are meridians in a twice-punctured torus ($\partial_\pm$ are depicted in Figure \ref{fig:pants-dec} for $P_i=P_4$).
        Let $\alpha$ be the equator of this twice-punctured torus and let $\delta_\alpha:\Sigma\rightarrow\Sigma$ denote the Dehn twist along $\alpha$.
        If $f_*\left(\pi_1P_i\right)=\left<f_*\left(\partial_+\right),f_*\left(\partial_-\right)\right>$ is cyclic, then replace $\partial_\pm$ with $\delta_\alpha\partial_\pm$; this amounts to defining a new pants decomposition of $\Sigma$.
        For this pants decomposition, $\left<f_*\left(\partial_+\right),f_*\left(\partial_-\right)\right>$ is not cyclic because $f_*\alpha$ and $f_*\partial_\pm$ do not commute by incompressibility (otherwise $\left[\alpha,\partial_\pm\right]$ would be represented by a simple closed curve in $\Sigma$ with nullhomotopic image in $M$).
        It might be that, after this modification, the adjacent pair of pants $P_{j}$ in the same twice-punctured torus as $P_i$ has cyclic image in $\pi_1M$.
        In this case, one applies the Dehn twist $\delta_\alpha$ a second time.
                
        Assume finally that $P_i$ is of type \ref{pants-dec-type-1}.
        Then $P_i$ is glued to itself to form a once-punctured torus.
        Denote by $\partial_1$ one of the two boundary components of $P_i$ that is glued to form a meridian in the once-punctured torus, and by $\beta$ the equator ($\partial_1$ and $\beta$ are depicted in Figure \ref{fig:pants-dec} for $P_i=P_6$).
        Then $f_*\left(\pi_1P_i\right)=\left<f_*\left(\partial_1\right),f_*\left(\partial_1\right)^{f_*\left(\beta\right)}\right>$, where the exponent denotes conjugation.
        If $f_*\left(\pi_1P_i\right)$ is cyclic, then there are $w\in\pi_1M$ and $k,\ell\in\mathbb{Z}$ such that
        \begin{equation*}
            f_*\left(\partial_1\right)=w^k=f_*(\beta)w^\ell f_*(\beta)^{-1}.
        \end{equation*}
        But $\pi_1M$ is Gromov-hyperbolic, and therefore it is known to be a \emph{CSA group}, in the sense that all its maximal abelian subgroups are malnormal --- see \cite{delaharpe-weber}*{Example 10}.
        Hence $\left<w\right>$ is malnormal (after possibly replacing $w$ with a generator of a maximal abelian subgroup containing it), and $f_*\left(\partial_1\right)\in\left<w\right>\cap\left<w\right>^{f_*(\beta)}\smallsetminus\{1\}$, so $f_*(\beta)\in\left<w\right>$.
        In particular, $f_*\left[\partial_1,\beta\right]=1$, which contradicts incompressibility.
        This proves that $f_*\left(\pi_1P_i\right)$ cannot be cyclic.
        
        Therefore, after performing the above modifications, we have a pants decomposition of $\Sigma$ for which $f_*\left(\pi_1P_i\right)$ is never a cyclic subgroup of $\pi_1M$.
        By Lemma \ref{lmm:spinning}, the restriction of $f$ to each $P_i$ can be homotoped to a pleated map.
        Moreover, these homotopies can be performed simultaneously as the image of each boundary component of a pair of pants is homotoped to the unique geodesic in its homotopy class.
        Hence, we obtain a pleated map homotopic to $f$, which is still an admissible surface and satisfies (\ref{eq:estim-gromnorm-2}).
    \end{proof}
    
    \begin{rmk}
        In fact, we will not need the estimate (\ref{eq:estim-gromnorm-2}) on the Gromov seminorm in Lemma \ref{lmm:pleated-adm-srf}: it will be enough for us to know that every rational class is represented by a pleated admissible surface.
    \end{rmk}

    \subsection{Bounded Euler class and area}
    
    A hyperbolic structure on a surface $S$ induces an action of $\pi_1S$ on the boundary of the hyperbolic plane, which is a circle.
    Hence, we get a circle action $\rho:\pi_1S\rightarrow\Homeo^+\left(S^1\right)$, defining a bounded Euler class $\eu_b^\mathbb{R}(\rho)\in H^2_b\left(\pi_1S;\mathbb{R}\right)$ as explained in \S{}\ref{subsec:eub}.
    We will call it the \emph{bounded Euler class of $S$} and denote it by $\eu_b^\mathbb{R}(S)$.
    It can also be seen as an element of $H^2_b(S;\mathbb{R})$ --- see \S{}\ref{subsec:hb-gp}.
    
    The following is implicit in Calegari's book \cite{cal-scl}*{Lemma 4.68}:
    
    \begin{lmm}[Bounded Euler class and area]\label{lmm:area-eub}
        Let $\gamma:\coprod S^1\rightarrow S$ be a collection of geodesic loops in a compact hyperbolic surface $S$.
        Let $\alpha\in H_2\left(S,\gamma;\mathbb{Q}\right)$ be a rational class.
        Then
        \[
            \area(\alpha)=-2\pi\left<\eu_b^\mathbb{R}(S),\alpha\right>.
        \]
    \end{lmm}
    \begin{proof}
        Lemma \ref{lmm:pleated-adm-srf} yields a pleated admissible surface $f:(\Sigma,\partial\Sigma)\rightarrow(S,\gamma)$ with $f_*[\Sigma]=n(\Sigma)\alpha$ for some $n(\Sigma)\in\mathbb{N}_{\geq1}$.
        Hence,
        \[
            \left<\eu_b^\mathbb{R}(S),\alpha\right>=\frac{1}{n(\Sigma)}\left<\eu_b^\mathbb{R}(S),f_*[\Sigma]\right>.
        \]        
        Recall from \S{}\ref{subsec:eub} that $\eu_b^\mathbb{R}$ is defined as the image of $-\frac{1}{2}\left[\Or\right]$ in $H^2_b\left(\Homeo^+\left(S^1\right)\right)$.
        The pleated structure on $\Sigma$ defines an ideal triangulation, and the Kronecker product $\left<\eu_b^\mathbb{R}(S),f_*[\Sigma]\right>$ is therefore given by
        \[
            \left<\eu_b^\mathbb{R}(S),f_*[\Sigma]\right>=-\frac{1}{2}\sum_{\sigma}\Or\left(f(\sigma)\right),
        \]
        where the sum is over all triangles $\sigma$ in this ideal triangulation, and $\Or\left(f(\sigma)\right)$ is $+1$ if $f(\sigma)$ is positively oriented, $-1$ if $f(\sigma)$ is negatively oriented, and $0$ if $f(\sigma)$ is degenerate.
        But each $f(\sigma)$ is an ideal triangle in $S$, and contributes $\pi\Or\left(f(\sigma)\right)$ to $\area\left(\sum_\sigma f(\sigma)\right)$ by the Gauß--Bonnet Theorem.
        Therefore,
        \begin{align*}
            \area(\alpha)&=\frac{1}{n(\Sigma)}\area\left(f_*[\Sigma]\right)=\frac{1}{n(\Sigma)}\area\left(\sum_\sigma f(\sigma)\right)\\
            &=\frac{\pi}{n(\Sigma)}\sum_\sigma\Or\left(f(\sigma)\right)=-\frac{2\pi}{n(\Sigma)}\left<\eu^b_\mathbb{R}(S),f_*[\Sigma]\right>\\
            &=-2\pi\left<\eu_b^\mathbb{R}(S),\alpha\right>.\qedhere
        \end{align*}
    \end{proof}
    
    A class $\alpha\in H_2\left(S,\gamma;\mathbb{R}\right)$ is said to be \emph{projectively represented by a positive immersion} if there is an admissible surface $f:(\Sigma,\partial\Sigma)\rightarrow(S,\gamma)$ with $f_*[\Sigma]=n(\Sigma)\alpha$ for some $n(\Sigma)\in\mathbb{N}_{\geq1}$, and such that $f$ is an orientation-preserving immersion.
    
    The following is now a straightforward generalisation of a result of Calegari \cite{cal-scl}*{Lemma 4.62}:
    
    \begin{mthm}{E}[Extremality of the bounded Euler class]\label{thm:eub-extr}
        Let $\gamma:\coprod S^1\rightarrow S$ be a collection of geodesic loops in a compact hyperbolic surface $S$.
        Let $\alpha\in H_2(S,\gamma;\mathbb{Q})$ be projectively represented by a positive immersion $f:\left(\Sigma,\partial\Sigma\right)\looparrowright\left(S,\gamma\right)$.
        Then
        \[
            \gromnorm{\alpha}=\frac{-2\chi^-(\Sigma)}{n(\Sigma)}=-2\left<\eu_b^\mathbb{R}(S),\alpha\right>.
        \]
        In other words, $f$ is an extremal surface and $-\eu_b^\mathbb{R}(S)$ is an extremal class for $\alpha$.
        In particular, $\gromnorm{\alpha}\in\mathbb{Q}$.
    \end{mthm}
    \begin{proof}
        Note that $\Sigma$ inherits a hyperbolic structure from $S$ for which $f$ is a local isometric embedding, and $\area(\Sigma)=n(\Sigma)\area(\alpha)$ (see Remark \ref{rk:isom-adm-surf-area}).
        By the Gauß--Bonnet Theorem,
        \[
            -2\pi\chi^-(\Sigma)=-2\pi\chi(\Sigma)=\area(\Sigma)=n(\Sigma)\area(\alpha).
        \]
        Therefore (using the topological interpretation of $\gromnorm{\cdot}$ --- see Proposition \ref{prop:ell1-gromnorm}),
        \[
            \gromnorm{\alpha}\leq\frac{-2\chi^-(\Sigma)}{n(\Sigma)}=\frac{1}{\pi}\area(\alpha)=-2\left<\eu_b^\mathbb{R}(S),\alpha\right>,
        \]
        where the last equality follows from Lemma \ref{lmm:area-eub}.
        We have $\inftynorm{\eu_b^\mathbb{R}(S)}\leq\frac{1}{2}$, so Bavard duality for $\gromnorm{\cdot}$ (Theorem \ref{thm:bav-dual-rel-grom}) gives
        \[
            -2\left<\eu_b^\mathbb{R}(S),\alpha\right>\leq\frac{\left<-\eu_b^\mathbb{R}(S),\alpha\right>}{\inftynorm{\eu_b^\mathbb{R}(S)}}\leq\gromnorm{\alpha}.\qedhere
        \]
    \end{proof}
    
    \begin{rmk}
        In the case where $S$ has non-empty boundary, the converse of Theorem \ref{thm:eub-extr} holds: if $\gromnorm{\alpha}=-2\left<\eu_b^\mathbb{R}(S),\alpha\right>$, then $\alpha$ is projectively represented by a positive immersion \cite{cal-scl}*{Lemma 4.62}.
        However, this uses the existence of extremal surfaces for $\gromnorm{\alpha}$ (see \cite{cal-scl}*{Remark 4.65}), which is not known if $S$ is closed.
    \end{rmk}
    
\bibliography{PhD-Refs.bib}

\end{document}